\documentclass[openany]{amsart}
\usepackage{amssymb, amsfonts, lacromay}
\usepackage{mathrsfs,comment}
\usepackage[usenames,dvipsnames]{color}
\usepackage[normalem]{ulem}
\usepackage{url}
\usepackage{tikz-cd}
\usepackage[all,arc,2cell]{xy}
\UseAllTwocells 
\usepackage{enumerate}             

\usepackage{todonotes}
\usepackage{hyperref}  
\hypersetup{%
  bookmarksnumbered=true,%
  bookmarks=true,%
  colorlinks=black,%
  linkcolor=true,%
  citecolor=true,%
  filecolor=true,
  menucolor=black,%
  pagecolor=black,%
  urlcolor=black,%
  pdfnewwindow=true,%
  pdfstartview=FitBH}

\def\makeautorefname#1#2{\expandafter\def\csname#1autorefname\endcsname{#2}}
\def\equationautorefname~#1\null{(#1)\null}
\makeautorefname{footnote}{footnote}%
\makeautorefname{item}{item}%
\makeautorefname{figure}{Figure}%
\makeautorefname{table}{Table}%
\makeautorefname{part}{Part}%
\makeautorefname{appendix}{Appendix}%
\makeautorefname{chapter}{Chapter}%
\makeautorefname{section}{Section}%
\makeautorefname{subsection}{Section}%
\makeautorefname{subsubsection}{Section}%
\makeautorefname{theorem}{Theorem}%
\makeautorefname{thm}{Theorem}%
\makeautorefname{sta}{Statement}%
\makeautorefname{cor}{Corollary}%
\makeautorefname{lem}{Lemma}%
\makeautorefname{prop}{Proposition}%
\makeautorefname{pro}{Property}
\makeautorefname{conj}{Conjecture}%
\makeautorefname{defn}{Definition}%
\makeautorefname{notn}{Notation}
\makeautorefname{notns}{Notations}
\makeautorefname{rem}{Remark}%
\makeautorefname{rems}{Remarks}%
\makeautorefname{quest}{Question}%
\makeautorefname{exmp}{Example}%
\makeautorefname{ax}{Axiom}%
\makeautorefname{claim}{Claim}%
\makeautorefname{assn}{Assumption}%
\makeautorefname{asses}{Assumptions}%
\makeautorefname{countcom}{Assumption}%
\makeautorefname{countcoms}{Assumptions}%
\makeautorefname{countmcom}{Assumption}%
\makeautorefname{KMZ0}{Construction}%
\makeautorefname{prob}{Problem}%
\makeautorefname{warn}{Warning}%
\makeautorefname{obs}{Observation}%
\makeautorefname{conv}{Convention}%

\newtheorem{thm}{Theorem}[section]

\newtheorem{cor}{Corollary}[section]
\newtheorem{prop}{Proposition}[section]
\newtheorem{lem}{Lemma}[section]

\newtheorem{conj}{Conjecture}[section]

\theoremstyle{definition}
\newtheorem{defn}{Definition}[section]

\newtheorem{KMZ0}{Construction}[section]
\newtheorem{exmp}{Example}[section]
\newtheorem{notn}{Notation}[section]
\newtheorem{notns}{Notations}[section]

\newtheorem{quest}[defn]{Question}
\newtheorem{rem}{Remark}[section]

\newtheorem{warn}{Warning}[section]

\newtheorem{conv}{Convention}[section]

\newcounter{assn}
\renewcommand{\theassn}{\Alph{assn}}
\newenvironment{ass}[1][]{\refstepcounter{assn}\par\medskip\noindent
   \textbf{Assumption~\theassn.} \rmfamily}{\medskip}

 \newcounter{countcom}
\renewcommand{\thecountcom}{\Alph{countcom}$_{com}$}
\newenvironment{asscom}[1][]{\refstepcounter{countcom}\par\medskip\noindent
   \textbf{Assumption~\thecountcom.} \rmfamily}{\medskip}
   
    \newcounter{countmcom}
\renewcommand{\thecountmcom}{\Alph{countmcom}$_{mcom}$}
\newenvironment{assmcom}[1][]{\refstepcounter{countmcom}\par\medskip\noindent
   \textbf{Assumption~\thecountmcom.} \rmfamily}{\medskip}

\makeatletter
\let\c@obs=\c@thm
\let\c@sta=\c@thm
\let\c@cor=\c@thm
\let\c@prop=\c@thm
\let\c@lem=\c@thm
\let\c@prob=\c@thm
\let\c@con=\c@thm
\let\c@constbghj=\c@thm
\let\c@conj=\c@thm
\let\c@defn=\c@thm
\let\c@notn=\c@thm
\let\c@notns=\c@thm
\let\c@exmp=\c@thm
\let\c@ax=\c@thm
\let\c@pro=\c@thm
\let\c@ass=\c@thm
\let\c@warn=\c@thm
\let\c@rem=\c@thm
\let\c@asscom=\c@thm
\let\c@assmcom=\c@thm
\let\c@conv=\c@thm
\let\c@sch=\c@thm
\let\c@equation\c@thm
\numberwithin{equation}{section}
\makeatother

\definecolor{orange}{rgb}{1,0.5,0}

\newcommand{\Dalg}{\bD\big[\PI[\ul{G\sT}]\big]}
\newcommand{\DGalg}{\bD_G\big[\PI_G[\ul{G\sT}]\big]}

\newcommand{\gen}{$\bF_\bullet$}

\newcommand{\bj}{\mathbf{j}}
\newcommand{\bk}{\mathbf{k}}
\newcommand{\bm}{\mathbf{m}}
\newcommand{\bn}{\mathbf{n}}

\renewcommand{\bB}{\overline }

\newcommand{\OGop}{G\mathscr{O}^{op}[\sT]}
\newcommand{\Cp}{\mathbb{C}^{\mathrm{pre}}}

\title{Group completions and the homotopical monadicity theorem}

\author{Hana Jia Kong}
\address{School of Mathematical Sciences, Zhejiang University, Hangzhou, China}
\email{hana.jia.kong@gmail.com}

\author{J. Peter May}
\address{Department of Mathematics, The University of Chicago, Chicago, IL 60637}
\email{may@math.uchicago.edu}

\author{Foling Zou}
\address{Institute of Mathematics, Chinese Academy of Sciences, Beijing, China}
\email{zoufoling@amss.ac.cn}


\begin{document}

\begin{abstract}  We abstract and generalize homotopical monadicity statements, placing in a single conceptual framework a range of old and recent recognition and characterization  principles in iterated loop space theory in classical, equivariant, and multiplicative frameworks. A short logical predecessor \cite{KMZ0} gives a formal homotopical analog of the classical categorical monadicity theorem. That applies to any adjunction relating any categories in which one can reasonably do homotopy theory.  The price of the generality is that one cannot easily do calculations with the results.  We show here how to extend the result to a calculationally meaningful recognition principle by bringing in a new monad nicely related to the categorical monad of the adjunction.  With a suitable notion of group completion and an approximation theorem, tweaking the proof leads to a generalized homotopical monadicity theorem that encompasses all of the known applications and some new ones.

 This paper is divided into three parts.  In the first, we give the general abstract theory, including an attempt at $\infty$-categorical reinterpretation, and treat the classical examples with structured spaces or $G$-spaces as input.  In the second, we develop a general context of composite adjunctions that feeds into the first.   It specializes, quite differently, to give infinite loop space machines that take either orbital presheaves or algebras over categories of operators as input.  A paper growing out of this theory \cite{KMZ2} will focus on new constructions and applications when the starting category is that of orbital presheaves of spaces and the output is $G$-spectra.   In the brief third part, we show how the multiplicative theory fits directly into the frameworks of the first and second parts.   New multiplicative constructions that are work in progress \cite{May26} will give new input that feeds into this multiplicative theory.
\end{abstract}

\maketitle

\tableofcontents


\part*{Introduction}

We give a conceptual reinterpretation of work that started with the introduction of operads in ``The geometry of iterated loop spaces'' \cite{MayGeo}.  The word ``geometry'' was meant to convey that the structures involved really were geometric at heart.   They implicitly reduced the homotopical study of iterated loop spaces to the study of configuration spaces.   The same idea works equivariantly \cite{GM3, MMO} and multiplicatively \cite{MayMult, Rant1}, but the equivariant multiplicative theory here is not in the earlier literature.\footnote{Equivariant multiplicative theory based on the Segal machine is given in \cite{GMMO, GMMO3}, but it is not general enough for the applications we have in mind and shoots through the space level, going as directly as possible from categorical input to spectrum level output. }

\subsection{Introduction to Part 1, the classical theory}

We start with a general conceptual reinterpretation of material in \cite{Rant1}.  That focused on direct comparison with the proof of the categorical monadicity theorem. Here we axiomatize a general homotopical variant of that classical categorical context. The axiomatization simplifies things by focusing attention on exactly what is needed. {It should serve as a guide to future applications.  A motivic application is work in progress \cite{Ajay}.}  We emphasize that the theory here primarily focuses on monads, not necessarily those coming from operads.  

In brief, adjoint pairs of functors $(\SI,\OM)$  are ubiquitous in mathematics, and so are monads.   The right adjoint 
$\OM\colon \sS\rtarr \sT$ always carries extra structure that is encoded by its taking values in the category $\GA[\sT]$ of algebras over the adjunction monad 
$\GA = \OM\SI$.  {As recalled in \cite{KMZ0}, the categorical monadicity theorem explains when the resulting functor $\OM_{\GA}\colon  \sS \rtarr \GA[\sT]$ to $\GA$-algebras is an equivalence of categories.   A homotopical generalization there explains quite generally when the same context gives an equivalence of homotopy categories.  That can be viewed as a special case of the theory here, albeit one with much less precise structure.}  In homotopical contexts, $\OM$ very often takes values in the category  $\bC[\sT]$ of $\bC$-algebras for some other monad $\bC$.  Moreover, the more precise axioms for this structure 
are generally not satisfied by the adjunction monad $\GA$, and that precision encodes both formal and  calculational information that is invisible to  $\GA$.

Writing $\OM_{\bC}\colon  \sS\rtarr  \bC[\sT]$  for the resulting more structured version of $\OM$, we find that $\OM_{\bC}$ has a  left adjoint  $\SI_{\bC}\colon \bC[\sT] \rtarr  \sS$ which we think of as a {\em monadically coequalized} variant of $\SI$. When applied to a free $\bC$-algebra $\bC X$, the coequalizer simplifies to  give an isomorphism  $\SI_{\bC}\bC  X \iso  \SI X$,  and the unit of the adjunction restricts on free $\bC$-algebras to  give a map of monads from $\bC$ to the adjunction monad $\GA$; see \autoref{alphamon} and \autoref{eye}. Approximating general $\bC$-algebras $Y$ by well-behaved approximations $\overline{Y}$, one finds a recognition principle of the following general form.

\begin{thm}\label{recprin}  There is a functor $\mathrm{Bar} \colon \bC[\sT] \rtarr \bC[\sT]$, written $Y\mapsto \overline{Y}$, and a natural equivalence  $\ze\colon \overline{Y} \rtarr Y$ such that the unit $\et_{\bC}\colon \overline{Y} \rtarr \OM_{\bC} \SI_{\bC} \overline{Y}$  is a group completion and is therefore an equivalence if $Y$ is grouplike.\end{thm}

In turn, this leads to  a characterization principle of the following general form, {which can be viewed as a variant homotopical monadicity theorem.}

\begin{thm}\label{adjequiv}  $(\SI_{\bC},\OM_{\bC})$ induces an adjoint equivalence from the homotopy category of grouplike $ \bC$-algebras in $\sT$ to that of $\OM$-connective objects in $\sS$.
\end{thm}

The terms {\em group completion}, {\em grouplike}, and {\em $\OM$-connective} have axiomatized meanings and can specialize to have different meanings in different contexts. {The notion of $\OM$-connectivity has nothing to do with the monad $\bC$ and is discussed in \cite[Section 2.3]{KMZ0}; it specializes to standard definitions in most contexts.  In contrast, no notion of group completion is relevant to the homotopical monadicity theorem of \cite{KMZ0}.  We shall have two  very different ideas of grouplike objects and group completions. The first is characterized axiomatically (Definitions \ref{screwy} and \ref{screwy2}),  but the examples are concrete and homological or  homotopical (for example Definitions \ref{Hopf} and \ref{Hopf2}); that idea is the one that leads both to our proofs and to myriads of explicit calculations.  The other is formal and conceptual (\autoref{fgplike}) and is characterized by a universal property on passage to homotopy categories (\autoref{Univ}), but it is not geared towards calculations.  We shall find an equivalence between them under appropriate assumptions (\autoref{new1}).  }

We emphasize that despite its obvious specialization in topology, $(\SI,\OM)$ can be any suitable adjoint pair in our general axiomatic context. Here $\overline{Y}$ might come as a cofibrant approximation of $Y$, but we focus on the highly structured choice given by the monadic bar construction $\overline{Y} = B(\bC,\bC,Y)$; it is often itself an example of a cofibrant approximation (see \autoref{MODEL}), but that is not our emphasis. {As discussed for example in \cite[Section 3.1]{MMO}, it is a kind of formal resolution that appears quite generally in ``derived homotopical algebra".}

For convenience in discussing applications and with potential new applications in mind, we begin with standing ``Assumptions" that together will lead to results of the form given in Theorems \ref{recprin} and \ref{adjequiv}.  

In \autoref{CONTEXT1} we establish our categorical context, explaining a general adjunction $(\SI_{\bC},\OM_{\bC})$ that will have many specializations.  While $\SI_{\bC}$ was defined in passing in previous papers, its centrality to our theory is a new perspective. 

In \autoref{CONTEXT2}, we explain our homotopical assumptions, giving remarks as we go about how they are verified in special cases.
The most substantial is \autoref{ass5}, the approximation theorem, which requires an always present natural map of monads $\BC X\rtarr \GA X$ to give a group completion for all $X$.  

{
In \autoref{BARCON}, we explain how these assumptions fit together to prove various versions of Theorems \ref{recprin} and \ref{adjequiv}.  \autoref{recprin0} gives the original concrete version of  \autoref{recprin}, before introduction of $\SI_{\bC}$. \autoref{recprin1} gives a slightly more detailed version, and \autoref{BPQthm} shows that a very general version of the Barratt-Priddy-Quillen theorem is an immediate consequence.  \autoref{equiv} elaborates \autoref{adjequiv}.   Theorems \ref{recprinnew} and \ref{adjequivnew} restate Theorems \ref{recprin} and \ref{adjequiv} in terms of the more conceptual abstract version of group completion.

In \autoref{CONTEXT3}, we say a bit about model categories and sketch how our proof of \autoref{adjequiv}  should imply an equivalence of associated  $\infty$-categories.  One piece of  further work  needed is the proof of a conjectured enhancement (\autoref{DKThm?}) of work of Dwyer and Kan \cite{DK2}.  The only other further work needed is a proof that $\SI_{\bC}\OM_{\bC}Z$ is $\OM$-connective when $Z$ is so (see \autoref{When}).  In the rest of the paper we are being old-fashioned and taking little interest in model categories or $\infty$-categories.  We care more about when precise point-set level categories and concepts are equivalent.  Many decades of concrete applications and calculations show that this perspective has some value.  One should be eclectic.
}

The original topological examples are developed in \autoref{SPACEG}.  They use monads associated to operads of topological spaces, with $\sT$ being the category of based spaces and $\sS$ being either a category of based spaces or a category of spectra, and similarly equivariantly. 

Nonequivariantly, we specialize to $n$-fold loop spaces in \autoref{SPACES}, summarizing our specializations of  Theorems \ref{recprin} and \ref{adjequiv} in  \autoref{space}.  We find that every $(n-1)$-connected or, synonymously, ``$n$-connective'', space is equivalent to the coequalized $n$-fold suspension of an $E_n$-space.  The emphasis is on a new perspective on the group completion of $E_n$ spaces to $n$-fold loop spaces.  When $n\geq 2$, two different classical notions (Definitions \ref{Hopf}, \ref{Hopf2}) of group completion are shown to be equivalent to each other and to the formal conceptual notion of \autoref{fgplike}.  

In the case $n=1$ only one of the classical notions, that of \autoref{Hopf2}, applies.  In this case, the approximation theorem is given a conceptual proof in \autoref{Moore}, which is largely a new perspective on material in a paper \cite{Fied} of Fiedorowicz.  Here the focus shifts to a comparison of two monads in $\sT$ with two analogous monads in the category $\sV$ that is the natural home of Moore loop spaces.  A first example, \autoref{adj0}, of the composite adjunction context of Part 2 plays a central role.  

Equivariantly, Sections \ref{GSPACES} and \ref{MooreG} give the analogous examples of $V$-fold loop $G$-spaces for a representation $V$ of $G$ and of $1$-fold loop $G$-spaces. Here the two classical notions of group completion are Definitions \ref{HopfG} and \ref{HopfG2} and our specializations of Theorems \ref{recprin} and \ref{adjequiv} are given in \autoref{Gspace}.  

While $G$ is finite in \autoref{GSPACES}, it can be any topological group in \autoref{MooreG}.  As far as we know, the material of \autoref{MooreG} is new, but we have written \autoref{Moore} in such a way that its equivariant generalization is immediate. The main results here, Theorems \ref{approxGn=1} and \ref{keyGFied}, are variants of the approximation theorem that follow by passage to fixed points from their nonequivariant versions, Theorems \ref{approxn=1} and \ref{keyFied}.

Passing to $n=\infty$, we give foundations in \autoref{SPECTRA} and rework the equivariant operadic infinite loop space machine in  \autoref{SpaceSpectra}. This shows how to construct $G$-spectra from $E_\infty$ $G$-spaces when $G$ is finite. To avoid duplicative exposition, the nonequivariant case is viewed as the case $G=\{e\}$ of the equivariant case.  The homological definition of group completion of $E_{\infty}$ $G$-spaces is the same as for $E_V$ $G$-spaces when $\bR^2\subset V$ (\autoref{HopfG}).  It is again equivalent to the conceptual definition of \autoref{fgplike}.  Every $E_{\infty}$-$G$-space group completes to the $0$th $G$-space of a connective $G$-spectrum, and every connective $G$-spectrum is equivalent to the coequalized suspension $G$-spectrum of a grouplike $E_{\infty}$ $G$-space, \autoref{Gspace}.  

Here $G$-spectra mean genuine $G$-spectra, incorporating representations.  There is also a very partial generalization from finite to compact Lie groups (\autoref{finiteindex}),  There is a parallel theory for classical (or naive) $G$-spectra, which ignores representations.  That works for any topological group $G$ (\autoref{classical}).

We should say something about our spectral foundations of \autoref{SPECTRA}. There are many good concrete point-set level categories of spectra and of $G$-spectra.  Only one is known to fit into the achingly elementary formal context of \autoref{ass1} that is our starting point.  No category of spectra can both fit into that context and be symmetric monoidal with the sphere spectrum as unit, but all good categories give equivalent homotopy categories, and in fact equivalent $\infty$-categories. 

\subsection{Introduction to Part 2, composite adjunctions}

In Part 2, we build a general context that will specialize to apply to infinite loop space machines that are defined in two quite different equivariant contexts.   The first concerns algebras over monads in categories related to categories of operators {(the forerunner of $\infty$-operads)}, such algebras being suitable covariant functors.  The second concerns algebras over monads in categories of orbital presheaves,  such algebras being suitable contravariant functors.   At first sight, this seems farfetched since we shall see many quite significant differences between these two contexts.  Nevertheless, both are special cases of a single general context, which itself is a special case of our original context.  We explain  the general context in \autoref{ABCDEF} and its specializations to categories of operators  in \autoref{catop} and to orbital presheaves in \autoref{orbpre}.  Paradoxically, the agreement of context allows us both to highlight precisely how the differences play out and to discover surprising parallels and connections.  

The categories of operators specialization is largely classical, starting from \cite{MT}.   Aside from the conceptual context, the main new feature is the introduction of a new kind of operad, called a $G$-operad.  Instead of having $G$-spaces $\sC(n)$ for non-negative integers $n$, it has $G$-spaces $\sC(\bn^{\al})$ for finite $G$-sets $\bn^{\al}$ determined by homomorphisms $\al\colon G\rtarr \SI_n$.  These operads and their associated monads fill a longstanding conceptual gap in the relationship between equivariant operads and monads on the one hand and the two kinds of equivariant categories of operators that appear prominently in \cite{MMO, GMMO3} and elsewhere. 

The operadic presheaf context can be viewed as a conceptual reinterpretation of an early paper of Costenoble and Waner \cite{CW}.   The general theory is described here, but its application requires non-obvious examples of monads in the category of orbital presheaves.   They are constructed from $G$-operads, but not in an obvious way.  The construction is worked out in \cite{KMZ2}, where examples may be found.  

{
We warn the reader that combinatorial details of equivariance will be slighted in this paper.  We say the bare minimum that is needed to show how the category of operators context dictates new $G$-operads and their monads that fit into our general framework.  Conventions and definitions are expanded in \cite{KMZ2}, and proofs omitted here are supplied there, where they are needed.   Our focus here is on conceptualization.
}

\subsection{Introduction to Part 3, the multiplicative theory}

{
We apply the context of \autoref{ass1} multiplicatively in \autoref{Mult1}.  This entails working with monad pairs as introduced by Beck \cite{Beck}.  We recall their basic theory in \autoref{pairs}.  We show how to construct monad pairs $(\bC,\bJ_0)$ from suitable operad pairs in \autoref{pairs2} and we specialize this general theory to  $G$-spaces in \autoref{pairs3}. }

We review the 1980's \cite{LMS} definition of $E_{\infty}$ ring $G$-spectra in \autoref{SPECTRA2}.  These are defined in the present context in \autoref{SPECTRA}, but they are equivalent to any modern reincarnation of the notion.  Their 1990's equivalents \cite{EKMM} were studied in \cite{EMGrings, EHCT, Gmod}.  For example, \cite{Gmod} shows how to construct Brown-Peterson $G$-spectra (not $E_{\infty}$ of course) for any compact Lie group $G$.  Since these older contexts are not well known today, a review is necessary.  

Taking $\sT$ and $\sS$ in our general theory to be (based) $\bJ_0$-$G$-spaces and $\bJ_0$-$G$-spectra, we establish a multiplicative special case of \autoref{ass1} in \autoref{MultThy}.  We then check the remaining assumptions and show how to construct $E_{\infty}$-ring $G$-spectra from $E_{\infty}$ ring $G$-spaces in \autoref{RingThy}.  Although not in the present context, the theory in this section has been known since the 1990's but it has not been written up before.\footnote{That time gap is mainly due to the unfashionability of equivariant stable homotopy theory in the following decades.}

We apply the composite adjunction context of \autoref{ass7} multiplicatively in \autoref{Mult2}.   We establish a general context involving maps of monad pairs in different categories in \autoref{Beck2}. While elementary, such a notion does not seem to appear in the literature.  We indicate how the theory specializes to categories of operators in \autoref{OpsMult}, where examples include equivariant generalizations of those summarized in \cite{Rant2}.  We indicate how the theory specializes to orbital presheaves in \autoref{PreMult},  where interesting applications are work in progress.  In \autoref{Finish}, we briefly indicate, without details, how the whole theory, additive and multiplicative,  can be generalized by composing the orbital presheaf and categories of operators adjunctions. 

\subsection{Remarks on recent related papers}
The theory here starts from a conceptual and equivariant elaboration of the exposition of \cite[Sections 8 and 9]{Rant1}.  The specialization to $G$-spaces is largely a contextualization and summary exposition of parts of \cite{GM3, MMO};  \cite{GM3} focuses on applications that start from categorical input (which we do not repeat here), and \cite{MMO} focuses on the development of the equivariant Segal machine and its comparison with the equivariant operadic machine.  {Neither of those papers deals with orbital presheaves or the multiplicative theory in our Parts 2 and 3 and in the sequels \cite{KMZ2, May26}. The paper \cite{GMMO3} develops the equivariant and multiplicative Segal machine, with no consideration of operadic or space-level structure. It is geared towards application in \cite{GM2} to the description of the category of genuine $G$-spectra as a category of presheaves of nonequivariant spectra.  It fails to deal with commutativity and $E_{\infty}$ ring $G$-spectra.  Those problems are solved in work in progress \cite{May26}.}

\part{The classical theory}

\section{The categorical context}\label{CONTEXT1}

\subsection{\autoref{ass1}:  an adjunction $(\SI,\OM)$ and a related monad $\bC$}

{Categorical monadicity, homotopical monadicity, and the present theory all start with an adjunction $\SI \dashv \OM$.  The first two work with its associated monad $\GA = \OM\SI$, but that has conceptual defects and is of little calculational use.  We focus here on better structured monads $\bC$.}

\begin{ass}\label{ass1}  We assume given an adjoint pair of functors $(\SI,\OM)$,  $\SI\colon \sT \rtarr \sS$ and $\OM\colon \sS\rtarr \sT$, together with a monad $\bC$ in $\sT$ and a natural action of $\bC$ on $\OM Y$ for objects $Y\in \sS$.  We assume further that $\sT$ and $\sS$ are cocomplete.
\end{ass}

\begin{notns}  We write $\vartheta\colon \bC \OM Y \rtarr \OM Y$  for the given natural action.  
We let $\bC[\sT]$ denote the category of $\bC$-algebras in $\sT$, and we let $\OM_{\bC}\colon \sS \rtarr \bC[\sT]$ denote the functor $\OM$, but viewed as taking values in $\bC[\sT]$.   Thus we write $\OM Y$ when forgetting that it is a $\bC$-algebra  and $\OM_{\bC} Y$ when remembering the action of $\bC$.
\end{notns}

\begin{rem}\label{bspt1} \autoref{ass1} sometimes needs minor but essential technical modification, as for example in based spaces $\sU_{\ast}$, which is cocomplete, versus nondegenerately based spaces 
$\sT$, which is not.  Formally, we then have a monad  $\bC$ on $\sU_{\ast}$ which restricts to a monad on $\sT$. We then view $\sT$ as a preferred subcategory of good objects of $\sU_{\ast}$, to which we restrict whenever possible.  

The idea is that $\sU_*$ is essential for formal generality, but $\sT$ is essential for homotopical precision.  Such cofibration issues are handled with care in \cite{MMO} and will only be mentioned in passing here.  The difference is of no practical importance since ``whiskering'' and ``bearding'' constructions of \cite[Proposition 1.6]{MT} and \cite[Proposition 2.12 and Section 11]{MMO} show that general objects of interest can always be replaced by equivalent good objects.
\end{rem}

Except where otherwise specified, we always assume that we are in the context given in \autoref{ass1}, even if we neglect to say so.  We shall define a ``monadically coequalized'' variant $\SI_{\bC}$ of $\SI$ and prove the following basic result in \autoref{MON2}.\footnote{$\SI_{\bC}$ was first defined in  \cite[(8.4)]{Rant1}, but its importance was not understood then.}

\begin{prop}\label{keyadj} The functor $\SI_{\bC}\colon \bC[\sT] \rtarr \sS$ is left adjoint to $\OM_{\bC}$.
\end{prop}

We fix notations for the units and counits of our adjunctions.  

\begin{notns}\label{alpha}  Let $\et\colon  \id \rtarr \OM\SI$ and $\epz\colon \SI\OM\rtarr \id$ be the unit and counit of the adjunction $(\SI, \OM)$.
Let $\et_{\bC}\colon  \id \rtarr \OM_{\bC}\SI_{\bC}$ and $\epz_{\bC}\colon \SI_{\bC}\OM_{\bC}\rtarr \id$ be the unit and counit of  $(\SI_{\bC},\OM_{\bC})$.
\end{notns}

Recall that the adjunction isomorphism is the composite
\begin{equation}\label{adj4} 
\xymatrix@1{  \sS(\SI_{\bC} X, Y)  \ar[r]^-{\OM_{\bC}} & \sT(\OM_{\bC}\SI_{\bC} X, \OM_{\bC} Y) \ar[r]^-{\et_{\bC}^*} & \sT(X, \OM_{\bC} Y)\\}
\end{equation} 

{There are two obvious examples of $\bC$ that satisfy \autoref{ass1}: $\bC$ could be the identity functor, in which case the theory is trivial and uninteresting, or it could be the adjunction monad $\GA = \OM\SI$.   In \cite{KMZ0}, we interpret the categorical monadicity theorem as characterizing when $\et_{\GA}$ and $\epz_{\GA}$ are isomorphisms. We interpret the homotopical monadicity theorem of \cite{KMZ0} as characterizing when $\eta_{\GA}$ is a weak homotopy equivalence and therefore when $\et_{\GA}^*$ and $\OM_{\GA}$ in \autoref{adj4} are isomorphisms or weak homotopy equivalences.  Here we focus on a more useful analog for $\et_{\bC}$, bringing  back  $\epz_{\bC}$ in \autoref{When}.}  

Aside from the notation, the following definition is standard.

\begin{defn} The (free, forgetful) adjunction $(\bF_{\bC},\bU_{\bC})$ of a monad $\bC$ is given by $\bF_{\bC} X = (\bC X, \mu)$ and $\bU_{\bC}(Y,\tha) = Y$. Note that the adjunction monad $\bU_{\bC}\bF_{\bC}$ coincides with the monad $\bC$.  
\end{defn}

{Clearly $\OM = \bU_{\bC}\com \OM_{\bC}$.  The following untwisting isomorphism is a formal consequence (see \autoref{formalBPQ1} and \autoref{formalBPQ2}).    It is a formal precursor of the Barratt--Priddy--Quillen theorem  (see \autoref{BPQthm}). 
} 

\begin{prop}\label{formalBPQ}  For $X$ in $\sT$,  $\SI_{\bC}(\bC X)$ is naturally isomorphic to $\SI X$. 
\end{prop}

These isomorphisms of functors imply the following result.

\begin{prop}\label{alphaC} The adjunction $(\SI,\OM)$ can be identified with the composite $(\SI_{\bC}\com \bF_{\bC}, \bU_{\bC}\com \OM_{\bC})$ of the adjunctions  
$(\SI_{\bC},\OM_{\bC})$ and $(\bF_{\bC},\bU_{\bC})$.
\end{prop}

It is illuminating to display this composite adjunction in a diagram.  We will see many specializations.

\begin{equation}\label{adj5} 
\xymatrix{
\sT    \ar@<.5ex>[rr]^{\SI}  \ar@<.5ex>[dr]^{\bF_{\bC} }& & \sS   \ar@<.5ex>[ll]^{\OM}   \ar@<.5ex>[dl]^{\OM_{\bC}}\\
&   \ar@<.5ex>[ul]^{\bU_{\bC}}  \bC[\sT]    \ar@<.5ex>[ur]^{\SI_{\bC}}& \\}
\end{equation}

\subsection{The map of monads $\al\colon \bC\rtarr \GA = \OM\SI$ and its adjoint $\be$}\label{MON}

Assuming \autoref{ass1}, we here relate the monads $\bC$ and $\GA=\OM\SI$.

\begin{defn} For $X\in  \sT$, define $\al\colon  \bC X \rtarr \OM\SI X$ to be the composite
\[   \xymatrix@1{  \al\colon \bC X\ar[r]^{\bC \et}   &  \bC\OM\SI X  \ar[r]^{\vartheta} &  \OM\SI X, \\}  \]
where $\vartheta$ is the action assumed in \autoref{ass1}. 
Define $\be\colon  \SI\bC X \rtarr \SI X$ to be the adjoint of $\al$, namely the composite
$$
   \xymatrix@1{  \be\colon \SI \bC X  \ar[r]^{\SI \al}  & \SI \OM\SI X  \ar[r]^{\epz} &  \SI X. \\} 
$$
\end{defn}

Most of the following categorical observations about these maps were first noticed in special cases
\cite[Theorem 5.2 and Examples 9.5]{MayGeo}, but the same formal diagram chasing proofs work in general, as in \cite{Rant1}.   The starting point is that $\al \colon \bC X \rtarr \OM_{\bC}\SI X$ is a map of $\bC$-algebras, which is immediate from the definitions. 

\begin{lem}\label{FreeFree}  The map $\al\colon \bC X\rtarr \OM_{\bC} \SI X$ is the map of $\bC$-algebras given by the universal property 
of the free functor  $\bF_{\bC}$, applied to the map $\et\colon X \rtarr \OM\SI X$.  \end{lem}

\begin{thm}\label{alphamon}  The following results hold in the context of \autoref{ass1}.
\begin{enumerate}[(i)]
\item  \label{item:1} $\al\colon \bC \rtarr \GA = \OM\SI$ is a map of monads in $\sT$.
\item  \label{item:2} $\al\colon \bC X\rtarr \OM_{\bC} \SI X$ is the unique map of $\bC$-algebras whose composite
with the unit $\et\colon X \rtarr \bC X$ of $\bC$ is the unit $\et\colon X \rtarr \OM\SI X$ of the adjunction $\SI\dashv\OM$. 
\item  \label{item:3}  The natural action $\vartheta\colon \bC \OM  \rtarr \OM $ is the composite 
$$ \xymatrix@1{ \bC\OM  \ar[r]^-{\al}  &  \GA \OM \ar[r]^-{\OM\epz} &  \OM,\\} $$ 
that is, the pullback along $\al$ of the action  $\OM\epz\colon \GA \OM Y\rtarr \OM Y$ of $\GA$.
\item  \label{item:4} $(\SI,\be)$ is a $\bC$-functor with action the composite
$$ \xymatrix@1{\SI\bC \ar[r]^-{\al}  & \SI \GA \ar[r]^-{\epz} &  \SI, \\} $$
that is, the pullback  along $\SI\al$ of the action  $\epz\colon \SI\GA \rtarr \SI$ of $\GA$.
\end{enumerate}
\end{thm}
\begin{proof}  We will recall the definitions of $\bC$-algebras and $\bC$-functors in \autoref{MON2}. 
 To prove~(\ref{item:1}), we see that $\alpha$ is compatible
    with the units, $\alpha \circ \eta = \eta$, by the following commutative diagram:
    \begin{equation*}
      \begin{tikzcd}
        &\bC\ar[rd,"\bC\eta"']\ar[rrd,"\alpha"]&&\\
        I\ar[rd,"\eta"']\ar[ru, "\eta"]&&\bC\Gamma\ar[r,"\vartheta", xshift=-5pt]&\Gamma\\
        &\Gamma\ar[ru,"\eta"]\ar[rru,equal]&&
      \end{tikzcd}
    \end{equation*}
We see that $\alpha$ is compatible with the product, $\al \com \mu = \mu \com \al \com \bC \al$, by the
 following commutative diagram:
\begin{equation*}
\xymatrix{
 \bC\bC \ar[r]^{\bC\bC\eta}\ar[dd]_{\mu}
& \bC\bC\GA  \ar[r]^-{\bC\vartheta} \ar[dd]_{\mu}  
&   \bC\GA\ar[r]^-{\bC\eta}\ar[dr]^{=} \ar[dd]^{\vartheta}
& \bC\GA\GA \ar[d]^{\bC\Omega\epsilon} \ar[r]^{\vartheta}
& \GA\GA \ar[dd]^{\Omega\epsilon=\mu}\\
& & &\bC\GA \ar[dr]^{\vartheta}&  \\
\bC \ar[r]^{\bC\eta} 
& \bC \GA \ar[r]^-{\vartheta}
& \GA\ar[rr]^{=}  & & \GA\\}
\end{equation*}
Here the left rectangle and right pentagon are naturality diagrams, the small triangle commutes by a triangle identity, the large triangle commutes tautologically, and the remaining rectangle commutes since $\vartheta$ is an action of $\bC$.  For (ii), the unit diagram and the left half of the product diagram above show that $\al$ is a map of $\bC$-algebras, and uniqueness holds since $\bC$ is the free $\bC$-algebra functor. 

For (iii), a triangle identity and naturality give that
$$ \vartheta = \vartheta\com \bC\OM\epsilon\com \bC\et = \OM \epz\com\vartheta\com \bC\et = \OM\epz \com \al. $$
Similarly, for (iv),  it is clear that $\SI$ is a $\GA$-functor via $\epz\colon \SI\GA \rtarr \SI$, and then it is immediate from the definition of $\be$ that it is a $\bC$-functor by pullback along $\bC\al$.
\end{proof}

Taken together, parts (i) and (ii)  have the following converse, which encodes a reformulation of \autoref{ass1}.

\begin{thm}\label{converse}  Let $(\SI,\OM)$ be an adjunction, $\SI\colon \sT \rtarr \sS$ and $\OM\colon \sS\rtarr \sT$, with associated monad 
$\GA = \OM \SI$ in $\sT$.  Let 
$\al\colon \bC \rtarr \GA$ be a map of monads in $\sT$.
\begin{enumerate}[(i)]
\item    $\bC$ acts naturally on objects $\OM Y$ via the pullback action $\vartheta$:
$$ \xymatrix@1{ \bC\OM  \ar[r]^-{\al}  &  \GA \OM \ar[r]^-{\OM \epz} &  \OM.}$$
\item   The given map of monads $\al$ is the composite
$$ \xymatrix@1{ \bC \ar[r]^-{\bC \et} &  \bC \GA \ar[r]^-{\vartheta} & \GA.\\} $$
\end{enumerate}
\end{thm}
\begin{proof}
Part (i) is clear.  Part (ii) holds by the following commutative diagram.
$$\xymatrix{
& \bC  \GA \ar[r]^-{\al}  \ar[ddr]^{\vartheta} & \GA\GA\ar[dd]^{\OM\epz}\\
\bC \ar[ur]^{\bC\et}  \ar[r]_{\bC\et} \ar[dr]_{=} & \bC\bC \ar[u]^-{\bC\al} \ar[d]^{\mu} & \\
& \bC \ar[r]_-{\al} &  \GA\\} $$
The upper left triangle and the rectangle commute since $\al$ is a map of monads; $\vartheta$ is defined by the upper right triangle and it follows that the lower right triangle commutes. The lower  left triangle commutes since $\bC$ is a monad. 
\end{proof}

\subsection{Monadic coequalizers and the adjunction $(\SI_{\bC},\OM_{\bC})$}\label{MON2}  We assume familiarity with coequalizers and split coequalizers (see e.g. \cite[Section 1]{KMZ0}). Let $(\bC,\mu,\et)$, abbreviated $\bC$, be our monad in $\sT$.  This means that $\bC\colon \sT\rtarr \sT$ is a monoid in the functor category of $\sT$ with product $\mu$ and unit $\et$. 

Recall that a $\bC$-algebra $(Y,\tha)$, abbreviated $Y$, is an object of $\sT$ with an action $\tha\colon \bC Y\rtarr Y$ such that $\tha\com \bC\tha = \tha\com \mu$  and $\tha\com \et = \id$.   The following diagram then exhibits $Y$ as a  split coequalizer.
\begin{equation}\label{split2}
\xymatrix@1{
\bC\bC Y  \ar@<2ex>[rr]^{\bC\tha}  \ar@<-1ex>[rr]^{\mu}  &&  \bC Y \ar[rr]^-{\tha} \ar@<+3ex>[ll]^{\et}& & Y \ar@<+2ex>[ll]^{\et}. \\}  
\end{equation}
Here $\bC\tha\com \et = \et\com  \tha$ by the naturality of $\et$ applied to the morphism $\tha$.\footnote{The first $\et$ is $\et$ on $\bC Y$; category theorists might usually write it as  $\et \com \bC$.}

Recall too that a $\bC$-functor  $(\SI,\be)$, abbreviated $\SI$, is a functor  $\SI\colon \sT \rtarr \sZ$ for some category $\sZ$ with an action  $\be\colon \SI\com \bC \rtarr \SI$ such that $\be\com\SI \mu= \be\com \be$ and $\be\com \SI \et = \id$ \cite[Definition 9.4]{MayGeo}.  We will be focusing on the case  $\sZ= \sS$.  In analogy with \autoref{split2}, the following diagram exhibits the functor $\SI$ as naturally given by a split coequalizer for each object $X\in \sT$.

\begin{equation}\label{split3}
\xymatrix@1{
\SI\bC\bC X  \ar@<2ex>[rr]^{\be}  \ar@<-1ex>[rr]^{\SI\mu}  && \SI \bC X \ar[rr]^-{\be} \ar@<+3ex>[ll]^{\SI \bC \et}& & \SI X \ar@<+2ex>[ll]^{\SI \et}. \\}  
\end{equation}
Here $\be\com \SI\bC \et = \SI\et\com \be$ by the naturality of $\be$ applied to the morphism $\et$. 

As seen in \autoref{MON}, the adjunction $(\SI,\OM)$ in \autoref{ass1} gives
an action $\be$ of $\bC$ on $\SI$. Our categorical starting point makes sense  for any $\bC$-functor $\SI$.

\begin{defn}\label{twist} Let  $\SI\colon \sT\rtarr \sS$ be a $\bC$-functor.  Define the $\bC$-coequalized functor 
$\SI_{\bC}\colon \bC[\sT] \rtarr \sS$ to be the tensor product $\SI\otimes_{\bC} (-)$.  Explicitly, on a $\bC$-algebra $Y$, it is given by  the coequalizer in $\sS$ displayed in the diagram
\begin{equation}\label{split4} 
 \xymatrix@1{
\SI {\bC} Y  \ar@<.7ex>[rr]^{\be}  \ar@<-.7ex>[rr]_{\SI\tha}  &&  \SI  Y \ar[rr]^-{q} & & \SI_{\bC}Y. \\}  
\end{equation}
\end{defn}

{
\begin{rem}\label{formalBPQ1} These definitions make a generalized version of \autoref{formalBPQ}  a tautology.  Specializing $Y$ in \autoref{split4} to $Y=\bC X$,  we see that the same pairs of arrows are being coequalized as in \autoref{split3}.     Therefore the two coequalizers can be identified, thus identifying $q\colon \SI \bC X \rtarr \SI_{\bC} \bC X$  with $\be$.  Observe that, since  $\be\com \SI_\bC \mu = \be\com \be$,  $\SI_\bC \mu$ agrees with $\be$ under the isomorphism $\SI_{\bC}(\bC \bC X)\iso \SI \bC X$. 
\end{rem}
}

Returning to \autoref{ass1}, we now specialize $\SI_{\bC}$ to that context.   

\begin{proof}[Proof of \autoref{keyadj}] Consider the following diagram, in which  $Y$ is in $\bC[\sT]$  and $Z$ is in $\sS$. 
\begin{equation}\label{adjs} \xymatrix{
\sS (\SI_{\bC}Y, Z) \ar[d]_{q^*} \ar@{-->}[r] &  \bC[\sT](Y, \OM_{\bC} Z) \ar[d]^{\subset}\\
\sS(\SI Y, Z)  \ar[r]_-{\iso} &  \sT(Y, \OM Z).\\}
\end{equation}
We claim that  a map  $f\colon Y \rtarr  \OM Z$  in $\sT$ is a map of $\bC$-algebras if and only if its adjoint $\tilde{f}\colon \SI Y \rtarr Z$ factors uniquely through $q\colon \SI Y \rtarr \SI_{\bC} Y$.  That will give the top isomorphism.  Thus consider the diagram
\begin{equation}\label{fact}
\xymatrix{
\bC Y \ar[r]^-{\bC f} \ar[d]_{\tha}& \bC\OM_{\bC} Z \ar[d]^{\vartheta} \\
Y \ar[r]_-{f} & \OM_{\bC} Z.\\}
\end{equation}
It commutes, so that $f$ is a map of $\bC$-algebras, if and only if its adjoint
commutes. The top row of the following diagram in $\sT$ is the adjoint 
$\widetilde{\vartheta\com \bC f}$ of $\vartheta\com \bC f$; its left column is $\be$, the adjoint of $\al$.
\begin{equation}\label{biggy}
\xymatrix{
\SI\bC Y \ar[r]^-{\SI\bC f} \ar[d]_{\SI\bC\et} & \SI\bC\OM_{\bC} Z \ar[r]^-{\SI\vartheta} \ar[d]_{\SI \bC \et}&  \SI\OM_{\bC}Z  \ar[dd]_{\SI\et}\ar[r]^-{\epz}  &  Z\\
\SI\bC\OM_{\bC} \SI Y \ar[r]^-{\SI\bC\OM_{\bC}\SI f}  \ar[d]_{\SI\vartheta} & \SI\bC\OM_{\bC} \SI \OM_{\bC} Z \ar[dr]_{\SI \vartheta} &  &  \\
\SI\OM_{\bC} \SI Y \ar[rr]_-{\SI\OM_{\bC}\SI f}  \ar[d]_{\epz} & & \SI\OM_{\bC} \SI \OM Z \ar[dr]_{\epz} & \\
\SI Y \ar[rrr]_{\SI f} & & & \SI\OM Z \ar[uuul]_{\id} \ar[uuu]^{\epz}\\}
\end{equation}
The diagram commutes by naturality and a triangle identity.   This identifies $\widetilde{\vartheta\com \bC f}$ with the composite
\begin{equation}\label{un} \xymatrix@1{ \SI\bC Y   \ar[r]^-{\be}  & \SI Y \ar[r]^{\SI f} & \SI\OM Z  \ar[r]^{\epz}  &  Z.\\} 
\end{equation}
The adjoint $\widetilde{f\com \tha}$ of $f\com \tha$ is the composite
\begin{equation}\label{deux}
\xymatrix@1{ \SI\bC Y \ar[r]^-{\SI\tha}    & \SI Y \ar[r]^{\SI f} & \SI\OM Z  \ar[r]^{\epz}  &  Z.\\}
\end{equation}
These two composites are equal if and only if the adjoint of \autoref{fact} commutes, which holds if and only if \autoref{fact} itself commutes.  By the universal property of the coequalizer $\SI_{\bC} Y$, this holds if and only if the composite $\epz\com \SI f$ factors uniquely through $\SI_{\bC} Y$.   This proves our claim.
\end{proof}

\begin{cor}\label{eye}  The unit $\et_{\bC}$ and counit  $\epz_{\bC}$  of the adjunction are described by the following diagrams.  Here a little chase shows that  $\epz\com  \be = \epz\com \SI\vartheta$. 
$$ 
\xymatrix{  
Y \ar[r]^-{\et}  \ar[dr]_{\et_{\bC}} & \OM\SI Y \ar[d]^{\OM q}   &   \text{and}  &      \SI {\bC} \OM Z  \ar@<.7ex>[rr]^{\be}  \ar@<-.7ex>[rr]_-{\SI \vartheta}  &&  \SI  \OM Z \ar[rr]^-{q}  \ar[d]^{\epz} & & \SI_{\bC} \OM_{\bC} Z  \ar[dll]^{\epz_{\bC}} \\
& \OM_{\bC}\SI_{\bC} Y                                                               &         &             & &   Z   & & \\}
$$
\end{cor}

{
For  $\bC = \GA$ and $\al = \id$, the following  consequence is \cite[Corollary 1.8]{KMZ0}.

\begin{cor}\label{eye3}  Let $\GA_{\bC} = \OM_{\bC}\SI_{\bC}$ be the adjunction monad in $\bC[\sT]$. The following is a composite of maps of monads: 
$$\xymatrix{\bC \ar[r]^-{\al} & \OM\SI \ar[r]^-{\OM q} & \OM_{\bC}\SI_{\bC}= \GA_{\bC}\\} $$
For $\OM q$, this means that the unit and product diagrams commute after application of the forgetful functor $\bU_{\bC}\colon \bC[\sT]\rtarr \sT$.
\end{cor}

Since $\OM$ is $\bU_{\bC} \OM_{\bC}$ and $\bF_\bC$ is the left adjoint to $\bU_{\bC}$, for $X\in \sT$ and $Z\in \sS$ we have
$$\sS(\SI X,Z)  \iso \sT(X, \bU_\bC \OM_{\bC}Z) \iso \bC[\sT](\bC X,\OM_{\bC}Z) 
\iso \sS(\SI_{\bC} \bC X,Z).$$
Using that $q = \be$ as  in \autoref{formalBPQ1}, inspection and a triangle identity give a more explicit version of \autoref{formalBPQ}.

\begin{prop}\label{formalBPQ2}  The natural isomorphism $\SI_{\bC}\bC \rtarr \SI$ is given by the following composite, which we name $\be_{\bC}$:
$$\xymatrix{\SI_{\bC}\bC \ar[r]^-{\SI_{\bC}\al} & \SI_{\bC} \OM_{\bC} \SI \ar[r]^-{\epz_{\bC}} & \SI \\} $$
\end{prop}


A closely related result plays a central role in our theory.  Intuitively, it shows that the map $\al$ can be interpreted as a special case of the unit $\et_{\bC}$.

\begin{lem}\label{eye2}  The  following commutative diagram identifies $\al$ as the composite 
of the isomorphism $\OM_{\bC}\be_{\bC}$ with $\et_{\bC}$:
$$\xymatrix{  
& \OM\SI X \ar[dr]_{\et_{\bC}}  \ar[drr]^-{=} & & \\
\bC X \ar[ur]^{\al} \ar[dr]_{\et_{\bC}}  & & \OM_{\bC}\SI_{\bC} \OM\SI X \ar[r]^{\OM_{\bC}\epz_{\bC}} & \OM_{\bC}\SI X \\
& \OM_{\bC} \SI_{\bC} \bC X \ar[ur]^-{\OM_{\bC}\SI_{\bC}\al}  \ar[urr]_-{\OM_{\bC} \be_{\bC}} & & \\}
$$
\end{lem}
}

 
We think of $\et_{\bC}$ as formal and $\al$ as relatively concrete.  It is $\al$ rather than $\et_{\bC}$ that leads to computational applications.

\begin{warn}  The functor $(\OM\com\SI)\otimes_{\bC}(-)$ is also defined, but it appears to be of little or no interest.  It must not be confused with $\OM_{\bC}\com \SI_{\bC}$.
\end{warn}

\section{The homotopical context}\label{CONTEXT2}

\subsection{Assumption \ref{ass2}: basic homotopical assumptions}
\autoref{ass1} specifies a minimal formal framework, and we flesh it out with minimal homotopical assumptions.  We separate out those basic assumptions that assume nothing substantial about the monad $\bC$ in this section.  These are the same as the assumptions used to prove the homotopical monadicity theorem for $\GA$ in \cite{KMZ0}.  

As there, our proofs of homotopical monadicity focus on derived homotopy theory, defined in terms of a bar resolution  of algebras over monads.  That is a special case of the two-sided monadic bar construction, as is the homotopically well-behaved variant $\bE$ of the left adjoint $\SI_{\bC}$ that will be used in the proofs.  These are constructed simplicially, and we write $s\sV$ for the category of simplicial objects in a category $\sV$. For a functor $F$, we write $F_*$ for the functor on simplicial objects given by applying $F$ objectwise on $q$-simplices, and similarly for natural transformations.  

\begin{ass}\label{ass2}  We assume that $\sT$ and $\sS$, in addition to being cocomplete, have standard notions of homotopy and have classes of weak equivalences satisfying the two out of three property.  We say that a map $f$  in $\bC[\sT]$ for any $\bC$ is a weak equivalence if it is a weak equivalence when considered as a map in $\sT$.   We assume that the functors $\OM$, $\SI$, and 
$\bF_{\bC}$, hence also $\OM_{\bC}$ and $\bC$, preserve weak equivalences, at least under restriction to good objects (as in \autoref{bspt1} for example).  We assume the following statements about simplicial objects in $\sT$ and $\sS$.  
\begin{enumerate}[(i)]
\item  There are realization functors $\bT \colon s\sT\rtarr \sT$ and $\bT \colon s\sS\rtarr \sS$ and they are left adjoints.
\item  The functors $\bT $  on  $s\sT$ and $s\sS$ preserve homotopies.
\item  The functors $\bT $  on  $s\sT$ and $s\sS$ preserve weak equivalences between Reedy cofibrant objects.\footnote{The name comes from Reedy's thesis \cite{Reedy}. The  word ``proper'' was sometimes used before Reedy cofibrancy was introduced \cite{MayGeo, MayPerm}.}
\item  Realization commutes with the (left adjoint) functor $\SI$. That is,  for $X_*$ in $s\sT$, there is a natural isomorphism
$\xymatrix@1{  \SI \bT X_* \iso \bT \SI_* X_*.}$
\item   For $X_*\in s\sS$,  the natural map $\ga\colon \bT \OM_* X_*\rtarr \OM \bT X_*$, which is the adjoint of 
 $$ \xymatrix@1{\SI \bT \OM_*X_* \iso \bT \SI_* \OM_* X_* \ar[r]^-{\bT\epz_*} & \bT X_*,\\}$$
 is a weak equivalence if $X_{*}$ is  levelwise $\OM$-connective.
\end{enumerate}
\end{ass}

The appropriate general definition of an $\OM$-connective object of $\sS$ was given in \cite{KMZ0} and will be recalled in \autoref{pfadjequiv}.  It coincides with more familiar notions in our examples.   The simplicial assumptions are discussed in \cite{KMZ0}.  All except (v) are essentially formal, with no real work required in their verification.

\begin{rem}\label{bspt2} As discussed in \autoref{MODEL} below, $\sT$ and $\sS$  are generally model categories such that $(\SI,\OM)$ and $(\bF_{\bC},\bU_{\bC})$ are Quillen adjunctions.  The reader may prefer to assume that holds and restrict the preservation properties accordingly.  This can depend on the choice of relevant model structures.  Our applications will focus mainly on topological model categories in which all objects are fibrant, hence the right adjoints preserve all weak equivalences, as we assume. The left adjoints do not, but rectifying this generally requires much less than cofibrant approximation.  Restricting attention to good objects as in \autoref{bspt1} is sufficient and much less of a nuisance in practice.   As said before,  we tacitly make this restriction throughout.  We prefer our objects as they come in nature, with their natural algebraic structure.
\end{rem}

\begin{rem}\label{ouch}  The assumption that $\SI$ preserves weak equivalences should be regarded as negotiable.  It fails in a very interesting way in \autoref{orbpre}, but, as there, it can be made true by cofibrant approximation.  Even when it holds, the functor $\SI_{\bC}$ is problematic.  We do not believe that $\SI_{\bC}$ preserves weak equivalences, even when restricted to good objects.  That will account for some subtleties later.  However, as we shall see in \autoref{Eweak}, the assumption on $\SI$ ensures that the composite $\bE = \SI_{\bC}\com \mathrm{Bar}$ defined in \autoref{Edefn} does preserve weak equivalences (implicitly, between good objects) regardless of whether or not  $\SI_{\bC}$ does so.
\end{rem}

As in \cite[2.7 and 2.8]{KMZ0}, we write $c_* J$ for the constant simplicial object at an object $J$ of any category.  Its $q$-simplices are $J$ for any $q$, with face and degeneracy maps the identity.  We assume, as may be formal and always holds in practice, that its realization $\bT c_*J$ can be identified with $J$.

\begin{lem}\label{triv}
For a simplicial object $K_*$ in a category $\sT$ and an object $J$ of $\sT$, a map $f\colon K_0 \rtarr J$ such that $f\circ d_0 = f\circ d_1$ induces a map $K_* \rtarr c_* J$  and thus a map $\xi\colon \bT K_* \rtarr J$ in $\sT$.
\end{lem}

\subsection{\autoref{ass3}: assumptions about the monad $\bC$}\label{Csec}

We need two basic simplicial properties of the monad $\bC$, both conceptual rather than homotopical.

\begin{ass}\label{ass3} 
\begin{enumerate}[(i)]
\item Realization commutes with the functor $\bC$. That is,  for $X_*$ in $s\sT$, there is a natural isomorphism
$\nu\colon  \bT \bC_* X_* \rtarr \bC \bT X_*$ such that the following diagrams commute.
$$ \xymatrix{
& \bT X_* \ar[dl]_-{|\et_*|}  \ar[dr]^{\et} &        & & 
\bT \bC_*\bC_* X_* \ar[d]_{\bT \mu_*}   \ar[r]^-{\nu} &\bC \bT \bC_* X_*  \ar[r]^-{\bC\nu} & \bC\bC \bT X_* \ar[d]^{\mu}\\
\bT\bC_*X_* \ar[rr]_-{\nu} & & \bC \bT X_* & & \bT \bC_* X_* \ar[rr]_-{\nu}  &  & \bC \bT X_* \\}
$$
Therefore $\bT \bC_*X_*$ is a $\bC$-algebra.  
\item For $X_*\in s\sS$, the natural map $\ga\colon \bT \OM X_*  \rtarr \OM \bT X_*$  is a map of $\bC$-algebras.  That is, the following diagrams commute.
$$ \xymatrix{
\bT \OM_*X_* \ar[r]^-{\et} \ar[d]_{\ga} & \bC \bT \OM_* X_* \ar[d]^{\bC\ga} & & 
\bC \bT \OM_* X_* \ar[d]_{\bC \ga} & \bT\bC_*\OM_* X_*  \ar[l]_-{\nu} \ar[r]^-{\bT \vartheta_*} &  \bT \OM_* X_* \ar[d]^{\ga}\\
\OM \bT X_* \ar[r]_-{\et} & \bC\OM\bT X_* & & \bC\OM \bT X_*  \ar[rr]_-{\vartheta}  &  & \OM \bT X_* \\}
$$
\end{enumerate}
\end{ass}

\begin{rem}\label{twist2}  By \autoref{twist},  Assumptions \ref{ass3}(i) and \ref{ass2}(iv) imply that, for simplicial objects $X_*$  in $\bC[\sT]$, 
$\bT X_*$ is in  $\bC[\sT]$ and
$$\SI_{\bC}\bT X_* \iso \bT \SI_{\bC}X_*.$$
\end{rem}

\begin{rem}\label{critical} These assumptions clearly differentiate $\bC$ from the adjunction monad $\GA$. For the latter, \autoref{ass3}(i) is false since $\SI$ commutes with $\bT $ and $\OM$ generally does not.  Due to this falsity, we do not have a proof that $\bT \OM K_*$ is a $\GA$-algebra, so that (ii) does not make sense.  We showed in \cite{KMZ0} how to accommodate this failing.   This difference between contexts is crucial. 
\end{rem}

The monads associated to operads are left adjoints constructed using colimits and products, so that (i) follows from the categorical Fubini theorem and the commutation of $\bT $ with products.   In that case, an explicit detailed proof of (i) is given in \cite[Theorem 12.2]{MayGeo}.
In the special case of a monad associated to a particular operad, (ii) is proven in \cite[Theorem 12.4]{MayGeo}, but the argument can almost be formalized as follows.   
$$ \xymatrix{
\bT \OM_*X_* \ar[r]^-{\et} \ar[d]_{\ga} & \bC \bT \OM_* X_* \ar[d]^{\bC\ga} & & 
\bC \bT \OM_* X_* \ar[d]_{\bC \ga} &  \bT \bC_*\OM_* X_* \ar[l]_{\nu} \ar[r]^-{\bT \vartheta_*} &  \bT \OM_* X_* \ar[d]^{\ga}\\
\OM \bT X_* \ar[r]_-{\et} & \bC\OM\bT X_* & & \bC\OM \bT X_*  \ar[rr]_-{\vartheta}  &  & \OM \bT X_* \\}
$$
The first diagram commutes by the naturality of $\eta$. 
Passing to adjoints and writing $\tilde{\ga}$ for the adjoint of $\ga$, the second diagram becomes
\[  \xymatrix{
\SI \bC \bT \OM_*X_* \ar[d]_{\SI\bC \ga}&  \ar[l]_-{\SI\nu}  \SI \bT \bC_*\OM_* X_*  \ar[r]^-{\SI \bT \vartheta} & \SI \bT \OM_*X_*\bT \ar[d]^{\SI\ga} \ar[r]^-{\iso} \ar[dr]^{\tilde{\ga}} & 
\bT \SI_*\OM_* X_*  \ar[d]^{\bT \epz}\\
\SI \bC\OM \bT X_*  \ar[rr]_-{\SI \vartheta}  &  & \SI \OM \bT X_* \ar[r]_-{\epz} & \bT X_*.\\}
\]
The rectangle is obtained by applying $\SI$ to what is in spirit a naturality diagram; it commutes by inspection in examples since $\vartheta$ is defined the same way before and after realization.  The lower triangle commutes by a chase of adjoints and the upper triangle commutes by definition.

\subsection{\autoref{ass4}: group completions}

The failings of the monad $\GA$ documented in \autoref{critical} are compensated for and related to the fact that no notion of group completion was relevant to the homotopical monadicity theorem of \cite{KMZ0}.  However, the notion of group completion used throughout iterated loop space theory requires clarification.   A priori, the following curious definition has nothing at all to do with groups or with \autoref{ass1}.    It formalizes the relevant properties of group completion, and groups do appear prominently in examples.\footnote{We thank Dustin Clausen for noticing serious sloppiness in our first version of the definition.}

\begin{defn}\label{screwy}  We abstract what we mean by notions of grouplike objects and group completions in a category $\sT$ with a subcategory of weak equivalences.  We mean that we have a category  $\sH$ with a forgetful functor $\bU\colon \sH \rtarr \sT$ together with a full subcategory $\sH_{gp}$ of $\sH$, whose objects are called grouplike.  We say that a map $f$ in $\sH$ is a weak equivalence if 
$\bU f$ is a weak equivalence in $\sT$ and we require that objects of $\sH$ weakly equivalent to grouplike objects be grouplike.

We mean further that we have a concomitant notion of a group completion $f\colon X\rtarr Y$ in $\sH$, where $Y$ is grouplike. The class of group completions must satisfy the  following properties.
\begin{enumerate}[(i)]
\item A map $f\colon X \rtarr Y$ in $\sH$, where both $X$ and $Y$ are grouplike, is a group
  completion if and only if it is a weak equivalence.
\item Consider a commutative diagram of maps in $\sH$
\begin{equation}\label{square}
 \xymatrix{
X \ar[r]^f   \ar[d]_{g} & Y \ar[d]^h\\
Z  \ar[r]_{j}  &  W.\\}
\end{equation}
\begin{enumerate}[(a)]
\item If $Z$ and $W$ are grouplike  and $f$ and $j$ are weak equivalences, so that the diagram is a weak equivalence $g\rtarr h$ in the arrow category of $\sH$, then $g$ is a group completion if and only if $h$ is a group completion.
\item If $Z$ and $W$ are grouplike, $g$ and $h$ are group completions, and $f$ is a weak equivalence, then $j$ is a weak equivalence.
\item If $Y$, $Z$, and $W$ are grouplike and $h$ and $j$ are weak equivalences, then $f$ is a group completion if and only if $g$ is a group completion.
\item If $Y$, $Z$, and $W$ are grouplike and  $f$ and $g$ are group completions, then $h$ is a weak equivalence if and only if $j$ is a weak equivalence.
\end{enumerate}
\end{enumerate}
We regard  (ii) as an analog of the two out of three property of weak equivalences. 
\end{defn}

The examples of $\sH$ given in Definitions \ref{Hopf} and \ref{Hopf2} when $\sT$ is the category of based spaces should help motivate this abstract  definition. Given $\sH$, we now bring a monad $\bC$ into the picture.  

\begin{defn}\label{screwy2}  Let $\bC$ be a monad whose forgetful functor $\bU_{\bC}\colon \bC[\sT] \rtarr \sT$ factors as $\bU \com \bV$, where $\bV$ is a forgetful functor $\bC[\sT]\rtarr \sH$. 

\begin{enumerate}[(a)] 
\item Say that a $\bC$-algebra $Y$ is grouplike if $\bV Y$ is grouplike and say that a map $f$ of $\bC$-algebras is a group completion if $\bV f$ is a group completion in $\sH$.
\item Define a group completion functor  $\bG$ to be a functor $\bG\colon \bC[\sT] \rtarr \sH_{gp}$ together with a natural group completion  $g\colon \bV Y \rtarr \bG Y$ in $\sH$ and thus $\bU_{\bC}Y \rtarr \bU \bG Y$ in $\bC[\sT]$ for $\bC$-algebras $Y$.  \end{enumerate}
\end{defn} 

We later drop the functor $\bU$ from the notation, writing this as $\bU_{\bC}Y \rtarr \bG Y$.  Note that $\bG$ depends only on $\bC$, independent of its relationship to $(\SI,\OM)$. 

\begin{rem} By (i) of \autoref{screwy}, $g$ in (b) is a weak equivalence if $Y$ is grouplike; by (i),  (ii)(c), and (ii)(d)  applied to a naturality diagram, a map $f\colon X\rtarr Y$ of $\bC$-algebras such that $Y$ is grouplike is a group completion if and only if $\bG f$ is a weak equivalence. 
\end{rem}

\begin{rem}  One might expect that group completions satisfy the universal property that if $g\colon X\rtarr Y$ is a group completion and $f\colon X\rtarr Z$ is a map to a grouplike object, then $f$ factors uniquely through $g$.  We do {\em not} assume that here.  A group completion with such a universal property will enter in \autoref{fully}.  
\end{rem}

With these definitions in hand, we return to the context of \autoref{ass1}.

\begin{ass}\label{ass4}  We assume the following properties of  $\sT$.
\begin{enumerate}[(a)]
\item $\sT$ comes with a category $\sH$ with the properties specified in \autoref{screwy}, and the forgetful functor $\bU_{\bC}\colon \bC[\sT] \rtarr \sT$ factors through a forgetful functor $\bC[\sT]\rtarr \sH$. 
\item $\OM_{\bC} Z$ is a grouplike object of $\bC[\sT]$  for every object $Z$ of $\sS$ and there is a group completion functor  $(\bG,g)$, as specified in \autoref{screwy2}(b).
\end{enumerate}
\end{ass}

\subsection{\autoref{ass5}: the approximation theorem}\label{APPROX}

We place ourselves in the context of Assumptions \ref{ass1} and \ref{ass2}, except that we ignore all references to simplicial objects in this section, and we also assume \autoref{ass4}.  The following assumption,\footnote{We call this assumption a theorem both because the ``approximation theorem'' has a long history and because, in contrast with the other assumptions, all of which make sense in greater generality, this assumption requires a theorem specific to the context at hand.}  is by far our most substantial one.    It says that the monad $\bC$ is a homotopical approximation of the monad $\OM\SI$, and we will comment on its proof when we specialize our framework to examples.

\begin{ass}\label{ass5} [Approximation theorem] For an object $X\in \sT$, the map  $$\al\colon  \bC X \rtarr \OM_{\bC}\SI X$$ defined in \autoref{alpha} is a group completion.  It is therefore a weak equivalence when $\bC X$ is grouplike.
\end{ass}

We will explain  how to go from this result to \autoref{recprin} in \autoref{pfrecprin}, and we will  explain how to go from \autoref{recprin} to \autoref{adjequiv} in \autoref{pfadjequiv}.

\begin{rem} We remark that $\bC X$  is very often not grouplike.  In our topological examples, it  is grouplike only when  $X$ is connected.  The only role of the group completion functor $\bG$ of \autoref{ass4} is to serve as a tool that  will enable us to deduce the recognition principle from the approximation theorem.
\end{rem}

\section{Proofs and variant versions of Theorems \ref{recprin} and \ref{adjequiv}}\label{BARCON}

\subsection{The original version of the recognition principle}\label{REC}

For a monad $(\bC,\mu,\et)$, a $\bC$-functor $(\SI,\be)$ and a $\bC$-algebra $(Y,\tha)$, we have a monadic bar construction  $B(\SI,\bC, Y)$, namely the realization of the evident simplicial object in $\sS$ with  $q$-simplices  $\SI \bC^q Y$; its faces and degeneracies are given by $\be$, $\mu$, 
$\tha$ and $\et$. The following definition has been implicit throughout.

\begin{defn}\label{Edefn}   For $Y\in \bC[\sT]$, define $\overline Y = B(\bC,\bC,Y)$ and $\bE Y = \SI_{\bC} \overline Y$.  
We give the functor $Y\mapsto \overline Y$ the name $\mathrm{Bar}$. 
\end{defn}

A standard extra degeneracy argument \cite[Proposition 9.8]{MayGeo}, using (ii) of  \autoref{ass2} and  \autoref{triv}, proves the following result. 

\begin{lem}\label{extra}   For $Y\in \bC[\sT]$,  the evident maps  $\bC^{q+1} Y \rtarr Y$ induce a natural map $\ze\colon \bB Y  \rtarr Y$ of $\bC$-algebras.   Viewed as a map in $\sT$,  it is a homotopy equivalence with homotopy inverse  $\nu\colon Y \rtarr \bB Y$ induced by the maps  $\et\colon Y \rtarr \bC^{q+1} Y$ .
\end{lem}
\begin{rem} We emphasize that $\et$ and $\nu$ are only maps in $\sT$, not in $\bC[\sT]$.
\end{rem}

Specializing  to the $\SI$ of \autoref{ass1},  \autoref{formalBPQ} gives isomorphisms 
$$  \SI \bC^q Y \iso \SI_{\bC}\bC^{q+1} Y.$$
This gives an isomorphism of simplicial objects in $\sS$, \autoref{formalBPQ1} and \autoref{formalBPQ2} implying  the compatibility of zeroth faces. Passing to realizations, using the commutation of $\SI_{\bC}$ with realization, this gives the following result.\footnote{This is \cite[(8.6)]{Rant1}, but it was not given sufficient emphasis there.}

\begin{lem}\label{oldkey} For $Y\in \bC[\sT]$,  $\bE Y = \SI_{\bC}\overline{Y}$ is naturally isomorphic to  $B(\SI,\bC,Y)$.
\end{lem}

By Assumptions \ref{ass2}(iii), together with our implicit restriction to good objects, this implies the following result.

\begin{lem} \label{Eweak}The functor $\bE$ preserves weak equivalences.
\end{lem}

By taking $Y=\bC X$, we have a symmetric variant of \autoref{extra}.  

\begin{lem}\label{extra2}   For $X\in \sT$,  the evident maps  $\SI \bC^{q+1} X \rtarr \SI X$ induce a natural map $\ze\colon \bE \bC X  \rtarr \SI X$ in $\sS$, and it is a homotopy equivalence with homotopy inverse  $\nu\colon \SI X \rtarr \bE{\bC X}$ induced by the maps  $\SI \bC^q\et \colon \SI \bC^q X \rtarr \SI \bC^{q+1} X$.
\end{lem}

\begin{rem}\label{subtle} \autoref{extra2} also applies with $\bC$ replaced by $\GA$, but \autoref{oldkey} does not. We shall use this fact implicitly in \autoref{pfadjequiv}.
\end{rem}

Observe that the extra degeneracy is seen in the first variable of $B(-,-,-)$, using $\eta$, in \autoref{extra} but is seen in the third variable, using the $\bC^{q}\et$, in \autoref{extra2}.

\autoref{oldkey} recovers the description of $\bE$ first given topologically in \cite{MayGeo}.  The functor $\bE$ and others of the same general form specialize to give all known variants of the operadic recognition principle and several new ones.  Interpreted model categorically, as in \autoref{MODEL}, we can view  $\bE$ as the composite of the well-structured cofibrant approximation $\mathrm{Bar}$ and the monadically coequalized functor $\SI_{\bC}$.  The equivalent  original form $B(\SI,\bC,Y)$ implicitly views this as a homotopically well-behaved derived approximation of the monadic tensor product $\SI\otimes_{\bC} Y$. 

We shall explain three versions of the recognition principle. The first, proven in this subsection,  is a generalized form of the original version in \cite{MayGeo}.  We show in the next subsection how that version implies the more conceptual version stated in \autoref{recprin}.  Finally, we give a new fully conceptual version in subsection \ref{fully}.

With Assumptions \ref{ass1}, \ref{ass2}, and \ref{ass3},  the following diagram of maps of $\bC$-algebras compares the identity functor with $\OM_{\bC} \bE$. 
\begin{equation}\label{master}
\xymatrix@1{  Y & \bB Y \ar[l]_-{\ze} \ar[r]^-{B\al} &   B(\OM_{\bC}\SI,\bC, Y) \ar[r]^-{\ga} &  \OM_{\bC}B(\SI,\bC,Y).\\} 
\end{equation}
Here $\ze$ is a homotopy equivalence in $\sT$ and $\ga$ is a weak equivalence.   The map $B\al$ is 
$$B(\al,\id,\id)\colon B(\bC, \bC, Y) \rtarr B(\OM_{\bC}\SI,\bC, Y).$$  
It  is a map of $\bC$-algebras since $\al$ is a map of  $\bC$-algebras.  Generalizing \cite[Theorem 2.3(ii)]{MayPerm}, we explain how the approximation theorem, \autoref{ass5}, implies that $B\al$ is a group completion.  As said before, this is the only place where the group completion functor $\bG$  of \autoref{screwy} and \autoref{ass4} comes into play.

\begin{prop} If  our Assumptions hold, then $B\al$ is a group completion.
\end{prop}
\begin{proof}
Let $Y$ be a $\bC$-algebra and consider the following commutative diagram.
\[  \xymatrix{
Y\ar[d]_{g} & B(\bC,\bC,Y) \ar[r]^-{B\al} \ar[d]_{Bg} \ar[l]_-{\ze} & B(\OM_{\bC} \SI,\bC,Y) \ar[d]^{Bg} \\
\bG Y & B(\bG\bC,\bC,Y) \ar[r]^-{B(\bG\al)} \ar[l]_-{\ze} & B(\bG\OM_{\bC}\SI,\bC,Y)  \\}
\]
By (ii)(a) of \autoref{screwy}, since $g$ is a group completion and the $\ze$ are weak equivalences, the middle vertical arrow  $Bg = B(g,\id,\id)$ is a group completion.  By (ii)(d) of \autoref{screwy}, $\bG\al\colon  \bG\bC X \rtarr \bG\OM_{\bC}\SI X$ is a weak equivalence for any $X\in \sT$ since 
$\al\colon \bC X \rtarr \OM_{\bC}\SI X$ and $g\colon \bC X \rtarr \bG\bC X$ are group completions and 
$g\colon \OM_{\bC}\SI X \rtarr \bG \OM_{\bC}\SI X$ is a weak equivalence
by (i) of \autoref{screwy}.   Therefore $B(\bG \al)=B(\bG\al,\id,\id)$ is the realization of a levelwise weak equivalence and is thus a weak equivalence.  Similarly, the right vertical arrow $Bg$ is the realization of a levelwise weak equivalence and is thus a weak equivalence.  By  (ii)(c) of  \autoref{screwy}, it follows that $B\al$ is a group completion.
\end{proof}

Putting things together,  the following generalization of the original version of the recognition principle follows from our Assumptions.  In view of the long history of applications, we view it as the most calculationally accessible version, although, as we shall see, we now view it as somewhat lacking in conceptual clarity and elegance.

\begin{thm}[Recognition principle, original version]\label{recprin0} The following statements about the maps of $\bC$-spaces in \autoref{master} hold.
\begin{enumerate}[(i)]
\item The map $\ze$ is a homotopy equivalence with a natural homotopy inverse $\nu$.
\item The map ${B\al}$ is a group completion and is therefore a weak equivalence if the $\bC$-algebra $Y$ is grouplike.
\item The map $\ga$ is a weak equivalence.
\end{enumerate}
Therefore the composite 
\begin{equation}\label{comp}
\ga\com B\al \com \nu\colon Y \rtarr \OM_\bC \bE Y
\end{equation}
is a  group completion and is thus a weak equivalence if $Y$ is grouplike. 
\end{thm}

\subsection{A more conceptual version of the recognition principle}\label{pfrecprin}
To go from the original version to the more conceptual version stated in  \autoref{recprin}, we use the following conceptual reinterpretation
of the composite $\ga\com B\al$.

\begin{prop}\label{Yeah}
The composite $\ga\com B\al$ in \autoref{master} can be identified with the unit 
$$\et_{\bC}\colon  \bB Y \rtarr \OM_{\bC}\SI_{\bC} \bB Y$$
of the adjunction $(\SI_{\bC},\OM_{\bC})$.
\end{prop}
\begin{proof}  Recalling \autoref{ass3}(i), for $X_*\in s\sS$ we  let $\ga_{\bC}\colon \bT{\OM_{\bC}}_* X_* \rtarr \OM_{\bC} \bT X_*$ denote $\ga$ regarded as a map of  $\bC$-algebras and we let $\tilde{\ga}_{\bC}\colon \SI_{\bC} \bT {\OM_{\bC}}_*X_* \rtarr \bT X_*$ denote its adjoint.  A small  diagram chase from Assumptions \ref{ass2}(iv) and   \ref{ass3}(i) shows that $\tilde{\ga}_{\bC}$ is the composite displayed in the middle square of the following diagram, in which the vertical isomorphisms are given by a natural isomorphism displayed in \autoref{ass2}(iv) and by  \autoref{formalBPQ}.   The isomorphism at the top right is the isomorphism  $\bE Y \iso \SI_{\bC} \overline{Y}$ given by \autoref{oldkey}, and we use that to identify the target of $\tilde{\ga}_{\bC}$ with $\SI_{\bC} \overline{Y}$ and thus to  identify the target of $\ga_{\bC}$ with $\OM_{\bC} \SI_{\bC} \overline{Y}$.

\[ \xymatrix{
\SI_{\bC}B(\bC, \bC, Y)  \ar[r]^-{\SI_{\bC}B\al} \ar[d]_{\iso}  &  \SI_{\bC}B(\OM_{\bC}\SI, \bC, Y) \ar[r]^{\tilde{\ga}_{\bC}}    \ar[d]^{\iso} 
&  B(\SI, \bC, Y)  \ar[r]^-{\iso} &   \SI_{\bC} B(\bC,\bC, Y)  \\
B(\SI_{\bC}\bC, \bC, Y) \ar[r]_-{B(\SI_{\bC}\et_{\bC})}  & B(\SI_{\bC}\OM_{\bC}\SI_{\bC}\bC, \bC, Y) 
\ar[r]_-{B\epz_{\bC}} & B(\SI_{\bC}\bC, \bC, Y)  \ar[u]_{\iso}   \ar[r]_-{=}&  B(\SI_{\bC}\bC, \bC, Y)  \ar[u]_{\iso}   \\}
\]
The left square uses \autoref{eye2} to identify $\al$ as $\et_{\bC}$ and the right square commutes by inspection of definitions.
The bottom composite is the identity, by a triangle identity.  The rightmost vertical isomorphism is the inverse of the leftmost one, so that isomorphism conjugates the bottom composite to the top composite.  Therefore the top composite is also the identity.   Since $\et_{\bC}$ is the adjoint of the identity, this gives the conclusion.
\end{proof}
Diagrammatically, we can embed \autoref{master} in the following commutative diagram.
\begin{equation}\label{forcon2}
\xymatrix{
 X & \ar[l]_-{\ze}  \overline X  \ar[r]^-{B\al} \ar[d] _{\et_{\bC}} &   B(\OM_{\bC}\SI,\bC, X) \ar[d]^{\ga} \\
    &  \OM_{\bC} \SI_{\bC} \overline X  \ar[r]^-{\iso} &  \OM_{\bC} B(\SI,\bC, X) \\}
\end{equation}

It follows that the recognition principle in the form stated in \autoref{recprin} is just a restatement of \autoref{recprin0}.  We restate the main point to emphasize the analogy with the monadicity theorem.

\begin{thm}\label{recprin1}  For a $\bC$-algebra $Y$, $\et_{\bC}\colon \overline Y \rtarr \OM_{\bC}\SI_{\bC} \overline Y$ is a group completion.  If $Y$ is grouplike, then $\et_{\bC}$ is a weak equivalence. In particular,  $\et_{\bC}\colon \overline{\OM_{\bC} Z}  \rtarr  \OM_{\bC}\SI_{\bC} \overline{\OM_{\bC} Z}$ is a weak equivalence for any  $Z\in \sS$. 
\end{thm}

Parenthetically, note that we can take $Y=\bC X$ for $X\in\sT$ in \autoref{master}.  In this case, we find that the approximation theorem can be seen as a special case of the recognition principle.  This can be viewed as  a generalized version of the Barratt--Priddy--Quillen theorem in topology that identifies the space  $QS^0 = \OM^{\infty}\SI^{\infty}S^0$ whose homotopy groups are the stable homotopy groups of spheres as a group completion of $\bC S^0$ for an $E_{\infty}$ operad $\bC$.  Up to homotopy, $\bC S^0$ is the disjoint union of a basepoint and the classifying spaces $B{\SI_n}$ of symmetric groups for $n\geq 1$.

\begin{thm}\label{BPQthm}[Barratt--Priddy--Quillen]  For $X\in \sT$, the following diagram of maps of $\bC$-algebras commutes.
\begin{equation}\label{master2}
\xymatrix@1{  \overline{\bC X} \ar[d]_{\ze} \ar[r]^-{B\al} &   B(\OM_{\bC}\SI,\bC, \bC X) \ar[r]^-{\ga_{\bC}} \ar[d]^{\ze} &  \OM_{\bC} \bE \bC  X  \ar[d]^-{\OM_{\bC}\ze}\\
\bC X \ar[r]_{\al} & \OM_{\bC} \SI X  \ar[r]_{=}   &  \OM_{\bC} \SI X \\}
\end{equation}
The vertical arrows are given by Lemmas \ref{extra} and \ref{extra2} and are weak equivalences.
\end{thm}

\subsection{$\OM$-connectivity and the homotopical monadicity theorem}\label{HMC}\label{pfadjequiv}
As discussed in \cite{KMZ0}, we regard \autoref{adjequiv} as a homotopical monadicity theorem.  We first recall the definition and some results about $\OM$-connectivity from Section 2.3 of \cite{KMZ0}.   These have nothing to do with the monad $\bC$.

\begin{defn} \label{LAcon} For $Z\in \GA[\sS]$,  \autoref{triv} gives a natural map 
$$  \xi\colon  B(\SI, \GA, \OM Z) \rtarr Z $$ 
in $\sS$. We say that $Z$ is $\OM$-connective if $\xi$ is a weak equivalence.  We let $\sS_{c}$ denote the full subcategory of $\OM$-connective objects in $\sS$.
\end{defn}

\begin{lem}\label{LAcon2} For any object $X$ of $\sT$, $\SI X$ is an $\OM$-connective object of $\sS$.  
\end{lem}

\begin{lem}\label{LAcon3}  Realization on $\sS$ takes levelwise $\OM$-connective objects to $\OM$-connective objects.  Thus, by \autoref{LAcon2},  $B(\SI, \GA, \OM Z)$ is itself $\OM$-connective for all $Z\in \sS$. 
\end{lem}

\begin{prop}\label{LAcon4} If $Z$ and $Z'$ are $\OM$-connective, then  a map $f\colon Z \rtarr Z'$ in $\sS$ is a weak equivalence if and only if 
$\OM f\colon \OM Z \rtarr \OM Z'$ is a weak equivalence in $\sT$. 
\end{prop}

\begin{lem}\label{LAcon5}   $Z$ is $\OM$-connective if and only if
 $$\ga\colon B(\GA,\GA,\OM Z)  \rtarr   \OM B(\SI,\GA,\OM Z)$$
 is a weak equivalence.  
 \end{lem}
 This shows the necessity of an assumption concerning $\OM$-connectivity in \autoref{ass2}(v) when $\bC = \GA$.  Now return to a $\bC$ that satisfies \autoref{ass3}. Lemmas  \ref{LAcon2} and \ref{LAcon3} and application of $B(\id, \al, \id)$ imply the following result.

\begin{lem}\label{Econ} If $Y$ is grouplike, then $\bE Y = B(\SI,\bC,Y)$ is $\OM$-connective.
\end{lem}

With these preliminaries on the target category of \autoref{adjequiv}, we adopt the following definition of its source category.

\begin{defn} Let $\bC[\sT]_{gr}\subset \bC[\sT]$  be the full subcategory of grouplike $\bC$-algebras.  \end{defn}

We need the following $\bC$-variant of the map $\xi$ in \autoref{LAcon}.  We view it as a derived variant of $\epz_{\bC}$ and use it in a kind of triangle identity to go from the recognition principle to the homotopical monadicity theorem.

\begin{prop}\label{NewAssE}   For  $Z\in \sS$, there is a natural map 
$$\xi\colon\bE \OM_{\bC} Z  = B(\SI,\bC, \OM_{\bC} Z)  \iso   \SI_{\bC} \overline{\OM_{\bC}Z}  \rtarr Z$$ 
such that the weak equivalence $\ze\colon \overline{\OM_{\bC}Z} \rtarr \OM_{\bC}Z$  factors as the composite
\begin{equation}\label{factor}
\xymatrix@1{
\overline{\OM_{\bC}Z} \ar[rr]^-{\et_{\bC}}  & & \OM_{\bC} \SI_{\bC} \overline{\OM_{\bC}Z} 
\ar[rr]^-{\OM_{\bC} \xi} & & \OM_{\bC}Z. \\}
\end{equation}
\end{prop}
\begin{proof}  We apply \autoref{triv} to  the simplicial object  $B_*(\SI,\bC,\OM_{\bC} Z)$ and  the counit $\epz\colon \SI \OM Z \rtarr Z$ of  the adjunction $(\SI,\OM)$ to obtain the promised map 
$$ \xi\colon \bE \OM_{\bC} Z = B(\SI,\bC,\OM_{\bC} Z)\rtarr Z.$$
\autoref{triv} applies since  an easy formal argument as in  \cite[Examples 9.5]{MayGeo} shows that $\epz$ coequalizes the pair
\[  \xymatrix@1{
\SI {\bC}  \OM Z  \ar@<.7ex>[rr]^-{\be}  \ar@<-.7ex>[rr]_-{\SI\vartheta}  &&  \SI   \OM Z. \\}  
\]  
\autoref{triv} also gives the diagonal arrow in the following commutative diagram of maps of $\bC$-algebras.
\[ \xymatrix{
B(\bC,\bC, \OM_{\bC} Z) \ar[d]_{\ze}   \ar[rr]^-{B\al} & & B(\OM_{\bC}\SI,\bC, \OM_{\bC} Z) \ar[d]^{\ga}  \ar[dll] \\
 \OM_{\bC} Z  & & \OM_{\bC}  B(\SI, \bC, \OM_{\bC} Z) \ar[ll]^-{\OM_{\bC} \xi}  \\}
 \]
By \autoref{Yeah}, this gives the stated factorization of $\ze$.  
\end{proof}

The following slightly more precise version of \autoref{adjequiv} follows directly.  

\begin{thm}\label{equiv}   The functors $\bE \colon \bC[\sT]_{gr} \rtarr \sS_c$ and $\OM_{\bC}\colon \sS_c \rtarr \bC[\sT]_{gr}$ induce inverse equivalences on passage to homotopy categories.  Thus every $\OM$-connective object of $\sS$ is weakly equivalent to $\SI_{\bC}\overline{Y}$ for some 
grouplike $\bC$-algebra $Y$.
\end{thm}
\begin{proof}
For $Y \in  \bC[\sT]_{gr}$,  the recognition principle gives a natural weak equivalence between $Y$ and $\OM_{\bC}\bE Y$.  For $Z\in \sS_c$, consider \autoref{factor}.  The map $\et_{\bC}$  is a weak equivalence since $\overline{\OM_{\bC} Z}$ is grouplike. Since the composite $\ze$  is also a weak equivalence,  so is $\OM_{\bC} \xi$. Since $Z$ and $B(\SI, \bC, \OM_{\bC} Z)$ are $\OM$-connective and  $\OM \xi$ is a weak equivalence, $\xi$ is a weak equivalence.
\end{proof}

\subsection{A fully conceptual reinterpretation of Theorems \ref{recprin} and \ref{adjequiv}}\label{fully}

Here we reformulate our main theorems in terms of a new conceptual notion of grouplike objects and group completions.   We would not know how to prove these results starting just from the new definitions.   We can view the notions given in Definitions \ref{screwy} and \ref{screwy2} and built into our work via \autoref{ass4} as scaffolding leading up to the entirely formal context that is given in this subsection. Here we give restatements of the main results that make no reference  to \autoref{ass4}, although that is used in their proofs.  One can imagine alternative proofs.

We first define $\bC$-grouplike objects and $\bC$-group completions.  The following definition would appear more intuitive if we replaced 
$\overline Y$ with $Y$, but we use $\overline{Y}$ to get around the fact that we do not know that the functor  $\SI_{\bC}$ preserves weak equivalences, as pointed out in \autoref{ouch}.

\begin{notn} We abbreviate notation by writing $\GA_{\bC} = \OM_{\bC}\SI_{\bC}$ in this section.  Recall that, by  \autoref{ass4},  
$\GA_{\bC}$ takes values in grouplike $\bC$-algebras.
\end{notn}

\begin{defn}\label{fgplike}  A $\bC$-algebra $Y$ is {\em $\bC$-grouplike} if  
$\et_{\bC}\colon   \overline {Y} \rtarr \GA_{\bC} \overline{Y}$ is  a weak equivalence.  A {\em $\bC$-group completion}  of a $\bC$-algebra $X$ is a map  $f\colon  X \rtarr Y$ of $\bC$-algebras such that $Y$ is $\bC$-grouplike and $\GA_{\bC} \overline{f}$ is a weak equivalence, where 
$$\overline{f} = B(\id,\id,f) \colon  B(\bC,\bC,X) \rtarr B(\bC,\bC, Y).$$
\end{defn}

The idea is illuminated by the following naturality diagram.

\begin{equation}\label{univ}
\xymatrix{
X\ar[d]_-f  & \overline{X} \ar[l]_{\ze}\ar[d]^{\overline{f}} \ar[r]^-{\et_{\bC}} & \GA_{\bC} \overline{X}  \ar[d]^{\GA_{\bC}\overline{f}}\\
Y &  \overline{Y} \ar[l]^{\ze} \ar[r]_-{\et_{\bC}} &  \GA_{\bC} \overline{Y}.\\} 
\end{equation}
Since the maps $\ze$ are equivalences, $f$ is a $\bC$-group completion if and only if $\overline{f}$ is, and we can interpret the diagram as showing that a $\bC$-group completion $f$ is weakly equivalent to $\et_{\bC}\colon \overline{X} \rtarr \GA_{\bC}\overline{X}$.

Retaining our general Assumptions, including \autoref{ass4}, we have the following comparison theorem.

\begin{thm}\label{new1}  A $\bC$-algebra $Y$ is grouplike if and only if it is $\bC$-grouplike.   A map $f\colon  X\rtarr Y$ of $\bC$-algebras is a group completion if and only if it is a $\bC$-group completion.
\end{thm}  
\begin{proof}  If $Y$ is $\bC$-grouplike, then $\et_{\bC}\colon  \overline{Y} \rtarr \GA_{\bC} \overline{Y}$  is a weak equivalence to a grouplike object, hence $\overline{Y}$ and therefore $Y$ are grouplike.  The second statement of \autoref{recprin1} says conversely that a grouplike $\bC$-algebra is $\bC$-grouplike.  

For the second statement of the theorem, consider a map $f\colon  X\rtarr Y$ of $\bC$-algebras, where $Y$ and therefore $\overline{Y}$ is grouplike.  By the first statement, 
$\overline{Y}$ is also $\bC$-grouplike.  Now consider the naturality diagram  \autoref{univ}.
The maps $\et_{\bC}$ are group completions by  \autoref{recprin1}, hence the bottom map $\et_{\bC}$ is a weak equivalence. 
If $f$ is a $\bC$-group completion, then  $\GA_{\bC}\overline{f}$ is a weak equivalence and $\overline{f}$ is a group completion by
(ii)(c) of \autoref{screwy}.   Conversely, if $f$ is a group completion, then  $\GA_{\bC}\overline{f}$ is a weak equivalence by (ii)(d) of \autoref{screwy}, and that means that $f$ is a $\bC$-group completion.
\end{proof}

\begin{cor}\label{Cgpgp} Taking $\sH$ and $\sT$ in \autoref{screwy} to both be $\bC[\sT]$, $\bC$-grouplike objects and $\bC$-group completions satisfy the conditions required of a category of grouplike objects and group completions in the category $\bC[\sT]$.
\end{cor}  
\begin{proof}  Restricting attention to $\bC[\sT]$-algebras and  their maps, the required properties of $\bC[\sT]$-grouplike objects 
and $\bC[\sT]$-group completions are inherited from those of grouplike objects and group completions in $\sH$.
\end{proof}

\autoref{Cgpgp} is conceptually satisfactory, especially if one thinks of the functor $\mathrm{Bar}$ as giving a highly structured cofibrant approximation, as in \autoref{MODEL} below,   but a direct proof without use of \autoref{new1}  seems problematic.   
In our new language, Theorems \ref{recprin} and \ref{adjequiv} can be restated as follows.

\begin{thm}\label{recprinnew}  There is a functor $\mathrm{Bar} \colon \bC[\sT] \rtarr \bC[\sT]$, written $Y\mapsto \overline{Y}$, and a natural equivalence  $\ze\colon \overline{Y} \rtarr Y$ such that the unit $\et_{\bC}\colon \overline{Y} \rtarr \OM_{\bC} \SI_{\bC} \overline{Y}$  is a $\bC$-group completion and is therefore an equivalence if $Y$ is $\bC$-grouplike.
\end{thm}

\begin{thm}\label{adjequivnew}  $(\SI_{\bC},\OM_{\bC})$ induces an adjoint equivalence from the homotopy category of $\bC$-grouplike $ \bC$-algebras in $\sT$ to the homotopy category of $\OM$-connective objects of $\sS$.
\end{thm}

While these theorems are {\em proven} using any notions of grouplike objects and group completions in $\sT$ and $\bC[\sT]$, as in Definitions \ref{screwy} and \ref{screwy2} and \autoref{ass4}, they are now {\em stated} entirely in terms of conceptual definitions within the category $\bC[\sT]$.  Moreover, our new definition of a $\bC$-group completion is characterized by a universal property in the relevant homotopy categories, as we now show.

\begin{defn}  Define $\bC[\sT]_{\bC\text{-}gp}$ to be the full subcategory of $\bC$-grouplike objects in $\bC[\sT]$.
\end{defn}

\begin{thm}\label{Univ}  Let $X$ be a $\bC$-algebra and $Y$ be a $\bC$-grouplike $\bC$-algebra.  The map
$$  \et_{\bC}^* \colon  \bC[\sT]_{\bC\text{-}gp}(\GA_{\bC} \overline{X}, Y) \rtarr \bC[\sT](\overline{X},Y)$$
induces a natural isomorphism  on passage to homotopy categories.
\end{thm}
Intuitively, the theorem says that any map of $\bC$-algebras $\overline X\rtarr Y$ factors uniquely up to homotopy through the $\bC$-group completion
$\et_{\bC}\colon \overline X \rtarr \GA_{\bC}\overline{X}$. 
\begin{proof}  In view of the equivalence $\ze\colon \overline Y \rtarr Y$, we may replace $Y$ by $\overline Y$ in the statement, and we carry out the proof with this replacement.   This allows us to use that $\et_{\bC}\colon \overline Y \rtarr \GA_{\bC}\overline{Y}$  is a weak equivalence and is therefore invertible in the homotopy category. Let us  write  $[X,Y]$ for the set of maps $X\rtarr Y$ in any given homotopy category.   

For $X\in \bC[\sT]$ and $Y\in \bC[\sT]_{\bC gp}$, we have the induced map
$$\et_{\bC}^*\colon [\GA_{\bC} \overline{X}, \overline{Y}]\rtarr [\overline{X}, \overline{Y}]$$
given by $\et_{\bC}^*([g]) = [g\com \et_{\bC}]$.   We define 
$$(\et_{\bC}^*)^{-1}\colon  [\overline{X}, \overline{Y}]\rtarr  [\GA_{\bC} \overline{X}, \overline{Y}]$$ 
by $(\et_{\bC}^*)^{-1}([f]) = [\et_{\bC}]^{-1}\com [\GA_{\bC}f]$. 
By the naturality diagram  \autoref{univ} (with $\overline f$ there replaced by $f$ here), $\GA_{\bC}f \com \et_{\bC} = \et_{\bC}\com f$.  Passing to homotopy and post-composing with $[\et_{\bC}]^{-1}$, this implies that  
$$(\et_{\bC}^*\com (\et_{\bC}^*)^{-1})([f]) = [f].$$  
Now consider the following diagram, whose top row represents $((\et_{\bC}^*)^{-1}\com\et_{\bC}^*)([g])$.
\begin{equation}\label{cute}  
\xymatrix{
\GA_{\bC} \overline X  \ar[r]^-{\GA_{\bC}\et_{\bC}}  \ar@{=}[dr]  &  \GA_{\bC} \GA_{\bC} \overline X \ar[r]^-{\GA_{\bC}g}  \ar@<-1ex>[d]_{\mu_{\bC}} 
&  \GA_{\bC} \overline Y  \ar[r]^-{[\et_{\bC}]^{-1}} & \overline Y\\
&  \GA_{\bC} \overline X \ar[r]_-{g} \ar@<-1ex>[u]_{\et_{\bC}} &  \overline Y \ar[u]_{\et_{\bC}}  \ar@{=}[ur] & \\}
\end{equation}
By the unit conditions of the monad $\GA_{\bC}$,  $\mu_{\bC}\com \GA\et_{\bC} = \id = \mu_{\bC}\com \et_{\bC}$; the middle square commutes by the naturality of $\et_{\bC}$.  Thus
$$\GA_{\bC}g\com \et_{\bC} = \et_{\bC} \com g = \et_{\bC} \com g \com \mu_{\bC}\com \et_{\bC}.$$
We claim that  $\et_{\bC}\colon \GA_{\bC}\overline{X}\rtarr\GA_{\bC}\GA_{\bC}\overline{X}$ is a weak equivalence.  Granting the claim, passing to homotopy, and precomposing with the resulting $[\et_{\bC}]^{-1}$,  it follows that 
$$[\GA_{\bC}g] =  [\et_{\bC} \com g \com \mu_{\bC}].$$
Precomposing with $[\GA_{\bC}\et_{\bC}]$ on $\GA_{\bC}\overline{X}$, using the left triangle, and post-composing with  $[\et_{\bC}]^{-1}$ on  
$\GA_{\bC}\overline{Y}$, this implies that
 $$((\et_{\bC}^*)^{-1}\com\et_{\bC}^*)([g]) = [g].$$
To prove the claim, consider the naturality square
$$  \xymatrix{  \overline{\GA_{\bC}\overline X}  \ar[r]^-{\et_{\bC}}  \ar[d]_{\ze} &  \GA_{\bC} \overline{\GA_{\bC}\overline X}  	\ar[d]^{\GA_{\bC}\ze} \\
\GA_{\bC}\overline X \ar[r]_-{\et_{\bC}}  &   \GA_{\bC} \GA_{\bC}\overline{X} \\}
$$
Since $\GA_{\bC}\overline{X}$ is grouplike, the top arrow $\et_{\bC}$ is a weak equivalence.  Since $\ze$ is a weak equivalence, one can check that  
$\GA_{\bC}\ze$ is a weak equivalence by inspection and \autoref{ass2}(iii).  Therefore the bottom arrow  $\et_{\bC}$ is a weak equivalence. 
\end{proof} 

 
 \begin{rem}  Note that the diagram \autoref{cute}  implies that $\mu_{\bC}$ is the homotopical inverse of 
 $\et_{\bC}\colon \GA_{\bC} \overline X\rtarr \GA_{\bC}\GA_{\bC}\overline{X}$.  
 \end{rem}
 
 \subsection{When is $\epz_{\bC}\colon \SI_{\bC}\OM_{\bC}Z \rtarr Z$ a weak equivalence?}\label{When} The question seeks a conceptual analog of \autoref{fgplike}, but our answer is less satisfactory.  As we explain in the next section, we believe the answer is intimately related to the question of passing from our equivalences of homotopy categories to equivalences of their associated $\infty$-categories.
 
\begin{defn}\label{fcon} An $\OM$-connective object $Z\in \sS$ is  $\bC$-{\em{connective}} if $\SI_{\bC}\OM_{\bC}Z$ is also  $\OM$-connective.
\end{defn}

Under our Assumptions, we have the following formal analog of \autoref{new1}. 

\begin{prop}\label{fcon2} Assume that $Z\in \sS$ is $\OM$-connective.  Then $Z$ is $\bC$-connective 
 if and only if  $\epz_{\bC} \colon \SI_{\bC}\OM_{\bC} Z \rtarr Z$ is a weak equivalence.
\end{prop}
\begin{proof}  The triangle commutes by the definition of 
$\be_{\bC}$ and the square commutes by naturality in the following diagram. We abbreviate 
$\OM_{\bC}Z$ to $\OM Z$:
\begin{equation}\label{Ohyeah}
\xymatrix{
B(\OM_{\bC}\SI_{\bC}\bC,\bC,\OM Z) \ar[d]_{B(\OM_{\bC}\SI_{\bC}\al, \id,\id) } 
\ar[drr]^{B(\OM_{\bC}\be_{\bC},\id,\id)} \\
B(\OM_{\bC}\SI_{\bC}\OM_{\bC}\SI,\bC,\OM Z) \ar[d]_{\xi} \ar[rr]^-{B(\OM_{\bC} \epz_{\bC}, \id, \id)}
& &  B(\OM_{\bC}\SI,\bC,\OM Z) \ar[d]^{\xi}\\
\OM_{\bC}\SI_{\bC} \OM Z \ar[rr]_{\OM_{\bC}\epz_{\bC}} & & \OM Z\\}
\end{equation}
Here $B(\OM_{\bC}\be_{\bC},\id,\id)$ is an isomorphism and the  right arrow $\xi$ is a weak equivalence.  We claim that the two left vertical arrows are both weak equivalences, hence the horizontal arrows in the square are weak equivalences.  Granting the claim, consider the commutative diagram
\begin{equation}\label{Ohyeah2}
\xymatrix{
B(\SI, \bC, \OM_{\bC}\SI_{\bC}\OM Z) \ar[rr]^-{B(\id,\id,\OM_{\bC}\epz_{\bC})} \ar[d]_{\xi}
& & B(\SI, \bC, \OM Z) \ar[d]^{\xi} \\
\SI_{\bC}\OM Z \ar[rr]_-{\epz_{\bC}}& & Z.\\}
\end{equation}
Its top arrow is a weak equivalence since $\OM_{\bC}\epz_{\bC}$ is a weak equivalence and realization preserves weak equivalences.  Since its right arrow $\xi$ is a weak equivalence, its left arrow is a weak equivalence if and only if its bottom arrow is a weak equivalence, which is the desired result.

Turning to the proof of the claim, note first that since $\OM Z$ is grouplike and $\al$ is a weak equivalence on grouplike objects, both 
$$ B(\al,\id,\id) \colon \overline{\OM Z}= B(\bC,\bC,\OM Z) \rtarr B(\OM\SI,\bC,\OM Z)$$
and the top map in the evident commutative diagram 
$$\xymatrix{   
B(\SI,\bC, \OM Z) \ar[rr]^-{B(\id,\al,\id)} \ar[dr]_{\xi} & &  B(\SI, \GA, \OM Z) \ar[dl]^{\xi}\\
&  Z &\\}
$$
are weak equivalences because they are realizations of levelwise weak equivalences.  In the diagram, $Z$ is 
$\OM$-connective if and only if the left arrow $\xi$ is a weak equivalence.  As in \cite[Remark 2.21]{KMZ0}, this holds if and only if
$\GA \colon \overline{Z} \rtarr   \OM B(\SI,\bC,\OM Z) $ is a weak equivalence.  

The top left vertical arrow in \autoref{Ohyeah} is the left vertical arrow in the following commutative diagram
{\small{$$\xymatrix{
B(\OM_{\bC}\SI_{\bC}\bC,\bC,\OM Z) 
\ar[d]_{B(\OM_{\bC}\SI_{\bC}\al, \id,\id)} 
\ar[rr]^-{\ga} & &  \OM_{\bC}\SI_{\bC}B(\bC,\bC,\OM Z) \ar[d]^{\OM_{\bC}\SI_{\bC}B(\al,\id,\id)}
& & \ar[ll]_-{\eta_{\bC}}  B(\bC,\bC,\OM Z) \ar[d]^{B(\al,\id,\id)}\\
B(\OM_{\bC}\SI_{\bC}\OM_{\bC}\SI,\bC, \OM Z) 
\ar[rr]^-{\ga} & &  \OM_{\bC}\SI_{\bC}B(\OM_{\bC}\SI, \bC,\OM Z)
& & \ar[ll]_-{\eta_{\bC}}  B(\OM_{\bC}\SI,\bC,\OM  Z)\\}$$
}}
We have written $\ga$ for the composite of  $\ga\colon |\OM_*-| \rtarr \OM|-|$ and the isomorphism 
$|{\SI_{\bC}}_{*}-|\iso \SI_{\bC}|-|$.  Since $\SI_{\bC}\bC\iso \SI$ and $Z$ is assumed to be connective, both maps $\ga$ are weak equivalences. The right square is included only to circumvent the fact that we do not know that the functor $\SI_{\BC}$ preserves weak equivalences. We have proven that its top map $\et_{\bC}$ is a weak equivalence.  The lower map $\et_{\bC}$ is a weak equivalence since, up to isomorphism, it is just
$\ga\colon B(\OM\SI,\bC,\OM) \rtarr \OM B(\SI,\bC,\OM)$.   Using two out of three, we see that the top left vertical arrow in \autoref{Ohyeah}  is a weak equivalence.  

Therefore, by the naturality of $\xi$ in  \autoref{triv},  the bottom 
vertical arrow of \autoref{Ohyeah} is a weak equivalence if 
$$ \OM_{\bC}\SI_{\bC} \ze\colon  \OM_{\bC}\SI_{\bC}  B(\bC,\bC, S) \rtarr \OM_{\bC}\SI_{\bC} \OM Z $$
is a weak equivalence.  That should  be true since $\ze$ is a homotopy equivalence.  However, since 
the homotopy is the realization of a simplicial homotopy in $\sT$ and $\SI_{\bC}$ is only defined on 
$\bC[\sT]$, the conclusion is not obvious.   To get around this, we use the map 
$\bC\rtarr \OM_{\bC}\SI_{\bC} = \GA_{\bC}$ of monads given in \autoref{eye3}., which we denote by $\tilde{\al}$.   The following diagram commutes by naturality.
$$\xymatrix{
& B(\bC, \bC, \OM Z)  \ar[dl]_{\ze}  
\ar[dd]^{B(\tilde{\al}, \tilde{\al}, \id)}  \ar[rr]^-{\et_{\bC}} 
& & \GA_{\bC}B(\bC,\bC,\OM Z) \ar[dr]^{\GA_{\bC}\ze} \ar[dd]_{\GA_{\bC}B(\tilde{\al}, \tilde{\al},\id)} &   \\
\OM Z & & & & \GA_{\bC} \OM Z \\
& B(\GA_{\bC}, \GA_{\bC}, \OM Z)  \ar[ul]^{\ze} \ar[rr]_{\et_{\GA_{\bC}}} & & \GA_{\bC}B(\GA_{\bC},\GA_{\bC}, \OM Z)
\ar[ur]_{{\GA_{\bC}} \ze} \\}
$$

The arrows $\ze$ are homotopy equivalences, the top one in $\sT$ and the bottom one in $\bC[\sT]$.  The point is that all of the simplicial degeneracies, not just the faces, used in the construction of the homotopy are maps in 
$\bC[\sT]$.  Since any functor preserves homotopy equivalences, it follows that the lower map 
$\GA_{\bC}\ze$ is a homotopy equivalence.  The argument above for $\et_{\bC}$ also applies to 
$\et_{\GA_{\bC}}$, so that both the top and bottom map in the square are weak equivalences.  Therefore, by three applications of two out of three, the lower map $\GA\ze$ is a weak equivalence.  Using the weak eqivalence $\ga$ in a triangle, it follows that the lower left vertical arrow $\xi$  in \autoref{Ohyeah} is a weak equivalence, completing the proof of our claim and thus of the proposition.
 \end{proof}

\section{Model categorical and $\infty$-categorical embellishments}\label{CONTEXT3}
\subsection{Model category contexts}\label{MODEL}

In our applications, we generally have model structures on $\sT$ and $\sS$ such that $(\SI,\OM)$ is a Quillen adjunction.  Moreover, a standard model categorical result (e.g. \cite[Theorem 16.2.5]{morecon}) generally applies to show that $\bC[\sT]$ is then a model category with fibrations and weak equivalences created by the forgetful functor  $\bU_{\bC}\colon \bC[\sT]\rtarr \sT$.  Then  $(\bF_{\bC}, \bU_{\bC})$ is a Quillen adjunction. Since $\OM$ preserves fibrations and weak equivalences, it is clear that $\OM_{\bC}$ preserves fibrations and weak equivalences.  That is, the following result holds.
\begin{prop}  If $(\SI,\OM)$ is a Quillen adjunction, then $(\SI_{\bC}, \OM_{\bC})$ is a Quillen adjunction.
\end{prop}

While thinking model theoretically did not lead us to our present perspective, it does conceptualize the results.  When working with spaces or $G$-spaces and using the standard Quillen model structures, the bar construction  $\overline{Y}$ has the homotopy type of a cofibrant $\bC$-algebra under minimal hypotheses (e.g. \cite[Lemma 5.50]{ABGHR1}).  It is therefore cofibrant and thus  a cofibrant approximation of $Y$ in $\bC[\sT]$ if we instead use the mixed model structures of \cite[Section 17.3 ]{morecon}.   We can then think of the functor $\bE$ of \autoref{Edefn} as the composite of the cofibrant approximation $\mathrm{Bar}$ and the left adjoint $\SI_{\bC}$ of a Quillen adjunction.  In the language of \cite{MMSS}, the homotopical monadicity theorem is then given by something close to a connective Quillen equivalence.  That is, if we restrict the target to connective objects and the source to grouplike objects, the Quillen adjunction is restricting to an equivalence of homotopy categories.   

\begin{rem}  In the example of orbital presheaves in Part 2,  we are led to an illuminating alternative point of view.
There the interest is entirely focused on objects in a model category that are definitely  {\em not}  cofibrant and we use classical cellular cofibrant approximation as a tool to convert such examples to well-behaved cofibrant objects that feed naturally into our machine.   Nevertheless, it is convenient and perhaps necessary there  to continue using the bar construction, even though it would be pointless to consider it as a cofibrant approximation.
\end{rem}

\subsection{From relative categories to $\infty$-categories}

The passage to $\infty$-categories fits more generally and perhaps more naturally into the language of equivalences of relative categories.  We recall here  how a 1980 result of Dwyer and Kan \cite{DK2} implies that many equivalences of homotopy categories directly imply equivalences of associated $\infty$-categories.  The original equivalences then amount to concrete classical category level sharpenings of equivalences of $\infty$-categories.  

Aside from modern contextualization, due to Joyal, Lurie, Rezk, Bergner and others \cite{Joyal, LurieHA, Rezk, Bergner}, the ideas and work to show this are due to Dwyer and Kan \cite{DK1, DK2}. We are just advertising and trying to apply their results, in part following a modern exposition of Bayliss \cite{Bay}.  As so often in the literature, we ignore size issues, relying on passage to larger universes to rectify the neglect. Some treatment is given in \cite{DK3}.

A relative category $(\sC,\sW)$ is a category $\sC$ together with a subcategory $\sW$ with the same objects and their identity maps.  A relative category is called a category with weak equivalences if $\sW$ contains all isomorphisms and satisfies the two-out-of-three property.  We are only interested here in those 
relative categories which are categories with weak equivalences.  Thus from here on relative categories mean categories with weak equivalences.  A morphism $\bF\colon (\sC,\sW)\rtarr (\sD,\sV)$ of relative categories is a functor  $\bF\colon \sC \rtarr \sD$ that restricts to a functor  $\sW$ to $\sV$.  
 
 \begin{defn} An adjoint  weak equivalence between relative categories $(\sC,\sW)$ and $(\sD,\sV)$ is an adjoint pair $\bF\dashv \bG$, such that the unit $\et\colon X \rtarr \bG\bF X$ is in $\sW$ and the counit $\epz\colon \bF\bG Y \rtarr Y$ is in $\sV$ for all $X\in \sC$ and $Y\in \sD$. 
 \end{defn}
 
 The following result is \cite[Corollary 3.6]{DK2}. 
 
 \begin{thm}\label{DKThm}
 An adjoint weak equivalence between relative categories $(\sC,\sW)$ and $(\sD,\sV)$ induces a weak equivalence between their hammock localizations $L^H(\sC,\sW)$ and $L^H(\sD,\sV)$.
 \end{thm}
 
The hammock localization is defined in \cite[2.1]{DK2}.   By a weak equivalence of hammock localizations we mean exactly what Dwyer and Kan prove: the induced functors
 $$ \pi_0 L^H(\sC,\sW) \rtarr \pi_0L^H(\sD,\sV) \ \ \text{and} \ \  \pi_0 L^H(\sD,\sV) \rtarr \pi_0L^H(\sC,\sW) $$
 give an equivalence of categories and the induced maps of simplicial sets
$$ L^H(\sC,\sW) \rtarr L^H(\sD,\sV) \ \ \text{and} \ \  L^H(\sD,\sV) \rtarr L^H(\sC,\sW) $$
are weak homotopy equivalences.

 The modern reinterpretation is as follows.
 
 \begin{thm}\label{infversion}  The hammock localization of $(\sC,\sW)$ is a model for the $\infty$-category associated to $(\sC,\sW)$, and a weak equivalence between hammock localizations can be interpreted as an equivalence between such $\infty$-categories.
 \end{thm}
The equivalence of $\pi_0$ categories is an equivalence of (ordinary) homotopy categories, and the weak equivalence of hammock localizations is a lifting to an equivalence of $\infty$-categories. 
  
A brief leisurely exposition of the definitions and results that make sense of \autoref{infversion} is given by Bayliss \cite{Bay}, an expository paper written for the 2025 Chicago REU. 
We are using the model of $\infty$-categories given by (enriched) simplicial categories.  It is equivalent to the quasicategory model by the coherent nerve functor, with adjoint the path functor. The hammock construction promotes a relative category to a simplicial category, and every simplicial category is weakly equivalent to the simplicial localization of a relative category \cite[2.5]{DK4}. The homotopy category $\pi_0L^H(C,W)$ is equivalent to the classical homotopy category, namely the localization $\sC[\sW^{-1}]$ \cite[3.1]{DK4}.   The weak equivalences of simplicial hom sets lift the equivalence of homotopy categories to a Dwyer Kan equivalence of simplicial categories, which is a model for an equivalence of $\infty$-categories.

If $\SI\dashv \OM$ is a Quillen equivalence of model categories, Theorems \ref{DKThm} and \ref{infversion} imply that $\SI\dashv \OM$ induces an equivalence of $\infty$-categories.    The model structures are not needed for the conclusion,  as the cited theorems show.  

\subsection{Homotopical adjunctions}
What we are actually seeing in \autoref{adjequivnew} is close to this, but a little different.  Some elementary categorical observations contextualize the framework.
Briefly, for any adjoint pair of functors $(\XI, \PS)$ between categories $\sC$ and $\sD$ and any functor  $\bF \colon \sC \rtarr \sC$, we define an induced ``twisted adjunction"
$(\XI \bF,\PS)$.  The example to keep in mind is $(\XI, \PS)=(\SI_{\bC},\OM_{\bC})$ and $\bF=\mathrm{Bar}$.
Passage to opposite categories gives an analogous twisted adjunction  $(\XI,\PS \bG)$ for any functor $\bG\colon \sD \rtarr \sD$, but that will not concern us.   

We describe  $(\XI\bF, \PS)$  before returning to homotopy and the monadicity theorems.  It is just the natural isomorphism
\begin{equation}\label{twist3}
\sD(\XI  \bF  Y, Z) \iso \sC(\bF Y,\Psi Z)
\end{equation} 
given by $(\XI,\PS)$.   Taking the identity map of  $Z = \XI \bF Y$, we obtain the unit map $\et\colon \bF Y \rtarr \Psi\XI \bF Y$.  Using that and the counit map
$\epz\colon \XI\PS Z \rtarr Z$,  we obtain the induced inverse isomorphisms
$$\xymatrix{
\al_{\bF } \colon  \sD(\XI \bF Y, Z) \ar[r]^-{\Psi} & \sC(\Psi \XI\bF Y, \Psi Z) \ar[r]^-{\et^*} & \sC(\bF Y, \Psi Z)\\}$$
 and
$$ \xymatrix{
    \be_{\bF  } \colon \sC(\bF Y, \Psi Z) \ar[r]^-{\XI} & \sD(\XI \bF Y, \XI\Psi Z) \ar[r]^-{\epz_*} & \sD(\XI \bF  Y, Z).\\}$$
Writing $\al = \al_{\id}$ and $\be = \be_{\id}$,  the composites $\al\com\be$ and $\be\com \al$ are also the identity.  Now suppose further that we have a natural transformation $\ze\colon \bF  \rtarr \id$.  Then we have natural transformations
$$ (\XI\ze)^* \colon \sD(\XI Y, Z) \rtarr \sD(\XI \bF Y, Z)  \ \ \text{and}  \ \ \ze^* \colon \sC(Y, \Psi Z) \rtarr \sC(\bF Y, \Psi Z). $$
Naturality shows that the four trapezoids in the following two diagrams commute.   
The composites in their inner and outer squares, starting at the top left, are both identity maps since they are $\be\al$ and $\be_{\bF}\al_{\bF}$ in \autoref{NatOne}  and
$\al\be$ and $\al_{\bF}\be_{\bF}$ in \autoref{NatTwo}. 

\begin{equation}\label{NatOne}
\xymatrix{
\sD(\XI \bF Y, Z) \ar[rrr]^-{\Psi} & & & \sC(\Psi \XI\bF Y, \Psi Z) \ar[ddd]^{\et^*}  \\
&  \sD(\XI Y, Z) \ar[r]^-{\Psi} \ar[ul]_-{(\XI \ze)^*} & \sC(\Psi \XI Y, \Psi Z) \ar[ur]^-{(\PS \XI \ze)^*} \ar[d]^{\et^*}  & \\
&\sD(\XI  Y, \XI\Psi Z)  \ar[u]^{\epz_*} \ar[dl]_{(\XI \ze)^*}& \sC(Y, \Psi Z)\ar[dr]^{\ze^*} \ar[l]_{\XI}  & \\
\sD(\XI \bF Y, \XI \Psi Z) \ar[uuu]^-{\epz_*} & & &  \sC(\bF Y, \Psi Z) \ar[lll]^{\XI} \\}
\end{equation}

\begin{equation}\label{NatTwo}
\xymatrix{
\sC(\bF Y, \PS Z) \ar[rrr]^-{\XI} & & & \sD( \XI\bF Y, \XI\PS Z) \ar[ddd]^{\epz_*}  \\
&  \sC(Y, \PS Z) \ar[r]^-{\XI} \ar[ul]_-{ \ze^*} & \sD(\XI Y, \XI\Psi Z) \ar[ur]^-{(\XI \ze)^*} \ar[d]^{\epz_*}  & \\
&\sC(\PS \XI Y, \Psi Z)  \ar[u]_{\et^*} \ar[dl]_{(\PS\XI\ze)^*}& \sD( \XI Y,  Z)\ar[dr]^{(\XI\ze)^*} \ar[l]_{\PS}  & \\
\sC(\PS \XI \bF Y, \Psi Z) \ar[uuu]^-{\et^*} & & &  \sD(\XI \bF Y, Z) \ar[lll]^{\PS} \\}
\end{equation}  

For homotopical contexts, we adopt the following definition.

\begin{defn} A homotopical adjunction is an adjunction $\XI \dashv \PS$ twisted by a functor $\bF$, as above, in which $\ze\colon \bF  Y \rtarr Y$ is a weak equivalence for all $Y\in \sC$ and the ``twisted units" $\et: \bF Y \to \PS\XI \bF Y$ and the counits $\epz: \XI \PS Z \to Z$ are weak equivalences upon restriction to full subcategories $\sC_g$ (perhaps grouplike objects) and $\sD_c$ (perhaps  connective objects) of $\sC$ and $\sD$. 
\end{defn} 

For example, $\XI\dashv \Psi$ might be a Quillen equivalence of model categories and $\bF$ might be a functorial cofibrant approximation.  Ignoring possible model structures, the $\ze$ in our monadical context will always be a natural weak equivalence.  We claim that the proof of Theorem \ref{DKThm}  adapts to this context of homotopical adjunctions to prove the following conjecture.

 \begin{conj}\label{DKThm?}
 A homotopical adjunction between relative categories $(\sC,\sW)$ and $(\sD,\sV)$ induces a weak equivalence between the hammock localizations \linebreak
 $L^H(\sC_g,\sW\cap \sC_g)$ and $L^H(\sD_c,\sV\cap \sD_c)$. 
 \end{conj}

A proof requires careful inspection, but we believe that this is true.  The intuition should be clear from the construction of  hammock localizations: we expect  that one can use $\ze$ to enlarge hammocks to equivalent ones.  By \autoref{infversion}, the conjecture implies that we have induced equivalences of $\infty$-categories between the $\infty$-categories associated to  $L^H(\sC_g,\sW\cap \sC_g)$ and $L^H(\sD_c,\sV\cap \sD_c)$.  

\subsection{The monadic homotopical adjunction}

The following result is a direct consequence of \autoref{fcon2}.

\begin{thm}\label{NotYet}   The adjunction $\SI_{\bC}\dashv\OM_{\bC}$ from $\bC[\sT]$ to $\sS$ and the functor $\mathrm{Bar}\colon \bC[\sT] \rtarr \bC[\sT]$ together with the weak equivalence 
$\ze\colon \mathrm{Bar} \rtarr \id$ specify a homotopical adjunction
from the relative subcategory $\bC[\sT]_{gp}$ of grouplike objects of $\bC[\sT]$ to the relative subcategory $\sS_{\bC c}$ of $\bC$-connective objects of $\sS$
\end{thm} 
\begin{proof}
We have the adjunction and the weak equivalences $\ze$.  The equivalence $\et_{\bC}\colon \overline{Y} \rtarr \OM_{\bC}\SI_{\bC}\overline{Y}$ is given  by \autoref{recprinnew}. \autoref{fcon2} shows that 
$\bC$-connectivity is a necessary and sufficient  condition for $\epz_{\bC}\colon  \SI_{\bC}\OM_{\bC}Z \rtarr Z$ to be a weak equivalence when $Z$ is $\OM$-connective.  
\end{proof}

Assuming \autoref{DKThm?} and applying \autoref{infversion}, this has the following consequence.  

\begin{cor}\label{YES?} The $\infty$-categories associated to the relative categories $\bC[\sT]_{gp}$ and $\sS_{\bC c}$ and their respective weak equivalences are equivalent.
\end{cor}

Of course, this is unsatisfactory since we do not yet have any results about the following question.

\begin{quest}\label{quest} When is it true that an $\OM$-connective object of $\sS$ is $\bC$-connective?
\end{quest}

\section{The classical theory of spaces and $G$--spaces}\label{SPACEG}
Adding some results not in the earlier literature, we record how classical and equivariant iterated and infinite loop space theory appear in our context.

\subsection{Delooping based spaces}\label{SPACES}  This is the original example and all proofs are in place, as we have sketched.  We focus on what is new in the interpretation of our Assumptions.

Here $\sT$ and $\sS$ are both taken to be the category of based spaces and $(\SI,\OM)$ in the general theory is taken to be $(\SI^n,\OM^n)$, the $n$-fold suspension and $n$-fold loop space functors.  We take  $\bC=\bC_n$  to be the monad on based spaces associated to an $E_n$-operad $\sC = \sC_n$ that acts naturally  on $\OM^n X$ for spaces $X$, such as the little $n$-cubes operad, the little $n$-discs operad, or the Steiner operad for $\bR^n$ or the product of one of these with any $E_{\infty}$-operad.  As in Remarks \ref{bspt1} and \ref{bspt2}, to ensure that these monads behave well homotopically, we implicitly restrict attention to nondegenerately based spaces when applying $\SI$ or $\bC$.  As said before, basepoint issues are handled carefully in \cite{MMO} and we shall say no more about them here.  In particular, we have Reedy cofibrancy wherever needed.

The following observation relates the conceptual notion of an $\OM^n$-connective space from \autoref{LAcon} to the classical notion of an $n$-connective, alias $(n-1)$-connected, space.

\begin{prop}\label{n=n} A space $X$ is $\OM^n$-connective if and only if it is $n$-connective.
\end{prop}
\begin{proof}
If $X$ is $\OM^n$-connective, then it is weakly equivalent to $B(\SI^n,\OM^n\SI^n, \OM^n X)$.  The homotopy groups of the latter are zero in dimensions less than $n$, so $X$ must be $n$-connective.  Conversely, if $X$ is $n$-connective, then inductive comparison of path space fibrations before and after realization, as displayed in \cite[(2.4)]{KMZ0} for $n=1$, shows that 
$$\ga^n\colon B(\OM^n\SI^n, \OM^n\SI^n, \OM^n X) \rtarr \OM^nB(\SI^n,\OM^n\SI^n,\OM^n X) $$
is a weak equivalence.  By \autoref{LAcon5}, that implies that $X$ is $\OM^n$-connective. 
\end{proof}

We clearly have a context as specified in \autoref{ass1}. \autoref{ass2} holds under the basepoint restrictions mentioned above, as is discussed in \cite{KMZ0}. \autoref{ass3} holds as discussed in \autoref{Csec}.  With any of the standard model structures on $\sT$ (e.g. \cite[Chapter 17]{morecon}), $(\SI^n,\OM^n)$ is a Quillen adjunction since $\OM^n$ clearly preserves fibrations and acyclic fibrations.   It remains to discuss  Assumptions \ref{ass4} and \ref{ass5}.

We start with the following classical definitions, modified very slightly to fit \autoref{screwy}.  This gives the context in which the approximation theorem is proven for $n\geq 2$.  Let $\mathrm{Gr}$ denote the Grothendieck group functor (alias group completion) from abelian monoids to abelian groups and let $i\colon \id \rtarr \mathrm{Gr}$ denote the natural map.   If $f\colon M \rtarr A$ is a homomorphism from an abelian monoid to an abelian group, we write $\tilde{f}$ for the map of groups $\mathrm{Gr} M \rtarr A$ such that $\tilde{f}\com i = f$.  We think of the following definition as giving a homologically defined category $\sH_{com}$ that satisfies \autoref{screwy}.  It is an elaboration of \cite[Definition 1.6]{GM3}.

\begin{defn}\label{Hopf}  Let $\sH$ be the category of homotopy associative Hopf spaces and Hopf maps.  It is understood that Hopf spaces are (nondegenerately) based at their unit objects and that Hopf maps are based.
A space in $\sH$ is said to be grouplike if  $\pi_0(X)$ is a group.    Let $\sH_{com}$ be the full subcategory of homotopy commutative Hopf spaces in $\sH$.  Define a group completion $f\colon Y\rtarr Z$ to be a map in $\sH_{com}$ such that $Z$ is grouplike, $f$ induces an isomorphism $\tilde{f}\colon \mathrm{Gr}\pi_0(Y) \rtarr \pi_0(Z)$, and, for any field $k$ of coefficients, $f_*$ localizes the homology of $Y$  by inverting $\pi_0(Y)\subset H_*(Y)$.  That is, $f_*$ induces an isomorphism
\[  \xymatrix@1{  \tilde{f}\colon H_*(Y)[\pi_0(Y)^{-1}] \ar[r]^-{\iso} & H_*(Z).\\}  \]  
It suffices to take $k$ to be $\bK$ and $\bF_p$ for any prime $p$. 
\end{defn}

The definition just given is designed for use when $n\geq 2$, so that the relevant Hopf spaces are homotopy commutative.  We focus on the case $n=1$ in \autoref{Moore}.

\begin{lem}\label{square2}  Using the notion of group completion given in \autoref{Hopf},  the category $\sH_{com}$ satisfies the conditions specified in \autoref{screwy}.
\end{lem}
\begin{proof}  This is easy, but we give some details in the hopes of making  \autoref{screwy} more intuitively sensible than it might have seemed on a first reading.  It is obvious that an object of $\sH$ that is weakly equivalent to a grouplike object is grouplike.
It is classical that an object  $X\in \sH_{com,gp}$ is weakly equivalent in $\sH$ to $X_0\times \pi_0(X)$, where $X_0$ is the component of the identity element $0$ of $\pi_0(X)$.  Recalling that Hopf spaces are simple, it follows from a theorem of Whitehead that a map between objects of $\sH_{com,gp}$ is a weak equivalence if and only if it induces isomorphisms on  $\pi_0$ and on $H_*$ for all fields of coefficients or, equivalently, for integer coefficients.   Now (i) of \autoref{screwy} is clear.  For (ii), looking back at \autoref{square}, we first assume that $Z$ and $W$ are grouplike.  On passage to $\pi_0$ or to $H_*$, the vertical composites in the following diagrams are either $g_*$ or $h_*$.
\[  \xymatrix{
\pi_0(X) \ar[r]^-{f_*} \ar[d]_{i} & \pi_0(Y) \ar[d]^{i}\\
\mathrm{Gr}\pi_0(X) \ar[r]^-{\tilde{f}_*} \ar[d]_{\tilde{g}_*}& \mathrm{Gr}\pi_0(Y)\ar[d]^{\tilde{h}_*}\\
\pi_0(Z) \ar[r]_-{j_*}  & \pi_0(W) \\}
\ \ \ 
\xymatrix{
H_*(X) \ar[r]^-{f_*} \ar[d]_{i} & H_*(Y) \ar[d]^{i}\\
H_*(X)[\pi_0(X)^{-1}] \ar[r]^-{\tilde{f}_*} \ar[d]_{\tilde{g}_*} & H_*(Y)[\pi_0(Y)^{-1}] \ar[d]^{\tilde{h}_*}\\
H_*(Z) \ar[r]_-{j_*}  & H_*(W) \\}
\]
In (ii)(a), $f$ and $j$ are weak equivalences, so $f_*$, the induced map $\tilde{f}_*$, and $j_*$ are isomorphisms.  It is immediate from $2$ out of $3$ in the bottom squares that $\tilde{g}_*$ is an isomorphism if and only if $\tilde{h}_*$ is an isomorphism, which verifies that (ii)(a) is satisfied.  In (ii)(b), $g$ and $h$ are group completions and $f$ is a weak equivalence, so $2$ out of $3$ in the bottom square implies that $j_*$ is an isomorphism, so that $j$ is a weak equivalence.  Finally, we assume that $Y$, $Z$, and $W$ are grouplike in \autoref{square}.  Then the bottom squares of the diagrams above reduce to
\[  \xymatrix{
\mathrm{Gr}\pi_0(X) \ar[r]^-{\tilde{f}_*} \ar[d]_{\tilde{g}_*}& \pi_0(Y)\ar[d]^{h_*}\\
\pi_0(Z) \ar[r]_-{j_*}  & \pi_0(W) \\}
\ \ \ 
\xymatrix{
H_*(X)[(\pi_0(X)^{-1}] \ar[r]^-{\tilde{f}_*} \ar[d]_{\tilde{g}_*} & H_*(Y) \ar[d]^{h_*}\\
H_*(Z) \ar[r]_-{j_*}  & H_*(W) \\}
\]
In (ii)(c), $h_*$ and $j_*$ are isomorphisms, hence $\tilde{f}_*$ is an isomorphism if and only if $\tilde{g}_*$ is an isomorphism.   In (ii)(d),  $\tilde{f}_*$ and 
$\tilde{g}_*$ are isomorphisms, hence $h_*$ is an isomorphism if and only if $j_*$ is an isomorphism.
\end{proof}

When $n\geq 2$, it is clear that a $\bC$-algebra is a homotopy associative and commutative Hopf space, so that $\bU_{\bC}$ factors through $\sH_{com}$. To verify \autoref{ass4}, it remains to construct a group completion functor $(\bG,g)$ as in \autoref{screwy2}.  Here we follow \cite[Theorem 13.5]{MayGeo} and
\cite[Theorem 1.14(ii)]{GM3}.  Let $\sM$ be the operad\footnote{alias $\mathrm{Asso}$, the associativity monad} (with permutations, as usual) whose algebras are topological monoids
\cite[Definition 3.1]{MayGeo}. Whichever sequence of compatible operads we choose, the operad $\sC = \sC_n$ comes with an inclusion of the operad $\sC_1$. The operad $\sC_1$ is an $A_{\infty}$ operad as defined in  \cite[Definition 3.5]{MayGeo}.  Moreover, there is a map of operads  $\de\colon \sC_1 \rtarr \sM$ such that the $\SI_j$-map  
$\de(j)\colon \sC_1(j)\rtarr \sM(j)$ is a $\SI_j$-equivalence for each $j$, so that the induced map of monads $\de\colon \bC_1\rtarr \bM$ is given by weak equivalences.  The following result is implied by \autoref{extra} and \autoref{ass2}.
 
\begin{lem}\label{Cone}  Let $Y$ be a $\bC_1$-algebra.  Then
\begin{enumerate}[(i)] 
\item  $\ze\colon B(\bC_1,\bC_1,Y) \rtarr Y$ is a map of $\sC_1$-algebras and a homotopy equivalence with inverse $\nu$.
\item  $B\de\colon  B(\bC_1, \bC_1,Y) \rtarr B(\bM,\bC_1,Y)$  is a weak equivalence.
\end{enumerate}
\end{lem}

\begin{notn}\label{class}
Recall that the classifying space $BM$ of a topological monoid $M$ is the geometric realization of the nerve of $M$, where $M$ is viewed as the morphism space of a topological category with a single object.  The inclusion of $1$-simplices gives a natural map $\io\colon \SI M \rtarr BM$. Its adjoint is a natural map ${\chi}\colon M \rtarr \OM BM$.
\end{notn}  

\begin{defn}\label{HopfbG}  Let $Y$ be a $\sC= \sC_n$-algebra, and regard $Y$ as a $\sC_1$-algebra by pullback along  $\sC_1\rtarr \sC_n$.  Define $\bF Y$ to be the topological monoid  
$B(\bM,\bC_1,Y)$ and define $\bG Y = \OM B \bF Y$.  Note that $\nu$ in \autoref{extra} is a Hopf map, and define $g\colon Y \rtarr \bG Y$ to be the natural composite of Hopf maps 
$$  \xymatrix@1{g\colon Y \ar[r]^-{\nu}  & B(\bC_1,\bC_1,Y) \ar[r]^-{B\de} &  B(\bM,\bC_1,Y) = \bF Y \ar[r]^-{\ze}  &\bG Y\\}$$
\end{defn}

We recall a version of the Barratt--Quillen group completion theorem.
\begin{thm}\label{gpcompthm}  If $M$ is a homotopy commutative topological monoid, then the map $\chi\colon M \rtarr \OM BM$  is a group completion.  If, further, $M$ is grouplike, then $\chi$ is a weak equivalence. Therefore, for $Y\in \bC[\sT]$, $g\colon Y \rtarr \bG Y$ is a group completion and is a weak equivalence if $Y$ is grouplike.
\end{thm}
\begin{proof}  The version in \cite[Theorem 15.1]{MayClass} assumes that left translation by an element $x$ is homotopic to right translation by $x$  in both $M$ and $\OM BM$.  Clearly that holds for $M$ when $M$ is homotopy commutative.  For $\OM BM$, that hypothesis is only used to ensure that $\pi_0(\OM BM) = \pi_1(BM)$ is commutative (see \cite[Lemma 15.2]{MayClass}).  This holds when $M$ is homotopy commutative since $\pi_0(\OM BM)$ is then the Grothendieck group of $\pi_0(M)$, by \cite[Proposition B.1]{Ramras}.\footnote{We thank Dan  Ramras for this reference.} The weak equivalence consequence is immediate from a comparison of fiber sequences that is given in \cite[Theorem 7.6]{MayClass}.
\end{proof}

This applies when $n\geq 2$ to complete the verification of \autoref{ass4}.  To deal with the case $n=1$, 
we adopt the following variant of \autoref{Hopf}, which is a variant of \cite[Definition 1.7]{GM3}. 

\begin{defn}\label{Hopf2}  
Let $\sH_m$ be the subcategory of $\sH$ whose objects are weakly equivalent to topological monoids and whose maps are weakly equivalent in the arrow category of $\sH$ to maps of monoids.  A space in $\sH_m$ is said to be grouplike if $\pi_0(X)$ is a group.    A map $f$ in 
$\sH_m$ is an  $m$-group completion if it is weakly equivalent in the arrow category of $\sH_m$ to a map $\chi\colon M\rtarr \OM B M$ for a topological monoid $M$.  The definition makes sense since, for a topological monoid $M$, $\OM B M$ is homotopy equivalent to the Moore loop space $\LA BM$, which is a topological monoid.
\end{defn}

We leave it as an exercise to prove the following analog of \autoref{square2}.

\begin{lem}\label{square3}   Using the notion of $m$-group completion given in \autoref{Hopf2},  the category $\sH_m$ satisfies the conditions specified in \autoref{screwy} when $n \geq 1$.
\end{lem}

\begin{rem} We could alternatively use the more rigid conceptual notions of  $\bC_1$-grouplike $\bC_1$-algebras and $\bC_1$-group completions that are given in \autoref{fgplike}, together with \autoref{Cgpgp}.
\end{rem}

Using $m$-group completions and \autoref{HopfbG}, \autoref{ass4} is now immediate for $n\geq 1$.  In view of the following result, the $m$-notion is mainly of interest when $n=1$ since it can be used interchangeably with the original notion of \autoref{Hopf} when $n\geq 2$. Let  $\sH_{m,com}$ be the full subcategory of homotopy commutative Hopf spaces in $\sH_m$.

\begin{lem}\label{weakstrong}  Let $f\colon Y\rtarr Z$ be a map in $\sH_{m,com}$. If $f$ is an $m$-group completion, then $f$ is a  group completion. Conversely, if $f$ is also a map of $\bC$-algebras (in particular, if $n\geq 2$) and is a group completion, then $f$ is an $m$-group completion.
\end{lem}
\begin{proof}  If $f$ is an $m$-group completion, then it is a group completion by \autoref{gpcompthm}.  For the converse, $f$ is equivalent in $\sH_{com}$ to $\bF Y \rtarr \bG Y = \OM B\bF Y$.
\end{proof}

It remains to consider the approximation theorem, \autoref{ass5}.   When $X$ is connected, it states that $\al =\al_n \colon \bC_n X\rtarr \OM^n\SI^n X$ is a weak equivalence and is proven in \cite[Theorem 6.1]{MayGeo} for all $n\geq 1$.  The proof goes  by induction on $n$, using a comparison of fiber sequences
\[ \xymatrix{ 
\bC_n X \ar[r] \ar[d]_{\al_n} & \bE_n X \ar[r]  \ar[d] &  \bC_{n-1} \SI X \ar[d]^{\al_{n-1}}\\
 \OM^n\SI^n X \ar[r] & P\OM^{n-1} \SI^n X \ar[r] &  \OM^{n-1} \SI^n X,\\} \]
taking $\al_0$ to be the identity functor on $\SI X$ to start the induction.   The work is in the construction of the contractible space $\bE_n X$ giving the quasi-fibration displayed in the top row.  When $n=1$, this weak equivalence is also implied by James' 1954 paper \cite{James} that introduced the James construction, which, by historical misuse, we have denoted $\bM$.   

For general spaces $X$, there is a proof for $n\geq 2$ by direct computation of mod $p$ homology for all primes $p$, with their Bockstein spectral sequences, by Fred Cohen \cite[Theorems 3.1, 3.2, and 3.12 and Corollary 3.3]{CLM}\footnote{Some small corrections are given by Wellington 
\cite[p. 10]{Well}}; the corresponding calculation with rational coefficients is easy and is essentially implied by the Bockstein spectral sequence calculations.  There is also a geometric proof by Graeme Segal \cite{Seg0}.   It proceeds by an induction analogous to that for the connected case.  For the case $n=1$, we shall explain a proof of the following version in \autoref{Moore}.  

\begin{thm}\label{approxn=1} For a $\sC_1$-space $Y$, $\al_1\colon \bC_1 Y \rtarr \OM\SI Y$ is an $m$-group completion. 
\end{thm} 

The following theorem is now an immediate specialization of our general results.  It is essentially a restatement of results first proven in \cite{MayGeo, MayPerm} and generalized equivariantly in \cite{GM3}.

\begin{thm}\label{space}  Take $\sC = \sC_n$.  For $\sC$-spaces $Y$, the unit $\et_{\bC}$ induces a group completion if $n\geq 2$ (an $m$-group completion if $n\geq1$) and therefore a $\bC$-group completion
$$Y \rtarr \OM^n_{\bC}\bE Y\iso \OM^n_{\bC}  \SI^n_{\bC} \bB Y.$$
Moreover, $n$-connective based spaces $Z$ are naturally equivalent to coequalized $n$-fold suspensions  $\SI^n_{\bC}  \overline{\OM^n_{\bC} Z}$.
\end{thm}

\subsection{Moore loop spaces and $1$-fold loop spaces}\label{Moore}
It is classical that the $1$-fold loop space $\OM X$ is equivalent to the Moore loop space $\LA X$, which is a topological monoid, and that, for connected $X$, $\OM\SI X$ is equivalent to the James construction on $X$, which we have denoted $\bM X$. Largely following Thomason \cite{Thom} and, more closely, Fiedorowicz \cite{Fied}, we put this into a conceptual framework that elucidates the case $n=1$ of \autoref{SPACES}.  We should say that there is a wealth of further information in \cite{Thom, Fied} that we shall not describe. At the end, we very briefly indicate how the ideas here are related to McDuff's proof \cite{McDuff} that all connected spaces are weakly equivalent to classifying spaces of discrete monoids, which is itself a variant of the proof by Kan and Thurston \cite{KT} that all connected spaces are equivalent to plus constructions of classifying spaces of discrete groups. 
 
In this subsection, $\sT = \sS$ is the category of based spaces.

\begin{defn}\label{Vee} Define $\sV$ to be the category of based spaces $\pi \colon X\rtarr [0,\infty)$ over $[0,\infty)$ such that $\pi^{-1}(0) = \ast$.  We take $0$ to be the basepoint of $[0,\infty)$, and we use the notation $(X,\pi)$ for $\pi \colon X\rtarr [0,\infty)$.
\end{defn}

We first recall the Moore loop space functor $\LA^v\colon \sT \rtarr \sV$ and its left adjoint $\SI_v\colon \sV\rtarr \sT$ and then use an adjunction $(\bL,\bR)$ between $\sV$ and $\sT$ to relate $(\SI_v,\LA^v)$ to $(\SI,\OM)$ via a composite adjunction.  The Moore loop space functor as usually understood is $\LA = \bL \LA^v\colon \sT \rtarr \sT$, and we shall
write $\XI = \SI_v\bR\colon \sT \rtarr \sT$ to match. Up to isomorphism, these are $\OM$ and $\SI$ and are therefore also adjoint. 
We then relate these adjunctions to two monads on $\sT$ and to two related monads on $\sV$ and finally show how these all fit together to give a conceptual understanding of \autoref{approxn=1}.

\begin{defn}\label{Mooreadj}  Define $\LA^v\colon \sT \rtarr \sV$ by letting $\LA^v X$ be the set of pairs $(f,s)$, where $s\in [0,\infty)$ and $f\colon [0,\infty) \rtarr X$ is a continuous map such that $f(0) = \ast$ and $f(t) = \ast$ for all $t\geq s$.  Topologize $\LA^v X$ as a subspace of 
$\mathrm{Map}([0,\infty),X)\times [0,\infty)$ and define $\pi(f,s) = s$.  Define $\SI_v\colon \sV\rtarr \sT$ by 
$$\SI_v(X,\pi)   = X\times [0,{\infty})/\{(x,t)| t = 0 \ \ \text{or} \ \  t\geq \pi(x) \}.$$
Both $\LA^v$ and $\SI_v$ extend to morphisms in obvious ways.  
\end{defn}

The lemmas from here to the diagram \autoref{adj0} are all easily verified from the definitions. Almost all are parts of \cite[Lemmas 6.2, 6.6, and 6.7]{Fied}.

\begin{lem}\label{Fied1} $(\SI_v, \LA^v)$ is an adjoint pair of functors between $\sV$ and $\sT$.
\end{lem}

\begin{defn}\label{TVadj} Define $\bL\colon \sV\rtarr \sT$ by $\bL(X,\pi) = X$.  Define $\bR\colon \sT\rtarr \sV$ by letting $\bR(X)$ be the subspace
$$\bR(X) =(\ast,0) \cup \{(x,s)| x\in X \ \ \text{and} \ \ s>0\}$$
of $X\times [0,\infty)$ and letting $\pi\colon \bR(X)\rtarr [0,\infty)$ be the projection, $\pi(x,s) = s$.
\end{defn}

\begin{lem}\label{Fied2}  $(\bL,\bR)$ is an adjoint pair of functors.  For $(X,\pi) \in \sV$, the unit $\et\colon (X,\pi) \rtarr \bR\bL (X,\pi)$ sends $x\in X$ to $(x,\pi(x))$. For $X\in \sT$, the counit $\epz\colon \bL\bR X \rtarr X$ sends $(x,s)$ to $x$. Both $\et$ and $\epz$ are homotopy equivalences.
\end{lem}

\begin{lem}\label{Fied3} For $(X,\pi)\in \sV$, define a natural homeomorphism
$$\io\colon \SI \bL X \rtarr \SI_v X$$ 
by $\io(x,t) = (x,t\pi(x))$ for $t\in I$.  Its inverse  is given by
$\io^{-1}(x,t) = (x, t/\pi(x))$ for $t\in [0,\pi(x)]$.  For $X\in \sT$, define a natural isomorphism $\bR\OM X \iso \LA^v X$ by
letting $(f,s)\in \bR\OM X$ correspond to $(g,s) \in \LA^v X$ if  $g(ts) = f(t)$ for $t\in I$. 
\end{lem}

\begin{lem}\label{Fied4} $(\SI_v,\LA^v)$ is isomorphic to the composite adjunction $(\SI\bL,\bR\OM)$.
\end{lem}

Nothing said above refers to products on loop spaces. 

\begin{defn}\label{MonV} We say that $(Y,\pi)$ is a monoid in $\sV$ if $Y$ is a monoid with unit $\ast$ and $\pi$ is a morphism of monoids, where $[0,\infty)$ is regarded as a monoid under addition.  Let $\mathrm{Mon}[\sT]$ and $\mathrm{Mon}[\sV]$ denote the categories of monoids in $\sT$ and monoids in $\sV$.  Clearly $\bL$ restricts to a functor $\bL\colon \mathrm{Mon}[\sV] \rtarr \mathrm{Mon}[\sT]$.
\end{defn}

\begin{lem}\label{Fied5} For $X\in \sT$, $\LA^v X$ is a monoid in $\sV$ with product
$$ (f,s)(g,t) = (f+ g, s+t),$$
where $f+g$ is $f$ on $[0,s]$ and $g$ on $[s,s+t]$, hence $\LA X = \bL\LA^v X$ is a monoid in $\sT$. 
\end{lem}

\begin{lem}\label{Fied6} If $M$ is a monoid in $\sT$, then $\bR M$ is a submonoid of $(M\times [0, \infty), \pi)$ in $\sV$. The adjunction 
$(\bL,\bR)$ restricts to an adjunction between $\mathrm{Mon}[\sV]$ and $\mathrm{Mon}[\sT]$.
\end{lem}

Turning to monads, we have the following general observation.  It applies to any adjunction $(\bL,\bR)$ and does not require $\epz\colon \bL\bR \rtarr \id$ to be an isomorphism, although later we will be most interested in the case when it is.

\begin{lem}\label{Fied7} If $\bC$ is a monad in $\sT$, then $\bR\bC\bL$ is a monad in $\sV$ with unit
$\bR\et_{\bC}\bL\colon \id \rtarr \bR\bC\bL$ and product
$$ \xymatrix@1{ \bR\bC\bL \bR\bC\bL \ar[r]^-{\bR\bC\epz} &  \bR\bC\bC \bL \ar[r]^-{\bR\mu_{\bC}\bL} & \bR\bC\bL.}$$
If $(X,\tha)$ is a $\bC$-algebra, then $\bR X$ is a $\bR\bC\bL$-algebra with action
$$ \xymatrix@1{ \bR\bC\bL \bR X \ar[r]^-{\bR\bC\epz} & \bR\bC X \ar[r]^-{\bR\tha} & \bR X.} $$
\end{lem}

Again letting $\sM$ be the associativity operad, recall that its associated monad on $\sT$ sends $X$ to the free monoid $\bM X$ on $X$ with unit $\ast$. Let $\et = \et_{\bM} \colon X\rtarr \bM X$ be the unit, $\et(x) = x$.   Then $\bM$ is the James construction, but we continue our historical abuse of notation with the following redefinition.

\begin{defn}\label{Fied8} Define the James construction $\bJ$ to be the monad $\bR\bM \bL$ on $\sV$.
\end{defn}

\begin{rem}\label{compare} Fiedorowicz instead defines $\bJ$ to be the free monoid monad $\bF$ on $\sV$.\footnote{In \cite{Fied}, $\bJ$ is used for both $\bM$ and his $\bJ$. We find a notational distinction to be helpful.} It is defined by
$\bF(X,\pi) = (\bM X,\tilde{\pi})$, where $\tilde{\pi}\colon \bM X \rtarr [0,\infty)$ is the unique map of monoids 
such that $\tilde{\pi} \com \et_{\bM} = \pi$.   From the point of view of $\sV$, this is the obvious sensible choice.  However, $\sV$ is introduced only to elucidate structure in $\sT$.  By freeness, there is a unique map of monoids 
$\nu\colon \bF(X,\pi) \rtarr \bJ(X,\pi)$ such that $\nu \com \et_{\bF} =\et_{\bJ}$.  These give a map of monads $\bF \rtarr \bJ$, and it is an equivalence. Another way of seeing $\nu$ is to view the identity map $\bM X \rtarr \bM X$ for $X\in \sT$ as a map of monoids
$\bL \bF(X,\pi) \rtarr \bM\bL(X,\pi)$.
Its adjoint is our equivalence 
$$ \nu\colon \bF(X,\pi)  \rtarr \bR \bM\bL(X,\pi) = \bJ(X,\pi) $$
of monoids on $\sV$.  Our $\bJ$ is essential to the conceptual comparison between structure in $\sV$ and structure in $\sT$ that is our focus, hence we will never use $\bF$. 
\end{rem}

From here, we abbreviate notation by setting $\bC = \bC_1$. As noted before, we have a map $\de\colon \sC \rtarr \sM$ of operads and thus an induced map $\de\colon \bC \rtarr \bM$ of monads.  By pullback of actions, it induces a functor
$\de^*\colon \bM[\sT] \rtarr \bC[\sT]$. The monad $\bD$ of the following definition is new.  It gives a convenient conceptual intermediary between $\bC$ and $\bJ$.\footnote{Use of it can simplify the proofs of other results of \cite{Thom, Fied} not considered here.} Examples like it will be central to Part 2.

\begin{defn}\label{monadD} Define $\bD$ to be the monad $\bR\bC\bL$ on $\sV$ and define $\de\colon \bD \rtarr \bJ$ to be the induced map $\bR\de \bL$ of monads on $\sV$.  By pullback of actions, it induces a functor $\de^*\colon \bJ[\sV] \rtarr \bD[\sV]$.
\end{defn}

We put things together in the following diagram, in which $(\SI,\OM) = (\SI^1,\OM^1)$.
\begin{equation}\label{adj0} 
\xymatrix{
\sV \ar@<.5ex>[rr]^{\bL}  \ar@<.5ex>[ddddrr]^{\bF_{\bD}} \ar@<.5ex>[dddd]^{\bF_{\bJ}} & & \sT  \ar@<.5ex>[ll]^{\bR}    \ar@<.5ex>[rr]^{\SI}  \ar@<.5ex>[ddr]^(.6){\bF_{\bC}} \ar@<.5ex>[dd]^{\bF_{\bM}}& & \sT   \ar@<.5ex>[ll]^{\OM} \ar@<.5ex>[ddl]^{\OM_{\bC}}  
\ar@/^4pc/@<.6ex> [ddddll]^-{(\bR\OM)_{\bD}}  \\
& & & & \\
& & \bM[\sT] \ar[r]^{\de^*}  \ar@<.5ex>[uu]^{\bU_{\bM}} \ar@<.5ex>[ddll]^(.6){\bR}|(0.34)\hole  &   \ar@<.5ex>[uul]^(.4){\bU_{\bC}}  \bC[\sT]   \ar@<.5ex>[uur]^{\SI_{\bC}} \ar@<.5ex>[ddl]^{\bR} & \\
& & & & \\
\bJ[\sV] \ar@<.5ex>@{-->}[uurr]^(.3){\bL}|(0.66)\hole  \ar[rr]_{\de^*}\ar@<.5ex>[uuuu]^{\bU_{\bJ}}& & \ar@<.5ex>[uuuull]^{\bU_{\bD}}  \bD[\sV] \ar@<.5ex>@{-->}[uur]^{\bL} \ar@/_4pc/@<.6ex> [uuuurr]^-{(\SI \bL)_{\bD}}& &  \\}
\end{equation}

\begin{rem}  Since $\epz\colon \bL\bR \rtarr \id$ is not an isomorphism, the two dotted arrows $\bL$ in the diagram are undefined. Nevertheless, \autoref{keyadj} gives the displayed left adjoint $(\SI\bL)_{\bD}$. 
\end{rem}

\begin{rem}\label{precursor}
Deleting $\bM[\sT]$ and $\bJ[\sV]$ and the functors to or from either, many similar situations will appear in Part 2, where we study composite adjunctions of the resulting form that give substantial new applications of our general context.  However, the resulting diagram here does not fit into the composite adjunction context since $\sM$ does not act on $\OM X$. The leftmost triangle in \autoref{adj0} pasted with the triangle with apex 
$\bD[\sV]$ give as much of the composite adjunction context as makes sense in the present context.\footnote{Nevertheless, \autoref{adjequiv} applies to give a new way to see that any connected space is weakly equivalent to the classifying space of a certain Moore loop space.  We omit the details since, by comparison with \cite[Lemma 15.4]{MayClass}, the new construction is equivalent to $B\LA X$.}  
\end{rem}

The small top right triangle in \autoref{adj0} is \autoref{adj5} for the case $n=1$ of \autoref{ass1} from \autoref{SPACES}. Clearly 
$(\SI_v,\LA^v)\iso (\SI\bL, \bR\OM)$ and $\bD$ also satisfy \autoref{ass1}.  Ignoring $\bM[\sT]$, the large central triangle with apex $\bD[\sV]$ is \autoref{adj5} for this case.  Since $\bJ\LA^v =  \bR M \bL \LA^v = \bR M \LA$, it is clear that $\bJ$ acts on all $\LA^v X$.  Therefore the adjunction $(\SI_v,\LA^v)$ and the monad $\bJ$ give us a third instance of \autoref{ass1}. Its diagram \autoref{adj5} adds 
$(\SI_{\bJ},\OM_{\bJ})$ relating $\sT$ and $\bJ[\sV]$ to the top row and left vertical arrows of \autoref{adj0}. 

All of our assumptions apply to all three of these examples of \autoref{ass1}.  Defining a map $f$ in $\sV$ to be a weak equivalence if $\bL f$ is a weak equivalence in $\sT$, we have \autoref{ass2}.  This works as in \autoref{SPACES} and is also the content of \cite[Lemma 0]{Fied}. Here the realization of a simplicial object $(X_*,\pi_*)$ in $\sV$ is $(X,\pi)$, where $X = |X_*|$ and $\pi$ is the realization of $\pi_*$.  This makes sense since the face and degeneracy operators commute with the $\pi_n$ defined on $n$-simplices. In addition to $\bL$ commuting with realization, it is easily checked that $\bR$ commutes with realization. \autoref{ass3} works as in \autoref{SPACES}. Using $m$-group completions, we verified  \autoref{ass4} in \autoref{SPACES}.  It remains to consider \autoref{ass5}, the approximation theorem, as formulated in \autoref{approxn=1}. 
We explain its proof, but without giving full details.

The following equalities are immediate from $\bL \bR = \id$ and our definitions.
$$ \bD = \bR\bC\bL, \ \ \bD\bR = \bR\bC, \ \ \bL\bD\bR = \bC \ \ \text{and} \ \
\bJ = \bR\bM\bL, \ \ \bJ\bR = \bR\bM, \ \ \bL\bJ\bR = \bM.$$
Under the isomorphism $\bR\OM\SI\bL \iso \LA^v \SI_v$, $\et$ for $\LA^v\SI_v$ factors as the composite
$$ \xymatrix@1{ \id \ar[r]^-{\et} & \bR\bL \ar[r]^-{\bR\et\bL} & \bR \OM\SI \bL \\} $$
and $\et\bR = \id \colon \bR \rtarr \bR\bL\bR  = \bR$.  Using these, we find that the top two squares make sense and commute in the 
following diagram.  Since the relationship between $\bM$ and $\bJ$ is formally the same as the relationship between $\bC$ and $\bD$,
the bottom two squares commute by the same arguments. It is formal that the middle two squares commute. We write $\vartheta_{\bC}$, $\vartheta_{\bD}$, $\vartheta_{\bJ}$, and $\vartheta_{\bM}$ for the natural actions.   

\begin{equation}\label{RealCute}   
\xymatrix{ 
\bR\bC X  \ar[r]^-{\bR\bC \et}  \ar@{=}[d]  &  \bR\bC \OM\SI X \ar[r]^-{\bR\vartheta_{\bC}}  \ar[d]^{\iso} &   \bR\OM\SI X \ar[d]^{\iso}  \\  
\bD\bR  X   \ar[r]^-{\bD \et}   \ar[d]_{\delta}  &  \bD \LA^v\SI_v \bR X \ar[r]^-{\vartheta_{\bD}} \ar[d]^{\de} &   \LA^v\SI_v \bR X  \ar[d]^{=} \\
\bJ\bR  X   \ar[r]^-{\bJ \et}   \ar@{=}[d]  &  \bJ \LA^v\SI_v \bR X \ar[r]^-{\vartheta_{\bJ}} \ar[d]^{\iso} &   \LA^v\SI_v \bR X  \ar[d]^{\iso} \\
\bR\bM X  \ar[r]_-{\bR\bM \et}  & \bR\bM \LA \XI X  \ar[r]_-{\bR \vartheta_{\bM}} & \bR\LA\XI X  \\} 
\end{equation}

The top squares identify $\bR\al_{\bC}$ with $\al_{\bD}$, the middle squares say that 
$\al_{\bJ}\com \de = \al_{\bD}$, and the bottom squares identify $\bR\al_{\bM}$ with $\al_{\bJ}$.  We repeat that $\de\colon \bC \rtarr \bM$ is a natural equivalence because 
$\sC(j) \rtarr \sM(j)$ is a $\SI_J$-equivalence for all $j$. It follows that $\de\colon \bD \rtarr \bJ$ is a
natural equivalence. Therefore \autoref{approxn=1} holds if the composite
$\al_{\bM} = \vartheta_{\bM} \com \bM\et$ on the bottom row is an $m$-group completion.  Fiedorowicz \cite {Fied} has given an ingenious argument from here, starting from an insight of Thomason \cite{Thom}.  We shall explain the idea but not give full details.

Let $\mathrm{Mon}$ be the category of monoids in $\sT$.  Below, monoids mean topological monoids with unit as basepoint (assumed to be non-degenerate).  Recall \autoref{class}.  The following idea comes from \cite{Thom}.

\begin{defn}\label{mongp} Say that a monoid $N$ is grouplike if $\pi_0(N)$ is a group. Clearly this holds if and only if $N$ is $m$-grouplike as a Hopf space.  Say that a map 
$f\colon M\rtarr N$ of monoids, where $N$ is grouplike, is a ${mon}$-group completion if the induced map $Bf\colon BM\rtarr BN$ is a weak equivalence.
\end{defn}

\begin{lem}\label{mongpcom} If $f\colon M \rtarr N$ is a ${mon}$-group completion, then $f$ is an 
$m$-group completion as a map of Hopf spaces.\end{lem}
\begin{proof}  This follows from the commutative diagram
$$\xymatrix{
 M \ar[r]^-{f} \ar[d]_{\chi} & N \ar[d]^{\chi}\\
 \OM BM \ar[r]_-{\OM B f}   & \OM B N. \\}
$$
If $Bf$ is a weak equivalence, then so is $\OM Bf$. Since $N$ is grouplike, comparison of the path space fibration of $BN$ with the universal quasi-fibration $EN \rtarr BN$ (e.g. \cite[Theorem 7.6]{MayClass}) shows that $\chi$ for $N$ is a weak equivalence. Therefore
$f$ is weakly equivalent to $\chi$ for $M$ in the arrow category of $\sH_m$. 
\end{proof}

Thus to prove that $\al_{\bM} \colon \bM X \rtarr \LA\SI X$ is a group completion, it suffices to prove that $B\al_{\bM}$ is a weak equivalence. We reduce that problem to another one using the following diagram, which is adapted from the proof of
\cite[Theorem 6.12]{Fied}.

\begin{equation}\label{FiedDiag}
\xymatrix{
\SI X \ar[rr]^-{\SI\et_{\bM}} \ar[d]_{\SI\et} & & \SI\bM X \ar[dl]_{\SI\bM\et}\ar[r]^-{\io} \ar[d]^{\SI\al_{\bM}} & B\bM X \ar[d]^{B\al_{\bM}}\\
\SI\LA\XI X \ar[drr]_{\epz} \ar[r]^{\SI\et_{\bM}} & \SI \bM\LA\XI X \ar[r]^{\SI\vartheta_{\bM}}& \SI\LA\XI X  \ar[r]^-{\io} \ar[d]^{\epz} 
& B\LA\XI X \ar[dl]^\xi\\
& & \XI X \iso \SI X & \\}
\end{equation}
On the top, we see two naturality diagrams and the definition of $\al_{\bM}$ as a composite. The composite $\SI\vartheta_{\bM}\com \SI \et_{\bM}$ is the identity, so the bottom left triangle commutes tautologically no matter what $\epz$ is. Using the 
isomorphism $\XI\iso \SI$, we can define
$\epz$ to be the counit of the adjunction $(\XI,\LA)$, and then $\epz\com \SI\et = \id$ on the left by a triangle identity.  The map $\xi$ is an instance of a natural map $\xi\colon B\LA Y\rtarr Y$ that is defined explicitly in \cite[Lemma 15.4]{MayClass}, where it is shown to be a weak equivalence when $Y$ is connected.  The relation $\xi\com \io = \epz$ is checked in \cite[Lemma 6.11]{Fied}.
Since $Y = \XI X$ is connected, $\xi$ is a weak equivalence.  Therefore $B\al_{\bN}$ is a weak equivalence by the following result,
which is \cite[Theorem 6.10]{Fied}. 

\begin{thm}\label{keyFied} The top composite $\io\com \SI\et_{\bM}$ in \autoref{FiedDiag} is a weak equivalence.
\end{thm}
\begin{proof}[Brief Sketch] The total singular complex $\mathbf{Sing}M$ of a monoid $M$ is a simplicial discrete monoid.  Standard properties of geometric realization as in \autoref{ass2} reduce the proof to the case when $X$ is discrete, which is proven in \cite[Lemma 6.9]{Fied} by use of the analog for monoids (\cite [Theorem 4.1]{Fied}) of a result of J.H.C. Whitehead \cite{White} (restated as \cite[Theorem 4.0]{Fied}) for groups. The result for monoids describes when $B$ takes pushouts of discrete monoids to pushouts (up to homotopy equivalence) of spaces, and that allows an induction up skeleta.
\end{proof}

The cited result of Whitehead and its analog for monoids underlie the results of Kan and Thurston and of McDuff \cite{KT, McDuff}, respectively, and the latter is given a simplified proof \cite[Theorem 3.5]{Fied} using \cite[Theorem 4.1]{Fied}. Fiedorowicz 
\cite[Theorem 6.15]{Fied} also gives a curious homological group completion result for the map $\chi\colon \bM X \rtarr \OM B MX \htp \OM\SI X$.  That result is a partial non-commutative generalization of \autoref{gpcompthm}.

\subsection{Delooping based G-spaces for finite groups $G$}\label{GSPACES}

Here $\sT$ and $\sS$ are both taken to be the category $G\sT$ of based $G$-spaces, tacitly nondegenerately based where necessary, and based $G$-maps.  Again, basepoint issues are handled carefully in \cite{MMO}.  We take $(\SI,\OM)$ in the general theory to be $(\SI^V,\OM^V)$ for some representation $V$ of $G$.   We always assume that $V$ contains the trivial representation $\bR$ to ensure that its fixed point spaces are in $\sH$; we assume that it contains $\bR^2$ when we wish to ensure that homology algebras are commutative.  Some relevant facts about passage to $H$-fixed points for $H\subset G$ are collected in \autoref{SoEasy} below. 

We take  $\bC = \bC_V$  to be the monad on based $G$-spaces associated to an $E_V$-operad $\sC_V$ that acts naturally on $\OM^V X$ for $G$-spaces $X$, such as the little $V$-disks operad or the Steiner operad for $V$ or the product of either with an $E_{\infty}$ $G$-operad.  Details about these operads are in \cite[Section 1.1]{GM3}.  This places us in a context as specified in \autoref{ass1}.  

A map $f$ of based $G$-spaces is a weak equivalence if each of its fixed point maps $f^H$ is a nonequivariant weak equivalence. We comment on the realization part of \autoref{ass2}.  Geometric realization commutes with finite products and pullbacks and therefore commutes with finite limits.  Passage to fixed points is a finite limit, so realization on $G$-spaces commutes with passage to fixed points.  From here, the verification that realization satisfies parts (i) through (iv) of \autoref{ass2} presents no difficulty.    Part (v) was proven by Hauschild \cite{Haus} by reduction to the nonequivariant case, and the proof also appears in \cite[pp. 495-496]{CW} and in \cite[Section 4.6]{ZouFactorization}. We have an analog of \autoref{n=n}.

\begin{prop}\label{Gconn}  A $G$-space $X$ is  $\OM^V$-connective in the sense of \autoref{LAcon} if and only if it is  $V$-connective, meaning that  $X^H$ is $|V^H|$-connective for each subgroup $H\subset G$.
\end{prop}
\begin{proof}  The proof is like that of \autoref{n=n}, although full details involve explaining how the map 
$\ga^V\colon \bT \OM^V_* X_* \rtarr \OM^V \bT X_*$ is proven to be a weak equivalence by passage to fixed points and use of the nonequivariant version there.
\end{proof}

\autoref{ass3} works as in \autoref{SPACES}.
With any of the standard model structures on $G\sT$ (e.g. \cite[section VI.5]{EHCT} and \cite[Chapter 17]{morecon} made equivariant), $(\SI^V,\OM^V)$ is a Quillen adjunction since $\OM^V$ preserves fibrations and acyclic fibrations. The following equivariant versions of Definitions \ref{Hopf} and \ref{Hopf2} are variants of \cite[Definitions 1.9 and 1.10]{GM3}.   

\begin{defn} A Hopf $G$-space is a based $G$-space $X$ with a product $G$-map $X\times X \rtarr X$ such that left and right multiplication by $\ast$ are $G$-homotopic to the identity.   A $G$-monoid  $M$ is a based $G$-space $M$ with an associative product $G$-map such that left and right multiplication by $\ast$ are the identity.
\end{defn}

\begin{defn}\label{HopfG}  Let $\sH_G$ be the category of $G$-homotopy associative Hopf $G$-spaces.  A space in $\sH_G$ is {\em grouplike} if  each $\pi_0(X^H)$ is a group.    Let $\sH_{G,com}$ be the full subcategory of $G$-homotopy commutative Hopf $G$-spaces in $\sH_G$. Define a group completion $f\colon Y\rtarr Z$ to be a map in $\sH_{G,com}$ such that $Z$ is grouplike and  each $f^H$ is a group completion in the sense of  \autoref{Hopf}.
\end{defn}

The definition just given is designed for use when $\bR^2\subset V$, so that the relevant Hopf $G$-spaces are $G$-homotopy commutative.

\begin{defn}\label{HopfG2}  
Let $\sH_{G,m}$ be the full subcategory of $\sH_G$ whose objects are weakly equivalent to topological $G$-monoids. A space in $\sH_{G,m}$ is {\em grouplike} if each $\pi_0(X^H)$ is a group.    When $\bR\subset V$, a map $f$ in $\sH_{G,m}$ is an  $m$-group completion if it is weakly equivalent in the arrow category of $\sH_{G,m}$ to a $G$-map $\et\colon M\rtarr \OM B M$ for some topological $G$-monoid $M$.  The definition makes sense since, for a topological $G$-monoid $M$,  $\OM B M$ is weakly equivalent to the Moore loop $G$-space $\LA BM$, which is a topological $G$-monoid.  Using \autoref{SoEasy}, it follows by passage to $H$-fixed points that $f^H$ is an $m$-group completion for all $H\subset G$.
\end{defn}

The following lemmas are now immediate from their nonequivariant versions in Lemmas \ref{square2}, \ref{square3}, and \ref{weakstrong}.

\begin{lem}\label{squareG2}  Using the notion of group completion given in \autoref{HopfG},  the category $\sH_{G,com}$ satisfies the conditions specified in \autoref{screwy} when $\bR^2\subset V$.
\end{lem}

\begin{lem}\label{squareG3} Using the notion of $m$-group completion given in \autoref{HopfG2}, the
category $\sH_{G,m}$ satisfies the conditions specified in Definition 1.13 when $\bR\subset V$.
\end{lem}

Let  $\sH_{G, m,com}$ denote the full subcategory of homotopy commutative Hopf $G$-spaces in $\sH_{G,m}$.

\begin{lem}\label{weakstrongG}  If a map $f\colon Y\rtarr Z$ in $\sH_{G,m,com}$ is an $m$-group completion, then $f$ is a  group completion. Conversely, if a map $f\colon Y\rtarr Z$ in $\bC_1[G\sT]$ is a group completion, then it is an $m$-group completion when regarded as a map in $\sH_{G,m,com}$. 
\end{lem}

By passage to fixed points, using \autoref{SoEasy}, we see that \autoref{gpcompthm}  remains true for $G$-monoids $M$.  Just as in the paragraph above \autoref{Cone}, the operad  $\sC_1 = \sC_{\bR}$ is contained in $\sC_V$, and $G$ acts trivially on it; $G$ also acts trivially on $\sM$.  With all group actions induced by the group action on $Y$, \autoref{HopfbG}  applies directly to $\sC_V$-$G$-spaces $Y$ to define $(\bG,g)$.  Since the terms other than $Y$ in the construction have trivial action by $G$,  the  construction commutes with passage to $H$-fixed points.   From here, the verification of \autoref{ass4} is implied by application of the nonequivariant version to fixed points. 

The approximation theorem, \autoref{ass5}, is more difficult.   The following analog of \autoref{approxn=1} will be proven in \autoref{MooreG}.\footnote{We find the proof by reference in \cite{GM3} to be unconvincing.}
 
\begin{thm}\label{approxGn=1} For a $\sC_1$-$G$-space $Y$, $\al_1\colon \bC_1 Y \rtarr \OM\SI Y$ is an $m$-group completion. 
\end{thm} 

However, we are interested in the case of general $V$, and we have nothing to add to the discussion in \cite[Section 1.2]{GM3}.
We restate \cite[Theorem 1.11]{GM3}.  It restricted to the Steiner operad $\sK_V$, but it applies to our $\sC_V$.
 
\begin{thm}[The approximation theorem] If $\bR\subset V$, $\al_V\colon \bC_V X \rtarr \OM^V\SI^V X$ is an $m$-group completion. 
Therefore, if $\bR^2\subset V$, $\al_V\colon \bC_V X \rtarr \OM^V\SI^V X$ is a group completion.
\end{thm}

We are very far from a computational proof along the lines of Cohen's nonequivariant proof.  The proof discussed in \cite{GM3} is due to Rourke and Sanderson \cite{RS} and starts from Segal's nonequivariant proof and their ``compression theorem''.  There is further work to be done here.

Thus all assumptions are satisfied,  with more detailed verifications available in \cite{GM3} and the other cited sources.   We remark that  
$$\bC S^0 \htp \amalg_{j\geq 0} F(V,j)/\SI_j,$$
where $F(V,j)$ is the configuration $G$-space of (ordered) $j$-tuples of distinct points of $V$.
It is surprising how little these naturally occurring equivariant configuration spaces have been studied.
We have the following analog of \autoref{space}.

\begin{thm}\label{Gspace} Take $\sC= \sC_V$.  For $\sC$-spaces $Y$, the unit $\et_{\bC}$ induces a group completion if $\bR^2\subset V$ (an  $m$-group completion if 
$\bR\subset V$)  and therefore a $\bC$-group completion
$$Y \rtarr \OM^V_{\bC}\bE Y\iso \OM^V_{\bC}  \SI^V_{\bC} \bB Y.$$
Moreover, $V$-connective based $G$-spaces $Z$ are naturally equivalent to  coequalized $V$-fold suspensions $\SI^V_{\bC}  \overline{\OM^V_{\bC} X}$.
\end{thm}

\begin{rem}\label{comLie}  It is proven in \cite{MMO} that the results of this subsection apply equally well to compact Lie groups $G$, provided that we restrict attention to just those representations $V$ all of whose isotropy groups have finite index in $G$.
\end{rem}

\begin{rem}   For a given $V$, define $\sF_V$ to be the family of subgroups $H$ of $G$ such that $G/H$ embeds in $V$.    If we restrict  to those $H\in \sF_V$ when defining the group completion property and $V$-connectivity, then, using the Steiner operad for $V$, the proof goes through to give a family version of  \autoref{Gspace}.
\end{rem}

\subsection{Moore loop $G$-spaces and $1$-fold loop $G$-spaces}\label{MooreG}

In this subsection, we adapt \autoref{Moore} to prove \autoref{approxGn=1}. We stay in the context of \autoref{GSPACES} with $n=1$, taking $\sC$ and $\bC$ to be $\sC_{\bR^1}$ and $\bC_{\bR^1}$, except that we no longer require $G$ to be finite; it can be any topological group here.  We add in $G\sV = G\sT/[0,\infty)$, where $G$ acts trivially on $[0,\infty)$.  As far as we know, nothing at all of the equivariant version of \autoref{Moore} appears in the literature.  However we claim that nearly all of it adapts without change equivariantly and, where change is needed, passage to fixed point spaces does the trick.

We use the same notations equivariantly for the equivariant versions of all of the functors and transformations in \autoref{Moore}. The introductory paragraphs apply verbatim.  We adapt \autoref{Mooreadj} by letting $G$ act on Moore loops 
just as it acts on ordinary loops: $(gf)(t) = gf(t)$.  We define a $G$-monoid $M$ to be a $G$-space which is a monoid such that $g(mn) = gm\cdot gn$.  Then all of the definitions and results recorded in \ref{Mooreadj} through 
\ref{RealCute} apply exactly as written but with $\sT$ and $\sV$ replaced by $G\sT$ and $G\sV$ and with monoids understood to be $G$-monoids.

We then have the category $G\mathrm{Mon}$ of $G$-monoids and $G$-monoid maps.  For a $G$-monoid $M$, we define the classifying $G$-space $BM$ and the natural $G$-maps $\io\colon \SI M \rtarr BM$ and $\chi\colon M\rtarr \OM BM$ exactly as in \autoref{class}.  That is, $G$-acts through its action of $G$ on $M$ and nothing else changes.\footnote{We ignore consideration of what $BM$ classifies since that is irrelevant here.}  We write out the changes to \autoref{mongp} needed from here.

\begin{defn}\label{Gmongp} Say that a $G$-monoid $N$ is grouplike if $N$ is $m$-grouplike as a Hopf $G$-space.  Say that a map 
$f\colon M\rtarr N$ of $G$-monoids, where $N$ is grouplike, is a ${G{mon}}$-group completion if the induced $G$-map $Bf\colon BM\rtarr BN$ is a weak equivalence of $G$-spaces.
\end{defn}

The proof of \autoref{mongpcom} applies to show that if $f$ is a $G$-group completion, then $f$ is an $m$-group completion as a map of Hopf $G$-spaces.  Then the diagram \autoref{FiedDiag} applies to reduce the proof of \autoref{approxGn=1} to the following statement.

\begin{thm}\label{keyGFied} For a $G$-space $X$, the composite $G$-map
\begin{equation}\label{topcomp}
\xymatrix{
\SI X \ar[r]^-{\SI\et_{\bM}}  & \SI\bM X \ar[r]^-{\io} & B\bM X \\}
\end{equation}
is a weak equivalence.
\end{thm}

This holds nonequivariantly by \autoref{keyFied}.  Thus the following result implies that it also holds equivariantly, as required. 

\begin{thm}\label{SoEasy}
\enumerate[(i)] Let $X$ be a based $G$-space and $M$ be a $G$-monoid.
\item $(\OM X)^H$ is naturally isomorphic to $\OM(X^H)$.
\item $(\SI X)^H$ is naturally isomorphic to $\SI(X^H)$.
\item $(\bM X)^H$ is naturally isomorphic to $\bM (X^H)$.
\item $(BM)^H$ is naturally isomorphic to $B(M^H).$
\item After passage to $H$-fixed points, the composite \autoref{topcomp} is isomorphic to 
$$
\xymatrix{
\SI X^H \ar[r]^-{\SI\et_{\bM}}  & \SI\bM X^H \ar[r]^-{\io} & B\bM X^H. \\}
$$
\end{thm}
\begin{proof}
Remember that basepoints are assumed to be non-degenerate when that is needed. It is needed here since we use the standard fact that passage to $H$-fixed points preserves pushouts, one leg of which is a $G$-cofibration.  Clearly (i) holds and it is now also clear that (ii) holds.  Let $\si_n\colon sX^n\subset X^{n+1}$ denote the space of points one coordinate of which is the basepoint, and similarly for $M$. Then $\si_n$ is a $G$-cofibration. 
For (iii), $\bM X$ is filtered by word length, and $F_{n+1}\bM X$ is obtained from $F_n X$ and $X^{n+1}$ by a unit identification given by a pushout diagram, one leg of which is $\si_n$.  For (iv), $BM$ is the geometric realization of a simplicial $G$-space with $n$-simplices $M^n$.  As such, it is filtered and $F_{n+1}BM$ is obtained from $F_nBM$ and $X^{n+1}$ by a pushout diagram, one leg of which is the product of $\si_n$ and the identity map of $\DE_{n+1}$.  The maps $\et_M$ and $\io$ are each the inclusion of the first filtration, hence (v) follows. 
\end{proof}

Parenthetically, while we have not pursued the details, in view of the result of McDuff \cite{McDuff} and its reproof by Fiedorowicz \cite{Fied}, it seems plausible that the following equivariant generalization holds.\footnote{Sunny Zhang has since proven this in a paper not yet completed.}

\begin{conj}\label{McConj} Let $G$ be a discrete group.  Then any $G$-connected based $G$-space $X$ is weakly equivalent to $B M(X)$ for some discrete $G$-monoid $M(X)$.
\end{conj}

As in the nonequivariant proof of \cite[Theorem 3.5]{Fied}, this holds if the classifying space of any topological $G$-monoid is equivalent to the classifying space of some discrete $G$-monoid, as is proven nonequivariantly in \cite[Theorem 3.4]{Fied}.

\subsection{Spectra and the adjunction $(\SI^{\infty},\OM^{\infty})$}\label{SPECTRA}

To pass from iterated loop space theory to infinite loop space theory, it is necessary to choose a good target category of spectra.  This is not the place for a full treatment, but we must explain our choice.  If one is primarily interested in formal properties of the stable homotopy category, one can modernize Boardman's original construction\footnote{M. J. Boardman.  Unpublished 1964 thesis.} of that category, for example by using Lurie's $\infty$-categories.  That in effect ignores the point-set level precision that is our focus and eliminates all hope of using classical adjunctions as in our \autoref{ass1}.   One might next try to use one's favorite category of spectra that is symmetric monoidal under the smash product.  There are several good choices, such as symmetric spectra, orthogonal spectra, and EKMM $S$-modules \cite{MMSS, EKMM}, and there are comparisons among them showing that each gives rise to a category equivalent to Boardman's original stable homotopy category.  But there is a conundrum.  By a result of Lewis \cite{Lewis} (or see \cite[Theorem 11.1]{Rant1}), any such choice is {\em incompatible} with an adjunction $(\SI^{\infty},\OM^{\infty})$ such that the unit for the smash product is the sphere spectrum  $S=\SI^{\infty}S^0$.   

The precursor \cite{LMS} of \cite{EKMM} gives the category $\sS$ of spectra that we shall use.  As said before, it is the only choice we know of in which the achingly elementary \autoref{ass1} makes sense. It is not symmetric monoidal in the usual sense, but it is the term $\sS_1$ of a symmetric monoidal  {\em graded} category of spectra.\footnote{An exposition giving a full treatment of this old idea is needed.}
The smash product in {\em every} known symmetric monoidal category of spectra is obtained by suitably internalizing an ``external smash product''   $\sS_1 \times \sS_1 \rtarr \sS_2$ of a symmetric monoidal graded category of spectra.  A first formalization of this point of view is in \cite{LMS}.  We shall find in Part 3 that the operadic internalization developed in \cite[Chapter VIII]{LMS} is well suited for the multiplicative elaboration of the present theory. 

We here recall the categories  of (LMS) prespectra and spectra from \cite{LMS} and exhibit the adjunction $(\SI^{\infty},\OM^{\infty})$. We let $U = \bR^{\infty}$ with its standard inner product. Define an indexing space to be a finite dimensional subspace of $U$ with the induced inner product. A (coordinate free) prespectrum $T$ consists of based spaces $TV$ and based maps 
$$\si\colon \SI^{W-V}TV\rtarr TW$$
with adjoints
$$\tilde{\si}\colon TV\rtarr \OM^{W-V}TW$$ 
for $V\subset W$.  Here $W-V$ is the orthogonal complement of $V$ in $W$ and $S^{W-V}$ is its one point compactification; 
$\tilde{\si}$  must be the identity when $V=W$, and the obvious transitivity condition must hold when $V\subset W\subset Z$.  A prespectrum $T$ is an {\em inclusion prespectrum} if each map $\tilde{\si}$ is an inclusion.  It is a  {\em spectrum} if each map $\tilde{\si}$ is a homeomorphism. We then usually write $E$ rather than $T$.    

A map $f\colon T\rtarr T'$ of prespectra consists of based maps   $f_V\colon  TV\rtarr T'V$ such that the diagram
$$\xymatrix{  
TV  \ar[r]^-{f_V}  \ar[d]_{\tilde{\si}}  &  T'V \ar[d]^{\tilde{\si}}   \\
\OM^{W-V} TW \ar[r]_-{\OM^{W-V}f_W}  & \OM^{W-V} T'W\\}
$$
commutes when $V\subset W$.  We let $\sP$ and $\sS$ denote the category of prespectra and its full subcategory of spectra.   The inclusion $\ell \colon \sS\rtarr \sP$ has a left adjoint spectrification functor $L\colon \sP\rtarr \sS$ \cite[Theorem I.2.2]{LMS}.   When $T$ is an inclusion prespectrum, 
$$(LT)(V) = \colim_{V\subset W} \OM^{W-V} TW,$$
where the colimit is taken over the maps
$$  \OM^{W-V}\tilde{\si} \colon \OM^{W-V} TW \rtarr \OM^{W-V}\OM^{Z-W}TZ \iso  \OM^{Z-V} TZ $$
This makes sense since, for $V\subset W$, 
$$ LT(V)  \iso \colim_{W\subset Z}\OM^{W-V}\OM^{Z-W} TZ \iso  \OM^{W-V}\colim_{W\subset Z} \OM^{Z-W} TZ  = \OM^{W-V} TW.$$

We may restrict attention to any cofinal set of indexing spaces $\sV$ in $U$;
we require $0$ to be in $\sV$ and we require the union of the $V$ in $\sV$ 
to be all of $U$. Up to isomorphism, the category $\sS$ is independent of the
choice of $\sV$.  The default is $\sV = \sA\ell\ell$.  We can define 
prespectra and spectra in the same way in any countably infinite dimensional real inner product space $U$.
The default is $U=\bR^{\infty}$.

For a based space $X$, we have an  obvious suspension prespectrum $\{\SI^VX \}$ with $V^{th}$ space  $\SI^VX$. The maps $\si$ are the evident identifications  $\SI^{W-V} \SI^V X\iso \SI^W X$, and their adjoints $\tilde{\si}$ are inclusions.   We define 
$$\SI^{\infty} X = L\{\SI^V X\}.$$
More explicitly, define
$$QX = \colim \OM^V\SI^V X,$$
where the colimit runs over the maps
$$ \OM^V\tilde{\si}\colon  \OM^V\SI^V X \rtarr  \OM^V \OM^{W-V}\SI^{W-V}\SI^V X \iso \OM^W\SI^W X.$$
Then the $V$th space of $\SI^{\infty} X$ is $Q\SI^V X$.  Let $\et\colon X\rtarr QX$ be the natural inclusion.  

For a spectrum $E$, we define $\OM^{\infty}E = E(0)$; we usually write it as $E_0$.  The functors $\SI^{\infty}$ 
and $\OM^{\infty}$ are adjoint; $QX$ is $\OM^{\infty}\SI^{\infty}X$, and $\et$
is the unit of the adjunction.  The counit $\epz\colon \SI^{\infty}\OM^{\infty}E  \rtarr E$ is adjoint to the map of prespectra
$\{\SI^VE_0\}\rtarr \ell E$ which at level $V$ is $\si\colon  \SI^V E_0\rtarr EV$.    We sometimes write  
$\GA^{\infty} = \OM^{\infty}\SI^{\infty}$ for the associated monad. 

\begin{defn}  Spectra have homotopy groups, and a map of spectra is a weak equivalence if it induces an isomorphism on homotopy  groups.
\end{defn}

For \autoref{ass1}, we recall the following definition.

\begin{defn}\label{oper1}   An operad $\sC$ of spaces is an $E_{\infty}$ operad if each $\sC(n)$ is $\SI_n$-free and contractible.  A space with an action of some $E_{\infty}$-operad is an $E_{\infty}$ space.  
\end{defn}

We take $\bC$ to be the monad associated to the infinite Steiner operad  $\sK_{\infty}$, as in \cite[Section 3]{Rant1} or \cite{GM3}.  It is an $E_{\infty}$ operad.   By passage to colimits from the actions of the Steiner operads  $\sK_V$ on $V$-fold loop spaces,  $\sK_{\infty}$ acts naturally on the infinite loop spaces  $E_0=\OM^{\infty}E$.  We could equally well take  $\bC$ to be the monad associated to any $E_{\infty}$ operad that acts naturally on infinite loop spaces, such as the product of any $E_{\infty}$ operad with $\sK_{\infty}$.   This puts us in the context of \autoref{ass1}.

It is very easy to generalize the context just established to $G$-spectra for a finite group $G$.  Everything said so far applies almost verbatim.   We replace $\sT$ by $G\sT$, so replace spaces and maps by $G$-spaces and $G$-maps, and we replace $U$ by a $G$-universe $U_G$.  This means that $U_G$ is the sum of countably many copies of each of a chosen set of irreducible representations of $G$; we insist that the trivial representation be in our set.  We then take our indexing spaces $V$ to be finite dimensional sub  $G$-inner product spaces of $U_G$.
A universe $U_G$ is {\em complete} if it contains all irreducible representations of $G$.  The $G$-spectra are then said to be {\em genuine}.  A universe $U_G = U$ is trivial if $G$ acts trivially, in which case we just see  spectra with $G$-actions. Such $G$-spectra are said to be  {\em classical}, or {\em naive}.   We write $G\sS$ for the category of genuine $G$-spectra indexed on a complete universe $U_G$ and the $G$-maps of $G$-spectra.  

The trivial universe $U$ is a sub-universe of any other universe $U_G$.  We define the $H$-fixed point spectrum of a genuine $G$-spectrum  $E$ by first restricting $E$ to the subuniverse $U$ and then taking $H$-fixed points levelwise.  That is,
$$  E^H(V) = E(V)^H $$ where $G$ acts trivially on $V$.   Homotopical properties of $G\sS$ are inherited from those of $\sS$ by passage to fixed points. 

\begin{defn}  A map $f\colon  D \rtarr E$ of $G$-spectra is a weak equivalence if each $f^H\colon D^H\rtarr E^H$ is a weak equivalence of spectra.
\end{defn}

We use the following definition to define $E_{\infty}$ operads of $G$-spaces.

\begin{defn}\label{famFn}   Let $\bF_n$ denote the family of all subgroups $\GA$  of $G\times \SI_n$ such that 
$\GA\cap \SI_n =\{e\}$.  Each such $\GA$ is the graph  $\GA_{\al} = \{(h,\al(h))\,|\, h\in H\}\subset G\times \SI_n$
of some homomorphism $\al\colon H\rtarr \SI_n$, where $H\subset G$. Taking $\al$ to be trivial, we see 
that $H\in\bF_n$ for all $n$ and all $H\subset G$.
\end{defn}

\begin{defn}\label{oper2}   An operad $\sC$ of $G$-spaces is an $E_{\infty}$ operad if each $\sC(n)$ is $\SI_n$-free and if $\sC(n)^{\GA}$ is contractible for all $\GA\in \bF_n$.  A $G$-space with an action of some $E_{\infty}$-operad is an $E_{\infty}$ $G$-space.  
\end{defn}

With these definitions in place, everything said above works in the same way to place us in the context of \autoref{ass1}.

\subsection{From $G$-spaces to genuine $G$-spectra for finite groups $G$}\label{SpaceSpectra}

We prove analogs of Theorems \ref{space} and \ref{Gspace} for spectra and $G$-spectra here.  We view the nonequivariant case as the case when $G$ is the trivial group. We follow the modernized sketch of the recognition principle that is given non-equivariantly in \cite[Section 9]{Rant1}.  That is made  equivariant and compared with the orthogonal $G$-spectrum machine in \cite[Sections 2.3 and 2.4]{ GM3}.  We restrict to a complete universe and delete it from the notation.   We take $\sT$ in the general theory to be the category $G\sT$ of based $G$-spaces.\footnote{The caveat  in \autoref{bspt1} concerning non-degenerate basepoints applies.}   We take $\sS$ to be the category $G\sS$ of (genuine) $G$-spectra and take  $(\SI,\OM)$ to be $(\SI^{\infty}, \OM^{\infty})$.  

We have already placed ourselves in the context of \autoref{ass1}, and we have defined the weak equivalences needed for the first portion of \autoref{ass2}.  In fact, by \cite[Theorem VII.4.4]{EKMM},  $G\sS$ is a model category, and it is again clear that $(\SI^{\infty},\OM^{\infty})$ is a Quillen adjunction since $\OM^{\infty}$ preserves fibrations and acyclic fibrations.   Just as for spaces and $G$-spaces, all $G$-spectra are fibrant objects.   Model categories with this property are often especially convenient since there is no need to keep track of how fibrant replacement behaves with respect to constructions of interest.  This feature holds for our $G$-spectra and can be viewed as the reason that $\OM^{\infty}$ behaves well both formally and homotopically. This feature does not hold for symmetric or orthogonal $G$-spectra.

The portion of \autoref{ass2} that deals with realization essentially follows by passage to colimits from \autoref{GSPACES}. Nonequivariantly, we define a spectrum to be connective if it is connective in the usual sense that its homotopy groups in negative degrees are zero.  We define a $G$-spectrum to be connective if all of it fixed point spectra are connective. We have another analog of \autoref{n=n},  proven similarly to \autoref{Gconn}.

\begin{prop}\label{Gconn2}  A $G$-spectrum is  $\OM^{\infty}$-connective in the sense of \autoref{LAcon} if and only if it is connective.
\end{prop}

\autoref{ass3} works as for $G$-spaces.  We also define group completion as for $G$-spaces.  The definition now is for $E_{\infty}$ $G$-spaces and is the same as for $E_V$-$G$-spaces: it is given by homological group completion on $H$-fixed point spaces for $H\subset G$.
From here it is implicit or explicit in \cite{GM3} that  Assumptions \ref{ass4} and \ref{ass5} hold.

\begin{thm}\label{spectra}  For $E_{\infty}$ $G$-spaces $Y$, the unit $\et_{\bC}$ induces a group completion and therefore a $\bC$-group completion  
$$Y \rtarr \OM^{\infty}_{\bC}\bE Y\iso \OM^{\infty}_{\bC}  \SI^{\infty}_{\bC} \bB Y,$$
and any connective $G$-spectrum $E$ is equivalent to  $\SI^{\infty}_{\bC}  \OM^{\infty}_{\bC} E$.
\end{thm}

\begin{rem}\label{alluniverses}  The theory here adapts without change for a general $G$-universe $U_G$, giving  an analog of \autoref{spectra}, provided that we take $\bC$ to be the monad associated to the Steiner operad $\sC$ for the universe  $U_G$. 
\end{rem}

\begin{rem}\label{finiteindex}  The discussion of $G$-spectra in \autoref{SPECTRA}  summarizes the restriction to finite groups $G$ of the theory of $G$-spectra developed for general compact Lie groups in \cite{LMS}.   As in \autoref{comLie},  \autoref{spectra} generalizes to compact Lie groups $G$ provided that we restrict attention to representations $V$ of $G$ whose isotropy groups have finite index in $G$.  Here we use a universe defined using only such representations and a  Steiner operad $\sC$ for a compact Lie group $G$ that is the colimit of Steiner operads $\sC_V$ for such $V$.  Details are in  \cite[Section 9.3]{MMO}.
\end{rem}
 
\begin{rem}\label{classical} For any topological group  $G$,\footnote{Again, we require $e\in G$ to be a non-degenerate basepoint to avoid pathology.}  we can use the trivial $G$-universe and a non-equivariant $E_{\infty}$ operad, viewed as $G$-trivial, to obtain the analog of \autoref{spectra} for classical (alias naive) $G$-spectra.   These are just spectra, as defined nonequivariantly,  with actions of $G$ on component spaces and equivariant structure maps such that $E_n\rtarr \OM E_{n+1}$ is a $G$-homeomorphism.  All of our assumptions hold.  The approximation theorem is obtained from its nonequivariant version by passage to fixed points.
\end{rem}
\part{Composite adjunctions}
 
\section{The general theory of composite adjunctions}\label{ABCDEF}   

\subsection{The categorical context and Assumption A}\label{A}
We start with a given  instance of the context described in \autoref{ass1}.   Thus we have an adjunction   
$(\SI,\OM)$ between categories  $\sT$ and $\sS$ and a monad  $\bC$ on $\sT$.  We have the formal consequences described in 
\autoref{CONTEXT1}. 

\begin{asscom}\label{ass6} We assume given a second adjunction
$(\bL,\bR)$ from a third cocomplete category $\sV$ to $\sT$ and we assume that $\epz\colon \bL\bR \rtarr \id$ is an isomorphism.  We also assume given a monad $\bD$ on 
$\sV$ together with a natural isomorphism $\om\colon \bD \bR \rtarr  \bR \bC$ of functors $\sT \rtarr \sV$ that is compatible with the units and products, $\et$ and $\mu$, of $\bC$ and $\bD$ in the sense that the following diagrams commute.\footnote{In the language of \cite[Definition 14.1]{Rant1}, these diagrams say that $(\bR,\om)$ is a lax map of monads $\bC\rtarr \bD$.}
\[ \xymatrix{
& \bR  \ar[dl]_{\et \bR}  \ar[dr]^{\bR \et} & \\
\bD \bR \ar[rr]_-{\om} &  & \bR\bC \\}
\ \ \ \ \ \ \
\xymatrix{
\bD\bD\bR \ar[r]^-{\bD\om} \ar[d]_{\mu\bR} & \bD\bR\bC \ar[r]^{\om \bC} & \bR\bC\bC \ar[d]^{\bR\mu}\\
\bD\bR \ar[rr]_-{\om} & & \bR\bC\\}
\]
\end{asscom}

\begin{rem}\label{weakcom} The condition that $\epz\colon \bL\bR \rtarr \id$ is an isomorphism plays an important role in our main examples.  We say we have a weak composite adjunction context when $\epz$ is only a weak equivalence.  The isomorphism condition is familiar categorically.  It means that $\sT$ is a reflective subcategory of $\sV$.  In particular, $\bR$ is full and faithful, the categorical monad $\bR\bL$ is idempotent (its $\mu$ is an isomorphism), and the adjunction is monadic.  Intuitively, $\bD$-algebras in $\sV$ are more general than $\bC$-algebras in $\sT$, so that a recognition principle for the new context is more general than the one for the original context.   Nevertheless, in our examples, we will show how to transform the more general context into the original one.  See \autoref{reduce}.
\end{rem}

The following diagram gives the picture.  It is a composite analog of \autoref{adj5} .  
\begin{equation}\label{adj6} 
\xymatrix{
\sV \ar@<.5ex>[rr]^{\bL}  \ar@<.5ex>[ddrr]^{\bF_{\bD}} & & \sT  \ar@<.5ex>[ll]^{\bR}    \ar@<.5ex>[rr]^{\SI}  \ar@<.5ex>[dr]^{\bF_{\bC}}& & \sS   \ar@<.5ex>[ll]^{\OM}   \ar@<.5ex>[dl]^{\OM_{\bC}}   
\ar@/^4pc/@<.6ex> [ddll]^-{(\bR\OM)_{\bD}}        \\
& & &   \ar@<.5ex>[ul]^{\bU_{\bC}}  \bC[\sT]   \ar@<.5ex>[dl]^{\bR}  \ar@<.5ex>[ur]^{\SI_{\bC}}& \\
& & \ar@<.5ex>[uull]^{\bU_{\bD}}  \bD[\sV] \ar@{-->}@<.5ex>[ur]^{\bL} \ar@/_4pc/@<.6ex> [uurr]^-{(\SI \bL)_{\bD}}& &  \\}
\end{equation}

We emphasize that the dotted arrow $\bL$ usually does not exist, but our context makes sense nevertheless.  Its companion arrow $\bR$ does always exist since we shall see in (i) of \autoref{formalMT2} that $\bR$  takes $\bC$-algebras in $\sT$ to $\bD$-algebras in $\sV$.  Therefore the arrow $(\bR\OM)_{\bD}=\bR\com \OM_{\bC}$ exists. This proves the following result.

\begin{lem}\label{com1} \autoref{ass1} holds with $(\SI, \OM)$ and $\bC$ replaced by $(\SI\bL, \bR\OM)$ and $\bD$.  
\end{lem}

We call this the composite adjunction context.
 We have maps $\al$ and $\be$ as in \autoref{MON} for this context, hence \autoref{twist} and \autoref{keyadj} construct the left adjoint, denoted $(\SI\bL)_{\bD}$, of $(\bR\OM)_{\bD}$.  

\begin{rem}\label{sillybit}   When the dotted arrow left adjoint $\bL$ to $\bR$ exists, the curved arrows might as well be erased since we then have both
$$(\bR\OM)_{\bD} = \bR \com \OM_{\bC} \ \ \text{and} \ \  (\SI\bL)_{\bD} = \SI_{\bC}\com \bL.$$
\end{rem}

{ We assume  that  Assumptions B through E of  \autoref{CONTEXT2} hold for $(\SI,\OM)$ and $\bC$, so that we have the conclusions they imply. We discuss these assumptions in the present context  in Sections  \ref{BC} and \ref{DE}, answering the following natural question.

\begin{quest}\label{Quest} Assuming that $(\SI,\OM)$ and $\bC$ satisfy Assumptions B through E, when do $(\SI \bL, \bR\OM)$ and $\bD$ also satisfy them?
\end{quest}

We relate the monads $\bC$ and $\bD$ and their algebras formally in \autoref{MONCOM}, and we discuss the subcategories $\sV_s$ and $\sV_{ss}$ of special and strictly special  objects of $\sV$ in \autoref{SecSpec},   We summarize our conclusions in \autoref{compmach}.  Finally, we show in \autoref{silly} how an example of our (weak) composite adjunction context applies to transform topological examples of \autoref{ass1} to simplicial ones.

\subsection{The homotopical context and Assumptions B and C}\label{BC}

We require the following analog of \autoref{ass2}. In practice, most parts follow directly from their analogs for our original adjunction $(\SI,\OM)$.

\begin{asscom}\label{ass7}  We assume that $\sV$, like $\sS$ and $\sT$, is cocomplete, has a standard notion of homotopy, and has a class of weak equivalences satisfying the two out of three property. We say that a map $f$  in $\bD[\sV]$ is a weak equivalence if $\bU_{\bD}f$ is a weak equivalence in $\sV$.   We assume that a map $f$ in $\sT$ is a weak equivalence if and only if $\bR f$ is a weak equivalence in $\sV$ and that the functor $\bF_{\bD}$ (hence also $\bD$) preserves weak equivalences, at least under restriction to good objects as in \autoref{bspt1}. 
We assume the following statements about simplicial objects in $\sV$.  
\begin{enumerate}[(i)]
\item  There is a (levelwise) realization functor $\bT \colon s\sV\rtarr \sV$ and it is a left adjoint.
\item  The functor $\bT $  on  $s\sV$ preserves homotopies.
\item  The functor $\bT $  on  $s\sV$ preserves weak equivalences between Reedy cofibrant objects.
\item  Realization commutes with the (left adjoint) functor $\bL$. That is, for $X_*$ in $s\sV$, there is a natural isomorphism
$\xymatrix@1{  \bL \bT X_* \iso \bT \bL_* X_*.}$
\end{enumerate}
\end{asscom}

\begin{rem}\label{Lstinks} In contrast to \autoref{ass2}, where $\SI$ was assumed to preserve weak equivalences, we do not assume that the functor $\bL$ preserves weak equivalences.  It will do so in the categories of operators context.  It will not do so in the orbital presheaf context, but it will do so there on restriction to cofibrant objects in view of \autoref{Lss} and \autoref{Gcellwonder}.
\end{rem}

Similarly, we require the following analog of \autoref{ass3}.

\begin{asscom}\label{ass8} 
\begin{enumerate}[(i)]
\item Realization commutes with the functor $\bD$. That is,  for $X_*$ in $s\sV$, there is a natural isomorphism
$\nu\colon  \bT \bD_* X_* \rtarr \bD \bT X_*$ such that the following diagrams commute.
$$ \xymatrix{
& \bT X_* \ar[dl]_-{\bT \et_*}  \ar[dr]^{\et} &        & & 
\bT \bD_*\bD_* X_* \ar[d]_{\bT \mu_*}   \ar[r]^-{\nu} &\bD \bT \bD_* X_*  \ar[r]^-{\bD\nu} & \bD\bD \bT X_* \ar[d]^{\mu}\\
\bT \bD_*X_* \ar[rr]_-{\nu} & & \bD \bT X_* & & \bT \bD_* X_*  \ar[rr]_-{\nu}  &  & \bD \bT X_* \\}
$$
Therefore $\bT \bD_*X_*$ is a $\bD$-algebra.  
\item For $X_*\in s\sT$,  the natural map $\ga_{\bR}\colon \bT \bR_* X_* \rtarr \bR \bT X_*$, which is the adjoint of the isomorphism
 $$ \xymatrix@1{\bL \bT \bR_*X_* \iso \bT \bL_*\bR_* K_* \ar[r]^-{\bT \epz_*} & \bT X_*,\\}$$
 is also an isomorphism.
 \end{enumerate}
\end{asscom}

\begin{rem}\label{different}  In our space level examples in Sections  \ref{catop} and \ref{orbpre}, $\ga_{\bR}$ is an isomorphism for entirely different non-formal reasons.  In  \autoref{catop}, $\bR$ is given by cartesian powers, and realization commutes with products.  In  \autoref{orbpre}, $\bR$ is given by passage to fixed points, and realization commutes with passage to fixed points for Reedy cofibrant simplicial $G$-spaces.  
\end{rem}

Composite versions of  Assumptions \ref{ass2}(v) and \ref{ass3}(ii) will be given in \autoref{tortuous}.  Therefore the following result holds in the contexts we consider in this paper.

\begin{lem}\label{ass2yes}   By composition, \autoref{ass7} and \autoref{tortuous} together with Assumptions \ref{ass2} and \ref{ass3} for the adjunction $(\SI,\OM)$ and the monad $\bC$ ensure that Assumptions \ref{ass2} and \ref{ass3} hold for the composite adjunction $(\SI\bL, \bR\OM)$ and the monad $\bD$, except that $\SI\bL$ need not preserve weak equivalences in general.
\end{lem}

We will come back to Assumptions \ref{ass4} and \ref{ass5} in \autoref{DE}.
}
   
\subsection{Relations between monads $\bC$ in $\sT$ and $\bD$ in $\sV$}\label{MONCOM}

We ignore $(\SI,\OM)$ in this subsection but assume given the adjunction $(\bL,\bR)$ such that $\epz\colon \bL\bR \rtarr \id$ is an isomorphism.  We relate monadic data in $\sT$ and $\sV$, assuming throughout that we are given monads $\bC$ on $\sT$ and $\bD$ on $\sV$ together with an isomorphism $\om\colon \bD\bR\rtarr \bR\bC$ that satisfies the compatibility diagrams of  \autoref{ass7}.  

In the special case of categories of operators, results like these were first observed in \cite[\S6]{MT} and were later elaborated in \cite[Appendix A]{Rant2}, but our statements here are adapted to our present general context.  The proofs are formal inspections of definitions and straightforward diagram chases.  These results are purely categorical and have nothing to do with group completion or the approximation theorem.  Note that part (ii) is in part a specialization of \autoref{Fied7}, but now using that $\epz$ is an isomorphism.

\begin{prop}\label{formalMT1}  The following conclusions relate the monads $\bC$ and $\bD$. 
\begin{enumerate}[(i)]  
\item The functor $\bL \bD \bR$ on $\sT$ is naturally isomorphic to $\bC$ via the composite
\[ \xymatrix@1{\bL \bD \bR \ar[r]^-{\bL\om} & \bL\bR \bC \ar[r]^-{\epz\bC} & \bC \\}  \]
and therefore  inherits a monad structure from that of $\bC$.  Moreover, under this isomorphism, the following compatibility diagrams commute.
\[  \xymatrix{
\bL\bR \ar[d]_{\epz} \ar[r]^-{\bL\et\bR} & \bL\bD\bR \ar[d]^{\iso}\\
\id \ar[r]_-{\et}  & \bC\\}
\ \ \ \ \ \ \ 
\xymatrix{
\bL\bD\bD\bR \ar[d]_{\bL\mu \bR} \ar[r]^-{\bL\bD\et\bD\bR} & \bL\bD\bR\bL\bD\bR \ar[r]^-{\iso} & \bC\bC \ar[d]^{\mu}\\
\bL\bD\bR \ar[rr]_-{\iso} & & \bC\\}  \]
\item The functor $\bR \bC \bL $ on $\sV$ is a monad with product and unit induced from those
of $\bC$ via the composites
\[ \xymatrix@1{ \bR \bC \bL \bR \bC \bL \ar[r]^-{\bR\bC \epz} & \bR \bC \bC \bL  \ar[r]^-{\bR \mu \bL } & \bR \bC \bL \\}  
\ \ \text{and} \ \ \xymatrix@1{ \id \ar[r]^{\et} & \bR \bL  \ar[r]^-{\bR \et \bL } & \bR \bC \bL.  \\} \]
Moreover, the isomorphism  $\bR\bC\epz\colon \bR\bC\bL\bR \rtarr \bR\bC$ is compatible with the units and products of 
$\bR\bC\bL$ and $\bC$ in the sense specified in \autoref{ass7}. 
\item   The composite
$$\xymatrix@1{\io: \bD  \ar[r]^-{\bD\et} & \bD  \bR \bL  \ar[r]^{\om \bL}_{\iso} &  \bR \bC \bL\\} $$
is a morphism of monads in $\sV$. 
\end{enumerate}
\end{prop}

\begin{rem}\label{puzzle}  This result  says in particular that the ``conjugate monad'' $\bR\bC \bL$ is the terminal monad in $\sV$ such that  $\bD\bR$ is compatibly isomorphic to  $\bR\bC$.  There usually are other such monads that are more relevant to the applications.   In the category of operators  context, others have long been known and are standard.  In the orbital presheaf context, others are difficult to construct.   In the sequel \cite{KMZ2}, we will rework a construction of Costenoble and Waner \cite{CW} to obtain such $\bD$ in that case.  Those $\bD$ are the examples that lead to applications.   
\end{rem}

\begin{prop}\label{formalMT2}  The following conclusions relate $\bC$-algebras to $\bD$-algebras.  Let $X$ be in $\sT$ and $Y$ be in  $\sV$.  
\begin{enumerate}[(i)]  
\item If $X$ is a $\bC$-algebra with action $\tha\colon \bC X\rtarr X$, then $\bR X $ is a $\bD$-algebra with action 
the composite 
$$\xymatrix@1{\bD \bR X\ar[r]^-{\om} & \bR\bC X\ar[r]^-{\bR  \tha} & \bR X\\}$$
and  $\epz\colon \bL\bR X \rtarr  X$ is an isomorphism of $\bC$-algebras.
\item  If $\bR X$ is a $\bD$-algebra with action  $\ps\colon \bD\bR X \rtarr \bR X$, then $X$ is a $\bC$-algebra with action  
$$\bL\ps \colon \bC X \iso \bL \bD \bR X \rtarr \bL  \bR X \iso X$$
and the action $\ps$ factors as the composite
$$ \xymatrix{
\bD \bR X  \ar[r]^-{\et} &  \bR\bL \bD \bR X  \iso \bR\bC X  \ar[r]^-{\bR\bL \ps} & \bR X.\\} $$
\item If  $\bD \iso \bR\bC\bL$ and $Y$ is a $\bD$-algebra with action  $\tha\colon \bD Y \rtarr Y$, then $\bL Y$ is a $\bC$-algebra with action  
$$ \bL \tha\colon \bC \bL Y \iso \bL\bR \bC \bL Y \rtarr \bL Y.$$
Therefore the dotted arrow $\bL$ in \autoref{adj6} exists when $\bD = \bR\bC\bL$.
\end{enumerate}
\end{prop}

\begin{prop}\label{formalMT22}   If $(F,\la)$ is a $\bC$-functor in some category $\sZ$, then $F\bL \colon \sV \rtarr \sZ$ is an
$\bR \bC \bL $-functor in $\sZ$ with action the composite
$$\xymatrix@1{F\bL \bR \bC \bL  \ar[r]^{F\epz\bC\bL} & F\bC \bL  \ar[r]^{\la \bL} & F\bL.\\} $$
Therefore, by pullback, $F\bL$ is a $\bD$-functor in $\sZ$ with action the composite
\[ \xymatrix@1{ F\bL \bD\ar[r]^-{F\bL \bD\et} 
&  F\bL \bD  \bR \bL  \ar[r]^-{F\bL \om\bL} &  F\bL \bR  \bC \bL  \ar[r]^{F\epz\bC\bL} & F\bC \bL  \ar[r]^-{\la \bL } & F\bL .\\} \]
\end{prop}

\subsection{Special and strictly special objects of $\sV$}\label{SecSpec}

We will see in \autoref{Segal}  that the first part of the following definition is standard in the context of categories of operators.   The second part will play a central role in both that context and the context of orbital presheaves.

\begin{defn}\label{special} We say that an object  $Y$ of $\sV$ is {\em special} if the unit 
$\et\colon Y\rtarr \bR\bL Y$ is a weak equivalence, and we let $\sV_s$ be the full subcategory of special objects of $\sV$. 
We say that $Y$ is {\em strictly special} if $\et$ is an isomorphism, and we let $\sV_{ss}$ be the full subcategory of strictly special objects of $\sV$.  \end{defn}

 \autoref{ass7} and the naturality of $\et$ imply the following conclusion. 

\begin{lem}\label{Lss}  The functor $\bL$ preserves weak equivalences when restricted to special objects of $\sV$.
\end{lem}

The condition that  $\epz\colon \bL\bR \rtarr \id$ is an isomorphism has the following implication.  

\begin{lem}\label{Rss}  The functor $\bR$ takes objects of $\sT$ to strictly special objects of $\sV$.   Therefore 
$(\bL,\bR)$ restricts to an adjoint equivalence relating $\sV_{ss}$ to $\sT$. \end{lem}
\begin{proof}
By the triangle identity, the following composite is the identity for  $X\in \sT$.
$$\xymatrix@1{ 
\bR X \ar[r]^-{\et}  &  \bR\bL\bR X \ar[r]^-{\bR\epz}_-{\iso} & \bR X\\}
$$
Therefore $\et\colon \bR X \rtarr \bR\bL \bR X$ is an isomorphism. 
\end{proof}

\begin{rem}\label{puzzle2}
In our examples, we will show, non-trivially, that it suffices to focus on strictly special objects of $\sV$.  Then \autoref{formalMT3} below, which is implied by the following lemma and Propositions \ref{formalMT1} and 
\ref{formalMT2}, always justifies taking $\bD$ to be $\bR \bC\bL$ when applying the machine. This holds even though the examples of interest come from actions of quite different monads $\bD$, as pointed out in
\autoref{puzzle}.
\end{rem}

\begin{lem}\label{YDY}   If $Y\in \sV$ is strictly special, then $\bD Y$ is strictly special.
\end{lem}
\begin{proof}
Consider the following commutative naturality diagram.
$$\xymatrix{
\bD Y \ar[r]^-{\bD \et}_-{\iso}  \ar[d]_{\et} & \bD\bR\bL Y \ar[r]^-{\om}_{\iso} \ar[d]^{\et}& \bR\bC \bL Y \ar[d]^{\et} \\
\bR\bL\bD Y \ar[r]_-{\bR\bL\bD \et}^-{\iso} & \bR\bL\bD \bR\bL Y  \ar[r]_-{\om}^-{\iso} & \bR\bL\bR\bC\bL Y \\} 
$$
By the triangle identity, the rightmost vertical arrow $\et$ is an  isomorphism with inverse $\bR\epz$.  Therefore the other  vertical arrows $\et$ are also isomorphisms. 
\end{proof}

Restricting to strictly special objects,  we can fill in the dotted arrow $\bL$ of \autoref{adj6}, and it then preserves weak equivalences by \autoref{Lss}.

\begin{cor}\label{formalMT3}    Restricted to strictly special objects, the  monads $\bD$ and $\bR\bC\bL$ are isomorphic and take values in  strictly special objects.  The  adjoint equivalence $(\bL, \bR)$ between 
$\sV_{ss}$ and  $\sT$ restricts to an adjoint equivalence between  $\bD[\sV_{ss}]$ and $\bC[\sT]$. 
\end{cor}

Using $\sV_{ss}$, we can complete our discussion of the composite versions of  Assumptions \ref{ass2}(v) and \ref{ass3}(ii).

\begin{lem}\label{DgammaR}  If $\ga_{\bR}$ is an isomorphism and $X_*$ is a simplicial $\bC$-algebra in $\sT$, then $\bT \bR X_*$ is a 
$\bD$-algebra in $\sV_{ss}$ such that  $\ga_{\bR}$ is an isomorphism of $\bD$-algebras.
\end{lem}  
\begin{proof}
By \autoref{Rss}, $\bR$ takes values in $\sV_{ss}$.  Visibly $\et\colon \bT \bR X_* \rtarr \bR\bL \bT \bR X_*$ is the composite
$$ \xymatrix@1{\bT \bR X_*  \ar[r]^-{\bT \et}  & \bT \bR\bL \bR X_* \ar[r]^-{\ga_{\bR}} & \bR \bT \bL \bR X_* \iso \bR\bL \bT \bR X_*\\}$$
and is therefore an isomorphism.  Thus $\bT \bR X_*$ is in $\sV_{ss}$.  It is a $\bD$-algebra such that $\ga_{\bR}$ is a map of 
$\bD$-algebras by use of the isomorphism $\nu\colon \bT \bD \bR X_* \rtarr  \bD \bT \bR X_*$ of \autoref{ass8}(i).
\end{proof}

\begin{lem}\label{tortuous}  Let $X_* \in s\sS$ and consider the natural map  $\bT \bR\OM X_* \rtarr \bR\OM \bT X_*$, namely the composite
\begin{equation}\label{gamma2} \xymatrix@1{ \bT \bR\OM X_* \ar[r]^-{\ga_{\bR}} &  \bR \bT \OM X_*  \ar[r]^{\bR\ga} & \bR\OM \bT X_*. \\} 
\end{equation}
If $\ga_{\bR}$ is an isomorphism, then it is an isomorphism of $\bD$-algebras in $\sV_{ss}$.  The map $\bR\ga$ is a map of $\bD$-algebras in $\sV$, and it is a weak equivalence if $X_*$ is levelwise $\OM$-connective.
\end{lem}
\begin{proof}  The first statement is immediate from \autoref{DgammaR}. For the second statement, $\ga$ is a map of $\bC$-algebras
by \autoref{ass3}(ii),  hence $\bR \ga$ is a map of $\bD$-algebras by (i) of \autoref{formalMT2}(i); $\bR\ga$ is a weak equivalence since $\ga$ is a weak equivalence by \autoref{ass2}(v) and $\bR$ preserves weak equivalences. 
\end{proof}

We conclude this section by describing a natural equivalence between the homotopy categories of $\sV_s$ and  $\sV_{ss}$.   This result is discussed in more detail for categories of operators, which is where we will use it, in \cite{Rant2}, so we will be brief.   Observe that, by \autoref{formalMT22}, $\bC \bL$ is a $\bD$-functor with values in $\bC[\sT]$.  

\begin{prop}\label{ssequiv}  For a $\bD$-algebra $Y$, define $\overline{Y}_{\bC}$ to be the strictly special $\bD$-algebra 
$$ \overline{Y}_{\bC} = \bR B(\bC \bL, \bD, Y).$$
If  $Y$ is special, then $Y$ is naturally weakly equivalent  as a $\bD$-algebra to $\overline{Y}_{\bC}$.  For a $\bC$-algebra $X$, $\bR X$ is weakly equivalent as a $\bD$-algebra to $\overline{\bR X}_{\bC}$.
\end{prop}
\begin{proof} Define $\overline{\et}$ to be either of the equal composites in the following diagram.
$$ \xymatrix{
\bD \ar[r]^-{\bD\et}  \ar[dr]_{\et} & \bD\bR\bL \ar[r]^-{\et \bD\bR\bL} & \bR\bL\bD\bR\bL \ar[r]^{\iso} 
&\bR\bC\bL\\
& \bR\bL \bD \ar[ur]_{\bR\bL\bD\et} & & \\}
$$
We then have the following natural maps.
$$\xymatrix{
Y & B(\bD,\bD, Y) \ar[l]_-{\ze} \ar[r]^-{B(\overline{\et},\id, \id)} & B(\bR\bC\bL, \bD, Y) \ar[r]^-{\ga_{\bR}} & 
\bR B(\bC\bL, \bD, Y) = \overline{Y}_{\bC} \\}
$$
Here $\ga_{\bR}$ is defined in \autoref{ass8}(ii) and discussed in \autoref{different}.  In the context of categories of operators, $\ga_{\bR}$ is a natural isomorphism because realization commutes with products.  As  usual, $\ze$ is an equivalence.  When $Y$ is special, we can deduce that $\overline{\et}$ is a weak equivalence; the diagrams in the proof of \autoref{YDY} are relevant.  Then $B(\overline{\et},\id, \id)$ is a weak equivalence by \autoref{ass7}(iii).  The weak equivalence of the last statement is the composite
$$\xymatrix{
 \overline{\bR X}_{\bC} = \bR B(\bC \bL, \bD, \bR X)\ar[r]^-{\bR B(\id,\overline{\et},\id)} & 
 \bR B(\bC\bL, \bR\bC\bL, \bR X) \iso \bR B(\bC, \bC,X) \ar[r]^-{\ze} & \bR X. \\}
 $$
The isomorphism is a specialization of the easy \autoref{barcom1} below. 
\end{proof}

\begin{rem}\label{reduce}
We can apply the equivalences of \autoref{ssequiv} to transform \autoref{adj6} to a diagram in which the category $\sV$ is  replaced by $\sV_{ss}$.  In the transformed diagram $\bL$ exists and preserves weak equivalences.  Then the units of both pairs $(\bL,\bR)$ are  given by weak  equivalences.    In effect, up to natural weak equivalence, the transformed diagram entirely reduces the composite adjunction context to the original context of the small right triangle.  However, we do not know how to effect the corresponding reduction in the richer multiplicative context of \autoref{Mult2}. 
\end{rem} 

\subsection{Assumptions D and E in the composite adjunction context}\label{DE}
We shall be a little informal here since the verifications differ in ways that we shall indicate briefly but that will become clearer when we explain the relevant concrete specializations of our composite adjunction context.  The framework is the same in all contexts, but the specifics of group completion and realization vary.  The approximation theorem, \autoref{ass5}, is the main point, and that works the same way in all contexts.

Our discussion of Assumptions  \ref{ass4} and \ref{ass5} focuses on the following diagram.  Write $\al_{com}$ and 
$\vartheta_{com}$  for the maps $\al$ and $\vartheta$ in the composite adjunction context.  The following diagram describes them in terms of $\al$ and $\vartheta$ in our original context of \autoref{ass1}.  We write $\et$ generically for the units of both adjunctions and monads.

\begin{equation}\label{compal}
\xymatrix{
 \bD \ar[r]^-{\bD \et}  \ar[d]_{\bD\et}  \ar@/^2pc/[rr]^-{\al_{com}}& \bD\bR\OM\SI\bL \ar[r]^-{\vartheta_{com}} \ar@{=}[d] 
 &  \bR\OM\SI \bL \ar@{=}[dd]\\
 \bD\bR\bL \ar[r]^-{\bD\bR\et} \ar[d]^{\iso}_{\om} & \bD\bR\OM\SI\bL \ar[d]_{\iso}^{\om} & \\
 \bR\bC \bL \ar[r]^-{\bR\bC\et}  \ar@/_2pc/[rr]_-{\bR\al}& \bR\bC\OM\SI \bL \ar[r]^-{\bR\vartheta} & \bR\OM\SI\bL \\}
 \end{equation}
The upper left square applies $\bD$ to the unit of a composite adjunction, the lower left square is a naturality diagram, and the right rectangle exhibits the definition of $\vartheta_{com}$ in terms of $\vartheta$.   

When applied to special objects of $\sV$, the top left vertical arrow $\bD\et$ is a weak equivalence, and it is an isomorphism when applied to strictly special objects.  Due to the appearance of $\bD\et$ in the diagram, we can only expect our recognition principle and characterization theorems to hold when we restrict attention to special objects.  This restriction also effects our discussion of group completions.

\begin{rem}\label{ass4hazy} In our examples, the objects of $\sV$ are functors  $\sJ\rtarr \sW$ for some small domain category $\sJ$ and some homotopical target category $\sW$.  In \autoref{catop},  $\sW$ is the category $\sT$ of our original adjunction.   However, in \autoref{orbpre}, our original $\sT$ is based $G$-spaces for a finite group $G$, whereas $\sW$ is based spaces. Therefore, even though both examples fit perfectly into the general framework here, they are very different in practice.  In the former, interest focuses on special objects.  In the latter, interest focuses on non-special objects, but in \autoref{Gcellwonder} we shall show model theoretically that a cellular approximation functor $\GA$ can be used to approximate all objects by strictly special objects.  Nothing like that holds in the categories of operators context; see \autoref{light}.
\end{rem}

\begin{defn}\label{ass4hazy2} Suppose that the objects of $\sV$ are functors $\sJ\rtarr \sW$  and  that grouplike objects of $\sW$ and group completions in $\sW$ have been defined using subcategories $\sH$ and $\sH_{gp}$ of $\sW$, as in \autoref{screwy}.  We transfer these notions to $\sV$ levelwise.  We define $\sH$ and $\sH_{gp}$  in $\sV$ to be the subcategories of functors $Y$ such that each $Y(j)$ is in $\sH$ or is in
$\sH_{gp}$.  We say that a functor $Z$ in $\sV$ is grouplike if each $Z(j)$ is grouplike and we say that a natural transformation  $f\colon Y \rtarr Z$ is a group completion if each $f\colon Y(j) \rtarr Z(j)$ is a group completion.  With these definitions, it is immediate in our examples that $\bR\colon \sT\rtarr \sV_s$ preserves grouplike objects and group completions.

It is also immediate in our examples that the forgetful functor $\bD[\sV]\rtarr \sV$ takes (levelwise) values in $\sH$.  Assume that $\sW$ has a group completion functor $(\bG,g)$ as in \autoref{screwy2}.   In the categories of operators context of \autoref{catop}, where $\sW = \sT$ is the target of $\bL$, we restrict to special objects and define $(\bG_{\sV_s},g_{\sV_{s}})$ by setting
$$ \bG_{\sV_s}=\bR\bG\bL  \colon \bD[\sV_s] \rtarr \sH_{gp} $$
and letting $g_{\sV_{s}}\colon Y\rtarr \bG_{\sV_s}Y$ be the composite 
$$  \xymatrix@1{Y \ar[r]^-{\et}  & \bR\bL Y \ar[r]^-{\bR g\bL} & \bR\bG\bL Y.\\}$$
In the orbital presheaves context of \autoref{orbpre}, where $\sW$ is not the target $\sT$ of $\bL$, we restrict to strictly special objects and define $(\bG_{\sV_{ss}},g_{\sV_{ss}})$  by setting
  $$(\bG_{\sV_{ss}}Y)(j) = \bG (Y(j))$$
and letting $g_{\sV_{ss}}\colon Y \rtarr \bG_{\sV_{ss}}Y$ be given levelwise by $g\colon Y(j) \rtarr \bG Y(j)$ for $j\in \sJ$.
\end{defn} 

\begin{lem}\label{ass4com}  With the definitions just given, \autoref{ass4} for  our composite adjunction $(\bR\OM,\SI\bL)$ follows from \autoref{ass4} for our original adjunction $(\SI,\OM)$ in both the categories of operators and orbital presheaves contexts, provided we restrict to the category $\sV_s$ in the former case and to the category $\sV_{ss}$ in the latter.
\end{lem}
\begin{proof}[Sketch]  We see that $(\bR\OM)_{\sD}$ takes values in $\sV_{ss}$ by \autoref{Rss}.  It takes grouplike values since  $\OM_{\bC}$ takes grouplike values in $\sT$ and $\bR$ takes grouplike objects in $\sT$ to grouplike objects in $\sV$.
\end{proof}

\begin{lem}\label{ass5com}
On special objects, \autoref{ass5} (the approximation theorem) holds for  the composite adjunction $(\bD\bL, \bR \OM)$ and the monad $\bD$. 
\end{lem}
\begin{proof}
This follows directly from \autoref{ass5} for our original adjunction $(\SI,\OM)$ and the monad $\bC$ via the diagram \autoref{compal}.   The implication uses that $\bR$ preserves grouplike objects and group completions, as follows directly from the definitions in the contexts to which we shall specialize. 
\end{proof}

\subsection{The composite adjunction machine}\label{compmach}

Replacing the adjunction $(\SI,\OM)$ with the composite adjunction $(\SI \bL, \bR \OM)$, Theorems \ref{recprin} and \ref{adjequiv} specialize schematically as follows.   As explained above, we must restrict attention to special objects of $\sV$ in the context of categories of operators.    We can and will restrict to strictly special objects in the context of orbital presheaves.  
 
\begin{thm}\label{recprincom}  There is a functor $\mathrm{Bar} \colon \bD[\sV] \rtarr \bD[\sV]$, written $Y\mapsto \overline{Y}$, and a natural equivalence  $\ze\colon \overline{Y} \rtarr Y$.  If $Y$ is special, the unit $\et_{\bD}\colon \overline{Y} \rtarr (\bR\OM)_{\bD} (\SI\bL)_{\bD} \overline{Y}$  is a group completion and is therefore an equivalence if $Y$ is grouplike.\end{thm}

\begin{thm}\label{adjequivcom}  $\big((\SI\bL)_{\bD},(\bR\OM)_{\bD}\big)$ induces an adjoint equivalence from the homotopy category of special grouplike $ \bD$-algebras in $\sV$ to the homotopy category of $\bR\OM$-connective objects of $\sS$.
\end{thm}

{In our examples, the $\bR \OM$-connective objects in $\sS$, as defined by specialization of \autoref{LAcon}, turn out to be the same as the 
$\OM$-connective objects.}

The idea is illuminated by the following diagram, which is the composite adjunction special case of the diagram
\autoref{forcon2}.  Here $Y$ is in $\bD[\sV]$ and must be special for the maps in the square to have the properties listed in our Assumptions.

\begin{equation}\label{adj}
\xymatrix{
 Y & \ar[l]_-{\ze}  \overline Y  \ar[r]^-{B\al_{com}} \ar[d] _{\et_{\bD}} &   B((\bR\OM)_{\bD}\SI\bL,\bD, Y) \ar[d]^{\ga} \\
    &  (\bR\OM)_{\bD}(\SI\bL)_{\bD} \overline Y  \ar[r]^-{\iso} &  (\bR\OM)_{\bD} B(\SI\bL,\bD, Y) \\}
\end{equation}

The following result, which starts from the map of monads $\io$ in (iii) of \autoref{formalMT1}, shows that, up to homotopy, when considering special objects of $\sV$ it is in principle possible to restrict attention to the monad $\bR\bC\bL$, rather than using a general monad $\bD$ satisfying \autoref{ass7}, in the results above.

\begin{prop}  Define $\io^*\colon  \bR\bC\bL[\sV_s] \rtarr \bD[\sV_s]$ by pullback of actions along $\io\colon \bD\rtarr \bR\bC\bL$.  Then $\io^*$ induces a equivalence on passage to homotopy categories.
\end{prop}
\begin{proof} We define a derived inverse $\io_*\colon \bD[\sV_s] \rtarr  \bR\bC\bL[\sV_s]$  by 
$$\io_* X =  B(\bR\bC\bL, \bD, X) $$
for $\bD$-algebras $X$.  Here the right action of $\bD$ on $\bR\bC\bL$ is the composite
$$ \xymatrix@1{\bR\bC\bL\bD \ar[r]^-{\bR\bC\bL\io} & \bR\bC\bL\bR\bC\bL \ar[r]^-{\mu} & \bR\bC\bL\\}$$
On special objects,  $\io \colon \bD\rtarr \bR\bC\bL$ is a weak equivalence.  The following diagrams display 
weak equivalences $\io^*\io_* \htp \id$ and $\io_*\io^* \htp \id$.
$$ \xymatrix@1{ X  & B(\bD, \bD, X) \ar[l]_-{\ze}  \ar[rr]^-{B(\io,\id,\id)} & & \io^*B(\bR\bC\bL,\bD, X)\\}$$
$$ \xymatrix@1{ B(\bR\bC\bL,\bD, \io^* Y) \ar[rr]^-{B(\id, \io, \id)} &  & B(\bR\bC\bL,\bR\bC\bL, Y)  \ar[r]^-{\ze} &  Y\\}$$
Here $X$ is a special $\bD$-algebra and $Y$ is a special $\bR\bC\bL$-algebra. 
\end{proof} 

Part (iii) of \autoref{formalMT1}, together with Propositions \ref{formalMT2} and \ref{formalMT22}, also leads to the following comparison of bar constructions.   It shows that our bar constructions constructed from $\bC$ are special cases of those constructed from  $\bR\bC\bL$.

\begin{cor}\label{barcom1}  For $\bC$-functors $F$ and $\bC$-algebras $X$, 
$$  B(F,\bC, X) \iso B(F\bL, \bR\bC\bL, \bR X) $$
\end{cor}
\begin{proof} Used iteratively, the isomorphism $\epz\colon \bL\bR \rtarr \id$ induces isomorphisms
$$F\bL(\bR\bC\bL)^q\bR X \rtarr  F \bC^q X$$
on $q$-simplices, and these isomorphisms commute with faces and degeneracies.
\end{proof}

Restricting to strictly special objects of $\sV$,   we have $\bD\iso \bR\bC\bL$ by \autoref{formalMT3}.  Here Propositions \ref{formalMT2} and \ref{formalMT22} lead to commutation results relating bar constructions constructed from $\bC$ to bar constructions constructed from $\bD$.

\begin{cor}\label{barcom2} For $\bC$-functors $F$ and strictly special $\bD$-algebras $Y$, 
$$  B(F,\bC, \bL Y) \iso B(F\bL, \bD, Y) $$
\end{cor}
\begin{proof} Used iteratively, the isomorphism $\et\colon \bD^iY \rtarr \bR\bL \bD^i Y$ given by \autoref{formalMT3}
induces isomorphisms
$$ F\bL \bD^q Y \rtarr F\bL (\bD\bR\bL)^q Y \iso F(\bL\bD\bR)^q \bL Y \iso F \bC^q \bL Y $$
on $q$-simplices, and these isomorphisms commute with faces and degeneracies.
\end{proof}

\begin{cor}\label{Iforgot}  If $Y$ is a strictly special $\bD$-algebra, then 
$$B(\bC,\bC,\bL Y)\iso \bL B(\bD,\bD,Y)$$
In particular, for a $\bC$-algebra $X$, 
$$B(\bC,\bC,X) \iso B(\bC,\bC,\bL \bR X)\iso \bL B(\bD,\bD,\bR X)$$
and therefore, applying $\bR$,
$$ \bR \overline{X} = \bR B(\bC,\bC,X) \iso   B(\bD,\bD,\bR X) =\overline{\bR X}$$
\end{cor}
\begin{proof}
Taking $F = \bC$ and using the isomorphism
$$\xymatrix{\bL \bC \ar[r]^-{\bL\bC \et} & \bL\bC\bR\bL \ar[r]^{\iso} & \bD \bL\\} $$
of functors on $\sV_{ss}$, we see the first statement as the specialization
$$B(\bC,\bC,\bL Y) \iso B(\bC\bL,\bD, Y)\iso B(\bL\bD,\bD,Y)  \iso \bL B(\bD,\bD,Y) $$
of \autoref{barcom2}.   Since $\bR X$ is strictly special by \autoref{Rss} and $\bL$ commutes with realization,
the second statement follows.
\end{proof}

Using this, we show  in the rest of this section that the machine starting from $\bC$-algebras $X$ is equivalent to the machine starting from $\bD$-algebras $\bR X$.  The starting point is the following commutative diagram, which follows from \autoref{compal} by an adjoint diagram chase.

\begin{equation}\label{compbe}
\xymatrix{
\SI \bL \bD \bR \ar[r]^-{\be_{com} \bR} \ar[d]^{\iso}_{\SI\epz\bR} & \SI \bL \bR \ar[d]_{\iso}^{\SI\epz}\\
\SI \bC \ar[r]_{\be}  & \SI \\}
\end{equation}
 
\begin{lem}\label{compbe2}  The functors $(\SI\bL)_{\bD}\bR$ and $\SI_{\bC}$  on $\bC$-algebras in $\sT$ are isomorphic.
 \end{lem}
 \begin{proof}
This follows from \autoref{formalMT2}, the diagram \autoref{compbe}, and a comparison of coequalizer diagrams.
\end{proof}

\begin{prop}\label{bCbDcomp2}   For $\bC$-algebras  $X$ in $\sT$, the unit $\et\colon \id \rtarr \bR\bL$ induces a natural isomorphism
$$ \tilde{\et} \colon \bE_{\bD} \bR X  \rtarr \bE_{\bC} X.$$
\end{prop}
\begin{proof}  By (i) of \autoref{formalMT1}, the isomorphisms $\epz\colon \bL\bR\rtarr \id$ and $\om\colon \bD\bR \rtarr \bR\bC$ induce an isomorphism  $\bC\iso \bL\bD\bR$.  Using iterated insertions of $\et$ and \autoref{Rss}, 
we obtain natural isomorphisms
$$   \bL \bD^q\bR \rtarr (\bL\bD\bR)^q   $$
and therefore 
$$\SI \bL \bD^q\bR X \rtarr \SI (\bL\bD\bR)^q X \iso \SI\bC^q X. $$
A check of faces and degeneracies shows that these are the maps on $q$-simplices of a simplicial map. On realization, this map gives the isomorphism
$$  \tilde{\et} \colon  B(\SI\bL, \bD, \bR X) \rtarr B(\SI,\bC, X) $$
We must apply this with $X$ replaced by $\overline{X}$, and \autoref{Iforgot}  gives that $\overline{\bR X}$ is isomorphic to $\bR \overline{X}$.
\end{proof}

The following corollary relates to the missing dotted arrow $\bL$ in \autoref{adj6}.

\begin{cor}  When restricted to grouplike objects, $\bR\colon \bC[\sT]\rtarr \bD[\sV]$ induces an equivalence 
of homotopy categories.
\end{cor} 
\begin{proof}
By Theorems \ref{adjequiv} and \ref{adjequivcom}, $\bE_{\bC}$ and  $\bE_{\bD}$ induce equivalences from the homotopy categories of grouplike $\bC$-algebras and grouplike $\bD$-algebras to the homotopy category of connective objects of 
$\sS$, hence the conclusion is immediate from the previous proposition.  In effect, $\bR = \bE_{\bD}^{-1}\com \bE_{\bC}$ on our homotopy categories.
\end{proof}

\subsection{From a topological machine to a simplicial machine}\label{silly}

In this brief section, we illustrate the general theory of composite adjunctions by showing how it works to translate any topological context into a simplicial one.   Of course, our general theory could well start in a simplicial context.  However, in Part 1, our examples all started with $\sT$ being the category of based spaces or based $G$-spaces.  We focus on based spaces $\sT$, but with obvious modifications the following discussion also applies with $\sT$ replaced  by $G\sT$.  

We consider the following diagram.  The top right triangle is any given instance of \autoref{ass1} in which $\sT$ figures, and $(\bT,\bS)$ is the standard geometric realization and total singular complex adjunction between based simplicial sets and based spaces. We take $\bD = \bT \bC \bS$ as our monad on based simplicial sets. We do not have the dotted arrow $\bT$, but by \autoref{keyadj} we do have the left adjoint  $(\SI \bT)_{\bD}$.
\begin{equation}\label{adj6Too} 
\xymatrix{
s\mathrm{Set} \ar@<.5ex>[rr]^{\bT}  \ar@<.5ex>[ddrr]^{\bF_{\bD}} & & \sT  \ar@<.5ex>[ll]^{\bS}    \ar@<.5ex>[rr]^{\SI}  \ar@<.5ex>[dr]^{\bF_{\bC}}& & \sS   \ar@<.5ex>[ll]^{\OM}   \ar@<.5ex>[dl]^{\OM_{\bC}}   
\ar@/^4pc/@<.6ex> [ddll]^-{(\bS\OM)_{\bD}}        \\
& & &   \ar@<.5ex>[ul]^{\bU_{\bC}}  \bC[\sT]   \ar@<.5ex>[dl]^{\bS}  \ar@<.5ex>[ur]^{\SI_{\bC}}& \\
& & \ar@<.5ex>[uull]^{\bU_{\bD}}  \bD[s\mathrm{Set}] \ar@<.5ex>@{-->}[ur]^{\bT} \ar@/_4pc/@<.6ex> [uurr]^-{(\SI \bT)_{\bD}}& &  \\}
\end{equation}

This is a weak example of \autoref{adj6} since the counit $\epz\colon \bT\bS\rtarr \id$ is a weak equivalence but not an isomorphism. The unit $\et\colon \id \rtarr \bS\bT$ is also a weak equivalence (by the standard definition of weak equivalences of simplicial sets), so all simplicial sets are special. Here all homotopical assumptions for the composite adjunction follow directly from their counterparts for the given top right triangle.  We find that Theorems \ref{recprincom} and \ref{adjequivcom} specialize to give the following result.

\begin{thm}\label{simpset} The recognition principle and homotopical monadicity theorems for the adjunction $(\SI \bT,\bS\OM)$ and the monad $\bD$ follow directly from the recognition principle and homotopical monadicity theorem for the adjunction $(\SI,\OM)$ and the monad $\bC$.
\end{thm}

\section{The specialization to categories of operators}\label{catop}

\subsection{The cast of characters}\label{cast1}

Categories of operators were introduced in \cite{MT}  in order to define a common generalization of operadic and Segal style infinite loop space theory.\footnote{The notion of an $\infty$-operad \cite[2.1]{LurieHA} is a modern reinterpretation and generalization.}
Nonequivariantly, \cite{MT} already used that categories of operators have associated monads.  The equivariant generalization was worked out in detail in \cite{MMO}, which uses the Steiner operad to give an explicit concrete  proof that the equivariant versions of the operadic and Segalic infinite loop space machines are equivalent.   We shall concentrate on what is new and on how previous work fits into our composite adjunction context.  What is new is equivariant, so we write things that way, but our framework is already interesting nonequivariantly.  Largely following \cite{MMO}, we begin by introducing the following categorical cast of characters.

{\large
\begin{equation}\label{cast}
\begin{tabular}{||c|c|c|c||c|c|c|c||}\hline 
$\SI$ & $\LA$ & $\PI$ & $\sF$& $\sC$ & $\bC$ & $\sD$ & $\bD$ \\ \hline \hline
$\SI_G$ & $\LA_G$ & $\PI_G$ & $\sF_G$ & $\sC_G^{fin}$ & $\bC_G^{fin}$ & $\sD_G$ & $\bD_G$ \\
  \hline 
\end{tabular}
\end{equation}}

The characters $\sC_G^{fin}$ and $\bC_G^{fin}$ are new.  They fill a conceptual gap that has long puzzled the senior author.  

\begin{defn}\label{finitecats}  We define categories and inclusions  $\SI \subset \LA \subset \PI \subset \sF$.   All have objects the set of based finite sets  $\bn =\{0,1, \cdots, n\}$ for $n\geq 0$, where $\bn$ has basepoint $0$.  All have morphisms given by based functions.   These functions are restricted as follows.
\begin{enumerate}[(i)]
\item  $\SI$: bijections
\item  $\LA$: injections
\item  $\PI$:  functions $\ph\colon \bm\rtarr\bn$ such that $|\ph^{-1}\{j\}|$ is $0$ or $1$ for $1\leq j\leq n$
\item  $\sF$:  all based functions $\bm\rtarr \bn$
\end{enumerate}
Note that $\PI$ adds projections to $\LA$, since $|\ph^{-1}\{0\}|$ is not limited in size in (iii).
\end{defn}

{
We adopt the following from \cite[Convention 1.4]{MMO}. 

\begin{conv}\label{finiteGset}
For a group $G$ and a homomorphism $\al\colon G\rtarr \SI_n$,  define $\bn^{\al}$ to be the based $G$-set specified by letting $G$ act on $\mathbf{n}$ by $g\cdot i = \al(g)(i)$  for $1\leq i\leq n$.   Every based  finite $G$-set with $n$ non-basepoint elements is  isomorphic to some  $\bn^{\al}$, and we understand based finite $G$-sets to be of this form throughout.
\end{conv}
}

By a $G$-category we understand a category enriched in $G$-sets.  

\begin{defn}\label{finiteGcats}
We define $G$-categories and inclusions  $\SI_G \subset \LA_G \subset \PI_G \subset \sF_G$.   All have objects the set of based finite $G$-sets  $\bn^{\al}$ for $n\geq 0$ and all homomorphisms $\al\colon G\rtarr \SI_n$.  All have morphisms given by based functions, with $G$ acting by conjugation.  That is, for a based function $f\colon  \bm^{\al} \rtarr \bn^{\be}$,
$$  (gf)(i) = \be(g) f \al(g^{-1})(i) \ \ \text{for} \ \ 1\leq i \leq m. $$
These functions are restricted as follows.
\begin{enumerate}[(i)]
\item  $\SI_G$: bijections
\item  $\LA_G$: injections
\item  $\PI_G$:  functions $\ph\colon \bm^{\al}\rtarr\bn^{\be}$ such that
  $|\ph^{-1}\{j\}|$ is $0$ or $1$ for $1\leq j\leq n$
\item  $\sF_G$:  all based functions $\bm^{\al} \rtarr \bn^{\be}$
\end{enumerate}
\end{defn}

{
\begin{conv}\label{finiteGset2}
Define  $\SI_{\bn^{\al}}$ to be the group $\SI(\bn^{\al}, \bn^{\al})$ of automorphisms of $\bn^{\al}$ in $\SI_G$.   It is the group $\SI_n$ with the action of $G$ given by conjugation by $\al$.  
\end{conv}
}

{
\begin{defn}\label{semi}
Define a homomorphism $\al_c\colon G\rtarr  Aut(\SI_n)$ by conjugation by $\al$, $ \al_c(g)(\si) = \al(g)\com \si \com \al(g)^{-1}$.  Define an extension of  the action of $G$ on $\bn^{\al}$ to an action of the semi-direct product  
$\SI_n \rtimes_{\al_c} G$  by $(\si,g)(i) = \si\big(\al(g)(i)\big).$
\end{defn}
}

\begin{lem} The categories of \autoref{finitecats} embed in the respective categories of \autoref{finiteGcats} by identifying 
$\bn$ with $\bn^{\epz_n}$, where $\epz_n\colon G\rtarr \SI_n$ is the trivial homomorphism.
\end{lem}

\begin{rem}\label{symmonF} It is clear that the category $\sF$ and its subcategories $\SI$, $\LA$, and $\PI$ are symmetric monoidal under the wedge operation  $\bm\wed \bn \iso \mathbf{m+n}$ that identifies $\bm$ with the first $m$ and $\bn$ with the last $n$ positive elements of $\mathbf{m+n}$, in order.   Similarly, $\sF_G$ and its subcategories   $\SI_G$, $\LA_G$, and $\PI_G$ are symmetric monoidal.  In fact, they are permutative, meaning that they are symmetric strict monoidal.  In the multiplicative theory, we will implicitly use that they are bipermutative  \cite[Section 5.3]{MQR}\footnote{The corrections to \cite{MQR} in \cite{MayMult} do not affect that section.} with multiplicative product given by the smash product of finite based sets, where the positive degree elements of $\mathbf{mn}$ are identified with pairs $(i,j)$ with $1\leq i\leq m$ and $1\leq j\leq n$, ordered lexicographically.  
\end{rem}

\subsection{Categories of operators  $\sD$ and $\sD_G$}\label{cast2}  

In our cast of characters \autoref{cast1}, $\sC$ is an operad in $G\sU$ for some bicomplete cartesian monoidal ground category $\sU$, $\sD$ is its associated category of operators over $\sF$ and $\sD_G$ is its associated category of operators over $\sF_G$.  We recall these notions and their relationships and then introduce the new $G$-operad $\sC_G^{fin}$.\footnote{Here we change notation from \cite{MMO}, which used $\sC_G$ to denote an operad of $G$-spaces.  We are using $\sC$ for that.}   A bit paradoxically, to understand both the conceptual role of $\sC_G^{fin}$ and the details of its definition, it seems best to start with categories of operators.  However, $\sC_G^{fin}$ will play a much larger role in the sequel \cite{KMZ2}, where more details will be given.

We must use a bit of enriched category theory in this section, and we use the following notations.

\begin{notn}\label{enriched} We say that a category $\sM$ enriched in $\sV$ is a $\sV$-category, where $\sV$ is a given symmetric monoidal category with product $\otimes$ and unit object $\mathcal{I}$.  We have morphism objects, categorically denoted $\ul{\sM}(X,Y)$,  in $\sV$ with composition maps  
$\ul{\sM}(Y,Z) \otimes \ul{\sM}(X,Y)\rtarr \ul{\sM}(X,Z)$ in $\sV$.   The underlying category $\sM$ has morphism sets $\sM(X,Y) = \sV(\mathcal{I}, \ul{\sM}(X,Y))$.  We focus on the cartesian monoidal category  $\sV = G\sT$, and then $\sM(X,Y)$ generally takes values in $\sT$ rather than just in sets. Such ``double enrichment'' is discussed  in \cite{GMR}.
\end{notn}

\begin{defn}\label{GCO/F} A category of operators\footnote{Here and in
    \autoref{GCO/FG} we again change notation from \cite{MMO}.}  $\sD$ over
  $\sF$ in $G\sT$, abbreviated $CO$ over $\sF$, is a $G\sT$-category with objects the based sets $\mathbf{n}$ for $n\geq 0$, together with $G\sU$-functors
\[\xymatrix@1{ \Pi \ar[r]^-{\io} & \sD \ar[r]^-{\xi} & \sF\\} \] 
such that $\io$ and $\xi$ are the identity on objects and $\xi\com \io$ is the inclusion. Here $G$ acts trivially on $\PI$ and $\sF$.  We say that $\sD$ is reduced if $\sD(\mathbf{m},\mathbf{n}) = \ast$  if either $m=0$ or $n=0$, and we restrict attention to  reduced $G$-$COs$ over $\sF$ henceforward.   A morphism $\nu\colon \sD\rtarr \sE$ of $G$-$CO$s over $\sF$ is a $G\sT$-functor over $\sF$ and under $\PI$. In particular, $\xi\colon \sD\rtarr \sF$ is a map of $G$-$COs$ over $\sF$ for any  $G$-$CO$ $\sD$ over $\sF$.
\end{defn}

\begin{defn}\label{GCO/FG} A $G$-category of operators $\sD_G$ over $\sF_G$ in $G\sT$, abbreviated $G$-$CO$ over $\sF_G$,  is a $G\sT$-category, with objects the based $G$-sets $\bn^{\al}$ for $n\geq 0$ and 
$\al\colon G\rtarr \SI_n$, together with $G\sT$-functors
\[\xymatrix@1{ \Pi_G \ar[r]^-{\io_G} & \sD_G \ar[r]^-{\xi_G} & \sF_G\\} \] 
such that $\io_G$ and $\xi_G$ are the identity on objects and $\xi_G\com \io_G$ is the inclusion.  
We say that $\sD_G$ is reduced if $\sD_G(\bm^{\al} , \bn^{\be}) = \ast$ if either 
$m=0$ or $n=0$, and we restrict attention to reduced $G$-$COs$ over $\sF_G$ henceforward.
A morphism $\nu_G\colon \sD_G\rtarr \sE_G$ of $G$-$CO$s 
over $\sF_G$ is a $G\sT$-functor over $\sF_G$ and under $\PI_G$.  In particular, $\xi_G\colon \sD_G\rtarr \sF_G$ is a map of $G$-$COs$ over $\sF_G$ for any  $G$-$CO$ $\sD_G$ over $\sF_G$.
\end{defn}

\begin{rem}\label{Dpoint} Since $\sD_G$ is reduced, each $\sD_G(\bm^{\al},\bn^{\be})$ is based with basepoint the component of the unique map 
$\bm^{\al} \rtarr \mathbf{0} \rtarr \bn^{\be}$.  In line with \autoref{bspt1}, when $G\sT$ is based $G$-spaces \cite[Addendum 4.6]{MMO} adds in cofibration 
conditions that are automatically satisfied by categories of operators constructed from operads.
\end{rem}

The full subcategory $\sD\subset \sD_G$ whose objects are the trivial $G$-sets $\bn$ is a category of operators over $\sF$, also denoted $\bU \sD_G$.  Conversely, we can prolong a $CO$ $\sD$ over $\sF$ to a  $G$-$CO$  $\sD_G$ over $\sF_G$, also denoted $\bP\sD$.   Moreover, as noted in \cite[Proposition 6.9]{GMMO3}, up to isomorphism all $\sD_G$ can be constructed in this fashion.

\begin{KMZ0}\label{DtoDG}  Let $\sD$ be a category of operators over $\sF$.  We define a prolonged $G$-category of operators
$\sD_G = \bP \sD$ over $\sF_G$ whose full subcategory of objects $\mathbf{n}$ is $\sD$.  The morphism object
$\sD_G(\bm^{\al},\bn^{\be})$ in $G\sT$ is the underlying object $\sD(\mathbf{m},\mathbf{n})$ of $\sU$  with $G$-action induced by conjugation and the original $G$-action on $\sD(\mathbf{m}, \mathbf{n})$.  Explicitly, for $f\in \sD_G(\bm^{\al}, \bn^{\be})$, 
\[ g\cdot f= \beta(g) \circ (gf) \circ \al(g^{-1});\]
We check that $g\cdot (h\cdot f) = (gh)\cdot f$ using that $G$ acts trivially on permutations since they 
are in the image of $\PI$.  Composition and identity maps are inherited from $\sD$ and are appropriately equivariant.
\end{KMZ0}

\begin{cor}\label{FtoFG}  Applied to $\sF$, the construction reconstructs $\sF_G$,  and it restricts to reconstruct 
$\SI_G$, $\LA_G$, and $\PI_G$ from $\SI$, $\LA$, and $\PI$, respectively. 
\end{cor}  

\subsection{Operads $\sC$ and $G$-operads $\sC_G^{fin}$}\label{cast3}

The $\sD$ and hence $\sD_G$ of interest are constructed from operads.  We will not repeat the complete definition of an operad from \cite{MayGeo} or \cite{MayOp1}.  The following recollection focuses on operads in $G\sU$, the category of unbased $G$-spaces, although we could work with $G$-objects in other categories.

\begin{defn}\label{operad}  Recall that an operad $\sC$ in $G\sU$ consists of objects $\sC(n)\in G\sU$, where $\sC(n)$ has a left action by $G$ and a right action by $\SI_n$ that commute with each other, together with a unit $G$-map $\id\colon  \ast \rtarr \sC(1)$ and structure 
$G$-maps  
$$\ga\colon \sC(k) \times \sC(j_1) \times \cdots \times \sC(j_i) \rtarr \sC(j),$$
where $j = j_1 + \cdots + j_k$, which are associative, unital, and equivariant as specified in \cite{MayGeo, MayOp1}.   We say that $\sC$ is reduced if $\sC(0) = \ast$, and we restrict attention to reduced operads henceforward. 
\end{defn}  

\begin{notn} Let $\ph_n\colon \bn \rtarr \mathbf{1}$ be the based function that sends all $i\geq 1$ to $1$.  Observe that $\sF$ is generated by $\PI$ and the $\ph_n$ and that
\begin{equation}\label{obvious}
  \ph_k\com (\ph_{j_1} \vee \cdots \vee \ph_{j_k}) = \ph_{j},
\end{equation}
where $j = j_1 + \cdots + j_k$. Similarly, let $\ph_n^{\al}\colon \bn^{\al} \rtarr \mathbf{1}$ be the based function that sends all $i\geq 1$ to $1$ and observe that $\sF_G$ is generated by $\PI_G$ and the $\ph_n^{\al}$.
\end{notn}  

\begin{warn}\label{Uhoh1}  There is no general analogue to \autoref{obvious} of the form
$$   \ph_{\bk^{\be}}\com (\ph_{\bj_1^{\al_1}} \vee \cdots \vee \ph_{\bj_k^{\al_k}}) = \ph_{\bj^{\ga(\be;\wed \al_r)}} $$
since there is no way to define a homomorphism $\ga(\be;\wed \al_r)\colon G\rtarr \SI_{j} $ that gives the target as a sensible finite $G$-set without a relationship between $\be$ and the  $\al_r$.  
\end{warn}

Thinking  of the based  finite $G$-set  $\bj_r^{\al_r}$ as the $r$th ordered block in the wedge sum 
$$\bj^{\wed\al_r} = \bj_1^{\al_1}\wed \cdots \wed \bj_k^{\al_k},$$
we see what that relationship must be.

\begin{defn}\label{composable} Say that $(\be, \{\al_r\})$ is \em{composable} if $j_r = j_s$ and  $\al_r = \al_{s}$ whenever 
$\be(g)(r)= s$ for some $g\in G$, where $r,s\in \{1, \cdots, k\}$.  Then define 
$$\ga(\be;\wed \al_r)\colon G\rtarr \SI_{j}$$ 
by letting
$$\ga(\be;\wed \al_{r})(g)(i) = \al_s(g)(i) $$ 
when $\be(g)(r) = s$; on the left, $i$ is in the $r$th ordered block of $j_r$ letters; on the right, $i$ is in the $s$th ordered block of 
$j_r =j_s$ letters.
\end{defn}

\begin{defn}\label{defGcat} Let $\sC$ be an operad in $G\sU$.  We construct a $CO$ over $\sF$,
which we denote by $\sD(\sC)$, abbreviated $\sD$ when there is no risk of confusion.   Similarly, we write
$\sD_G = \sD_G(\sC)$ for the associated prolonged $G$-CO over $\sF_G$.  The morphism objects in $G\sT$ of $\sD$ are 
\[ {\sD}(\mathbf{m}, \mathbf{n}) = \coprod_{\ph\in\sF(\mathbf{m},\mathbf{n})} \prod_{1\leq j\leq n} \sC(\ph_j)  \]
with $G$-action induced by the $G$-actions on the $\sC(n)$.   Here $\ph_j = |\ph^{-1}(j)|$.
Write elements in the form $(\ph,c)$, where $c = (c_1,\dots,c_n)$.  
For $(\ph,c)\colon \mathbf{m}\rtarr \mathbf{n}$ and $(\ps,d)\colon \mathbf{k}\rtarr \mathbf{m}$, define
\[ (\ph,c)\com (\ps,d) = \big(\ph\com \ps, 
\prod_{1\leq j\leq n}\ga(c_j;\prod_{\ph(i) = j} d_i)\si_j\big). \]
Here $\ga$ denotes the structural maps of the operad.  The $d_i$ with $\ph(i) =j$ 
are ordered by the natural order on their indices $i$
and $\si_j$ is that permutation of $(\phi\com\ps)_j$ letters which converts 
the natural ordering of $(\phi\com\ps)^{-1}(j)$ as a subset of $\{1,\dots,k\}$ to
its ordering obtained by regarding it as $\coprod_{\ph(i)=j}\ps^{-1}(i)$, so ordered
that elements of $\ps^{-1}(i)$ precede elements of $\ps^{-1}(i')$ if $i< i'$ and 
each $\ps^{-1}(i)$ has its natural ordering as a subset of $\{1,\dots,k\}$. 
The identity element in $\sD(\mathbf{n},\mathbf{n})$ is $(id, \id^n)$, where $\id$ on
the right is the unit element in $\sC(1)$. The map $\xi\colon \sD \rtarr \sF$
sends $(\ph,c)$ to $\ph$.  The inclusion $\io\colon \PI\rtarr \sD$ sends 
$\ph\colon \mathbf{m}\rtarr \mathbf{n}$ to $(\ph, c)$, where $c_i = \id \in \sC(1)$ if $\ph(i) = 1$ and
$c_i = \ast\in \sC(0)$ if $\ph(i) = 0$. 
\end{defn}

\begin{exmp} The commutativity operad $\sN$ has $n$th space a point for all $n$.
We think of it as a $G$-trivial $G$-operad. Then $\sF = \sD(\sN)$, again regarded 
as $G$-trivial. 
\end{exmp}

\begin{defn}\label{defGop} Conversely, suppose given a CO $\sD$ over $\sF$ such that the iterated wedge sum in $\sF$ of \autoref{symmonF} is covered by an analogous associative, unital, and equivariant wedge sum in $\sD$. {
Here, since $\sD$ is reduced, each $\sD(\bm,\bn)$ has basepoint the unique map that factors through $0$.
}
\begin{equation}\label{Dplus}
\xymatrix{
\sD(\bj_1,\mathbf{1}) \times \cdots \times \sD(\bj_k,\mathbf{1}) \ar[r]^-{\wed} \ar[d]_{\xi^k} & \sD(\bj, \bk)\ar[d]^{\xi} \\
\sF(\bj_1,\mathbf{1}) \times \cdots \times \sF(\bj_k,\mathbf{1}) \ar[r]_-{\wed} & \sF(\bj, \bk)\\}
\end{equation}
This condition holds in the examples we know.   We then construct an operad $\sC(\sD) = \sC$, called the operad associated to $\sD$.
We define $\sC(n)$ to be the component of $\ph_n$ in $\sD(\bn,\mathbf{1})$ with its given left  action by $G$ and with right action by $\SI_n$ induced by composition with $\SI_n \subset \PI(\bn,\bn)$.  We define $\id\in \sC(1)$ via the identity morphism of $\sD(\mathbf{1}, \mathbf{1})$,  and we define the structural maps $\ga$ to be the restrictions to the components of $\ph_k$ and the $\ph_{j_r}$ of the composites
\begin{equation}\label{canonical}
\xymatrix@1{
\sD(\bk,\mathbf{1}) \times \sD(\bj_1,\mathbf{1})  \times \cdots \times \sD(\bj_k,\mathbf{1})  \ar[r]^-{\id\times \wed}  &
\sD(\bk,\mathbf{1})  \times \sD(\bj,\bk) \ar[r]^-{\com} & \sD(\bj,\mathbf{1}). \\}
\end{equation}
Since $\sC(n)$ lies over  $\ph_n$, it is unbased, so in $G\sU$.
\end{defn}

 \begin{rem}\label{CDDC} This gives a functor $\sC(-)$ from the full subcategory of those categories of operators over $\sF$ in $G\sT$ which  satisfy the assumption of \autoref{Dplus}  to the category of operads in $G\sU$.  This functor is left adjoint to the functor $\sD(-)$ of \autoref{defGcat}.  The unit of the adjunction is given by the evident identifications  $\sC = \sC(\sD(\sC))$.  The counit  $\epz\colon \sD\sC(\sD) \rtarr \sD$ is given by use of $\wed$.   It is an isomorphism in the examples we care about, but presumably not in general.
 \end{rem}

\begin{warn}\label{Uhoh2} There is no general analogue to \autoref{canonical}  for $G$-categories of operators over 
$\sF_G$.  When we start with general finite $G$-sets, the target and source of the middle term are different. 
\end{warn}

To get around this problem, we define the following twisted product.
  \begin{defn}\label{Gcomposable}
  For composable $(\beta,\{\alpha_{r}\})$, let $\prod_{r \in \bk^{\beta}}
  \sD_G(\bj_r^{\alpha_r}, \mathbf{1})$ be the space $\prod_{1 \leq r \leq k}
  \sD_G(\bj_r^{\alpha_r}, \mathbf{1})$ with twisted $G$-action given in precise analogy with 
  the definition of $\ga(\be;\wed \al_{r})$ in \autoref{composable} so as to ensure that $\vee$ on
  $\prod_{r \in \bk^{\beta}} \sF_G(\bj_r^{\alpha_r}, \mathbf{1})$ lands in
$\sF_G(\bj^{\ga(\be;\wed \al_{r})}, \bk^{\beta})$.
In analogy with \autoref{defGop}, we require that the iterated wedge sum in $\sF_G$ be covered by an
associative, unital, and equivariant wedge sum in $\sD_G$, as in \autoref{Dplus}:
\begin{equation}
\xymatrix{
\prod_{r \in \bk^{\beta}} \sD_G(\bj_r^{\alpha_r}, \mathbf{1})  \ar[r]^-{\wed} \ar[d]_{\xi^k} & \sD_G(\bj^{\ga(\be;\wed \al_{r})}, \bk^{\beta})\ar[d]^{\xi} \\
\prod_{r \in \bk^{\beta}} \sF_G(\bj_r^{\alpha_r}, \mathbf{1})  \ar[r]_-{\wed} & \sF_G(\bj^{\ga(\be;\wed \al_{r})}, \bk^{\beta})\\}
\end{equation}
We then define $\ga_G$ to be the composite
\begin{equation}\label{canonical3}
  \xymatrix@1{
 \sD_G(\bk^{\be},\mathbf{1})\times \prod_{r\in
  \bk^{\be}}\sD_G(\bj_r^{\al_r},\mathbf{1}) \ar[d]^-{\mathrm{id} \times \vee} \\
\sD_G(\bk^{\be},\mathbf{1})\times \sD_G(\bj^{\ga(\be;\wed \al_{r})}, \bk^{\be})
\ar[d]^-{\circ}\\
 \sD_G(\bj^{\ga(\be;\wed \al_{r})},\mathbf{1}), \\}
\end{equation}
With $G$ acting through $\ga(\be;\wed \al_{r})$,  $\ga_G$ is then a $G$-map.
\end{defn}

We can now give a definition of $G$-operads $\sC_G^{fin}$ that allows a comparison analogous to that of \autoref{CDDC}. The problems explained in Warnings \ref{Uhoh1} and \ref{Uhoh2} and the answers to those problems given in Definitions \ref{composable} and \ref{Gcomposable} pave the way.

\begin{defn}\label{crude}  A $G$-operad $\sC_G^{fin}$ in $G\sU$  consists of $G$-objects  $\sC_G^{fin}(\bn^{\al})$  in $G\sU$ for all $\bn^{\al}$, where 
$\sC_G^{fin}(\bn^{\al})$ has a left action by $G$ and a right action by $\SI_n$ that together give a right action of the semi-direct product   $\SI_n \rtimes_{\al_c} G$,
together with a unit $G$-map $\id \colon \ast \rtarr \sC_G^{fin}(\mathbf{1})$ and structure 
$G$-maps  
$$\ga_G\colon \sC_G^{fin}(\bk^{\be}) \times \sC_G^{fin}(\bj_1^{\al_1} ) \times \cdots \times \sC_G^{fin}(\bj_k^{\al_k}) \rtarr 
\sC_G^{fin}(\bj^{\ga(\be; \wed \al_r)})$$
for composable $(\be,\{\al_r\}_{1\leq r\leq k})$.  We require the $\ga_G$ to be both
$$\SI_{k} \rtimes_{\be_c}  G \ \ \text{and} \ \  (\SI_{j_1} \times \cdots \times
  \SI_{j_k})\rtimes_{(\al_1,...,\al_k)_c} G$$ equivariant, where these groups act on the source of $\ga$ via their actions on its coordinates and where they act on the target via embeddings as subgroups of  $\SI_{j}$ via block permutations and permutations within blocks.
The $\ga_G$ must be associative and unital for composable data, as specified in \cite[Definition 1]{MayOp1}. We say that $\sC_G^{fin}$ is reduced if $\sC_G^{fin}(\mathbf{0}) = \ast$, and we restrict attention to reduced $G$-operads henceforward. 
\end{defn} 

\begin{prop}\label{CtoCG}  A $G$-operad $\sC_G^{fin}$ in $G\sU$ restricts on $G$-trivial based finite $G$-sets $\bn$ to an
operad  $\sC = \bU\sC_G^{fin}$ in $G\sU$.  Conversely, an operad $\sC$ in $G\sU$ prolongs to a $G$-operad $\sC_G^{fin} = \bP \sC$ in $G\sU$.  
\end{prop}
\begin{proof}   The prolongation is analogous to the prolongation from $\sD$ to $\sD_G$ spelled out in
\autoref{DtoDG}, starting from the definition $\sC_G^{fin}(\bn^{\al}) = \sC(n)^{{\al}}$, with notation as in \autoref{Xal}. 
\end{proof}

The evident analogue of \autoref{CDDC} holds and we have the following diagram of categories and adjunctions. 
\begin{equation}\label{OpsCats}
\xymatrix{
\text{Operads in}  \ G\sU  \ar@<.5ex>[r]^-{\bP} \ar@<.5ex>[d]^{\sD(-)} &  
G\text{-Operads in} \ G\sU \ar@<.5ex>[d]^{ \sD_G(-)}  \ar@<.5ex>[l]^-{\bU} \\
CO's  \ \text{over}\  \sF   \ar@<.5ex>[u]^{\sC(-)} \ar@<.5ex>[r]^-{\bP}  & 
G\text{-}CO's\ \text{over} \ \sF_G  \ar@<.5ex>[l]^-{\bU} \ar@<.5ex>[u]^{\sC_G(-)} \\}
\end{equation}

We summarize our conclusions.

\begin{prop}  In the adjunctions in \autoref{OpsCats}, the composites 
$$  \bU\bP,\ \  \sC(-)\sD(-) \ \ \text{and}\ \   \sC_G(-)\sD_G(-)  $$
are the identity, and the composite $\bP\bU$ is naturally isomorphic
to the identity. Moreover, the square of left adjoints and the square of right adjoints commute.
\end{prop}

\subsection{The associated monads $\bD$ and $\bD_G$}\label{cast4}

The ground category of the monads $\bC$ and $\bC_G^{fin}$ is $G\sT$.  The following notations and definition specify the ground categories of the monads $\bD$ and $\bD_G$.  Recall from \autoref{enriched} that $G\sT(X,Y)$ denotes the based space of $G$-morphisms $X\rtarr Y$
and that  $\ul{G\sT}(X,Y)$ denotes the based $G$-space of morphisms $X\rtarr Y$ in the enriched category $\ul{G\sT}$.  
Note that
$$G\sT(X,Y) = \ul{G\sT}(X,Y)^G.$$

\begin{rem}\label{bspttoo} In line with \autoref{bspt1},  \cite{MMO} used $G\sU_*$ when working formally with general based $G$-spaces and used $G\sT$ when thinking homotopically and restricting to nondegenerately based $G$-spaces.  Again, the distinction is carefully handled in \cite{MMO}, and we find the use of two notations distracting in the conceptual framework on which we are focusing.
\end{rem}

\begin{defn}  Let $\PI[G\sT] =\PI[\ul{G\sT}]$ be the category of functors $\PI \rtarr G\sT$ or, equivalently, the category of $G$-functors $\PI\rtarr \ul{G\sT}$.  The equality holds since $G$ acts trivially on $\PI$ so that a $G$-functor $X$
from $\PI$ to $\ul{G\sT}$ must take values fixed under $G$, that is, equivariant maps.  We call $X$ a $\PI$-$G$-object in $\sT$. 
Analogously, let $\PI_G[\ul{G\sT}]$ be the category of  $G$-functors $\PI_G\rtarr \ul{G\sT}$.  We call $X$ a $\PI_G$-$G$-object in $\sT$.   Restriction to the trivial $\bn$ gives a forgetful functor  $\bU\colon \PI_G[\ul{G\sT}] \rtarr \PI[\ul{G\sT}]$.
\end{defn}

\begin{rem}  For reasons as in \autoref{bspt1}, \cite[Definitions 2.3 and 2.33]{MMO} adds in a cofibration condition.
\end{rem}

\begin{prop}\label{PUone}  The forgetful functor  $\bU\colon \PI_G[\ul{G\sT}] \rtarr \PI[\ul{G\sT}]$ has a left adjoint prolongation functor $\bP$, and $(\bP,\bU)$ is an adjoint equivalence. It restricts to an adjoint equivalence between
$\LA[\ul{G\sT}]$ and $\LA_G[\ul{G\sT}]$. 
\end{prop}
\begin{proof}  This is the formal part of \cite[Theorem  2.38]{MMO}.   For 
$X\in \PI[\ul{G\sT}]$,  $\bP X$ is given by the categorical tensor product
$$(\bP X)(\bn^{\al})  = \PI_G(-, \bn^{\al})\otimes_{\PI} X.$$
Since $\bU$ is full and faithful, the unit $\et\colon X \rtarr \bU\bP X$ is the evident identification. 
The counit $\epz\colon \bP \bU Y \rtarr Y$ is an isomorphism by a Yoneda type verification using the identity function on $\bn$ regarded as an element of $\PI_G(\bn,\bn^{\al})$.  The statement with $\PI$ and $\PI_G$ replaced by $\LA$ and 
$\LA_G$ is proven in the same way.
\end{proof}

{
The following definition, which combines \cite[Definition 1.3 and Example 1.5]{MMO},  leads to more combinatorially explicit descriptions of the functor $\bP$ and natural transformation  $\epz^{-1}\colon \Id \rtarr \bP\bU$.

\begin{defn}\label{Xal} Let $Y$ be a $(G\times \Sigma_n)$-space and $\al\colon G\rtarr \Sigma_n$ be a homomorphism. We define $Y^\al$ to be the $G$-space with underlying space $Y$ and with a new $G$-action $\cdot_\al$ given by
$g\cdot_\al y= (g,\al(g))\cdot y$.  Thus $Y^{\al}$ is $Y$ with $G$-action twisted by $\al$.  Said another way $Y^{\al}$ is just notation for $Y$ with its action by the graph subgroup $\GA_{\al} = \{(g,\al(g)\}$ of $G\times \SI_n$, pulled back along the isomorphism  $\pi^{-1}\colon G\rtarr \GA_{\al}$, where $\pi\colon \GA_{\al} \rtarr G$ is the projection.   On passage to fixed point spaces, $ (Y^{\al})^G \iso Y^{\GA_{\al}}$.
The example of interest is obtained by taking $Y = X^n$ for a based $G$-space X. Then we have the identification 
$(X^n)^\al=X^{\mathbf{n}^\al}$, where  $X^{\mathbf{n}^\al}$ denotes the set of based maps $\bn^{\al}\rtarr X$ with $G$ acting by conjugation.   Explicitly, this is  the cartesian power $X^n$ with twisted action by $G$ given by
$g(x_1,\dots,x_n) = (gx_{\alpha(g^{-1})(1)},\dots, gx_{\alpha(g^{-1})(n)})$.
\end{defn} 
}
With this notation in hand, the abstract categorical tensor product defining the prolongation functor $\bP$ in \autoref{PUone} collapses to an explicit formula: for a $\Pi$-$G$-space $X$, we have a canonical $G$-homeomorphism$$ (\bP X)(\bn^\alpha) \cong X(\bn)^\alpha. $$Consequently, for a $\Pi_G$-$G$-space $Y$, the inverse of the counit $\epsilon^{-1}: Y \to \bP\bU Y$ evaluated at $\bn^\alpha$ is explicitly given by the canonical identification $Y(\bn^\alpha) \xrightarrow{\cong} (\bU Y(\bn^\alpha))^\alpha$.

\begin{defn}  The monads $\bD$ in $\PI[\ul{G\sT}]$ and $\bD_G$ in $\PI_G[\ul{G\sT}]$ are given by the categorical tensor products
\begin{equation}\label{DX}
 (\bD X)(\bn) =  \sD(-,\bn) \otimes_{\PI} X  \ \ \text{and} \ \ (\bD_G Y)(\bn^{\al}) =  \sD_G(-,\bn^{\al}) \otimes_{\PI_G} Y
 \end{equation}
for a $\PI$-$G$-functor  $X$ and a $\PI_G$-$G$-functor $Y$.  Here $\sD(-,\mathbf{n})$ and $\sD_G(-,\bn^{\al})$ are the functors represented by $\bn$ and $\bn^{\al}$.  Explicitly, these tensor products  are the coequalizers  
\begin{equation}
 \xymatrix@1{
\coprod_{\bk,\bm} \sD(\bm,\bn)\times \PI(\bk,\bm) \times X(\bk) \ar@<.7ex>[r]^{}  \ar@<-.7ex>[r]  &  
\coprod_{k}\sD(\bk,\bn) \times X(\bk) \ar[r]  &  (\bD X)(\bn) \\}  
\end{equation}

\begin{equation}
\xymatrix@1{\underset{\bk^{\ga},\bm^{\be}}{\coprod} \sD_G(\bm^{\be},\bn^{\al})\times \PI_G(\bk^{\ga},\bm^{\be}) \times Y(\bk^{\ga}) \ar@<.7ex>[r]
 \ar@<-.7ex>[r]  &
\underset{\bk^{\ga}}{\coprod}\sD_G(\bk^{\ga},\bn^{\al}) \times Y(\bk^{\ga}) \ar[r]  &  (\bD_G Y)(\bn^{\al}). \\}  
\end{equation}
The parallel arrows are given by composition in $\sD$ or in $\sD_G$ and the action of $\PI$ on $X$ or of $\PI_G$ on $Y$.
The products  $\mu\colon \bD\bD\rtarr \bD$ and $\mu\colon \bD_G \bD_G \rtarr \bD_G$ 
are derived from the composition in $\sD$ or $\sD_G$ and thus from the structure maps $\ga$ 
of $\sC$ when $\sD = \mathfrak{D}(\sC)$. The unit maps $\et$ are derived from the identity morphisms in these categories and thus from the unit  $\id\in \sC(1)$.   We define a $\sD$-$G$-algebra to be a 
$\bD$-algebra in $\Pi[\ul{G\sT}]$ and we define a
$\sD_G$-$G$-algebra to be a $\bD_G$-algebra in $\Pi_G[\ul{G\sT}]$.
We let $\bD\big[\PI[\ul{G\sT}]\big]$ and $\bD_G\big[\PI_G[\ul{G\sT}]\big]$ denote the respective categories of algebras.  Restriction to trivial $\bn$ gives a forgetful functor  
$\bU\colon \bD_G\big[\PI_G[\ul{G\sT}]\big] \rtarr \bD\big[\PI[\ul{G\sT}]\big]$.
\end{defn}

Ignoring cofibration conditions, the following definition is the formal part of  \cite[Definitions 4.7 and 4.8]{MMO}.

\begin{defn}  Define a $\sD$-$G$-space to be a $G\sT$-functor  $\sD \rtarr \ul{G\sT}$.  Define
a $\sD_G$-$G$-space to be a $G\sT$-functor  $\sD_G \rtarr \ul{G\sT}$.  In both cases, maps are $G\sT$-natural transformations.
\end{defn}

Again ignoring cofibration conditions, the following result is \cite[Proposition 5.17]{MMO}.  Its proof is an immediate comparison of definitions.

\begin{prop}\label{DAlgIso}  The category of $\sD$-$G$-spaces is isomorphic to the category of $\bD$-algebras in $\Pi[\ul{G\sT}]$.   The category of $\sD_G$-$G$-spaces is isomorphic to the category of $\bD_G$-algebras in $\Pi_{G}[\ul{G\sT}]$.  
\end{prop}

The following result is \cite[Theorem 4.11]{MMO}.  The proof there is a detailed combinatorial inspection of definitions starting from  \autoref{PUone},  but we give a quick conceptual argument.

\begin{prop}\label{PUtwo}  The prolongation functor $\PI[G\sT] \rtarr \PI_G[G\sT]$ extends to a prolongation functor 
$\bP\colon \Dalg \rtarr \DGalg $  left adjoint to the  forgetful functor  $\bU\colon \DGalg  \rtarr \Dalg$, and $(\bP,\bU)$ is an adjoint equivalence.
\end{prop}
\begin{proof}   The functor $\bP$ is defined on $\bD$-algebras $Y$ by the categorical tensor product
$$(\bP Y)(\bn^{\al}) = \sD_G(-,\bn^{\al})\otimes_{\PI} Y(-),$$
where the variable $(-)$ runs over the objects $\bm$ of $\sD$.
It is clear from the definition that $\bP Y$ is naturally a  $\sD_G$-$G$-space.  By \autoref{DAlgIso}, that  means that $\bP Y$ is a $\bD_G$-algebra.   
The adjunction is immediate.  

Since 
$\bU$ is full and faithful, for a $\bD$-algebra $X$  the unit $\et \colon X\rtarr \bU\bP X$ is the evident identification.  Consider the counit $\epz\colon \bP\bU Z \rtarr Z$ for a $\bD_G$-algebra $Z$.   Here  $Z$ is given by a split coequalizer as in \autoref{split2}.  Since $\bP\bU$ preserves split coequalizers, it is immediate by comparison of coequalizers that $\epz$ is an isomorphism if it is an isomorphism when $Z = \bD_G Y$ for a $\PI_G$-$G$-space $Y$.  Moreover, by \autoref{PUone}, we may assume without loss of generality that $Y = \bP X$, where $X$ is a $\PI$-$G$-space  and $\bP$ here is the prolongation functor of that result.  Not inserting notation for variables, we have the following formal identification.
$$  \bD_G \bP X = \sD_G\otimes_{\PI_G} (\PI_G \otimes_{\PI} X) \iso \sD_G\otimes_{\PI} X \iso \sD_G\otimes_{\sD}(\sD\otimes_{\PI} X) = \bP\bD X.$$
Therefore $\epz\colon \bP\bU \bD_G \bP X \rtarr \bP X$ 
can be identified with $\epz\colon \bP\bU \bP \bD X \rtarr \bP \bD X, $ 
which is an isomorphism with inverse $\bP \et$ by a triangle identity.
\end{proof}

\subsection{The associated monads $\bC$ and $\bC_G^{fin}$}\label{cast5}

We turn to the monads associated to operads $\sC$ and $G$-operads $\sC_G^{fin}$.  In contrast to the case of categories of operators, these monads are special cases rather than  general cases of certain categorical tensor products.  We start with the general definition.    Recall  $\LA$ and $\LA_G$ from Definitions \ref{finitecats} and \ref{finiteGcats}.  Fix an operad $\sC$ of $G$-spaces and a $G$-operad $\sC_G^{fin}$, which we can take to be $\bP(\sC)$. Let $\sD = \sD(\sC)$ and 
$\sD_G = \sD_G(\sC_G^{fin}) \iso \bP\sD(\sC)$.  

\begin{lem}  The operad $\sC$ restricts to a contravariant functor $\sC\colon \LA \rtarr G\sU$. The $G$-operad $\sC_G^{fin}$ restricts to a contravariant functor $\sC_G^{fin}\colon \LA_G \rtarr \ul{G\sU}$.
\end{lem}  
\begin{proof} The first statement goes back to \cite[Construction 2.4]{MayGeo}; details are recalled in \autoref{OtoJ} below.  An analogous explicit equivariant description is given in \cite{KMZ2}.  The action can be described implicitly by the commutative diagram
$$ \xymatrix{\sC(n) \times \LA(\bm,\bn)  \ar[d]_{\subset}    \ar[r] &  \sC(m)\ar[d]^{\subset} \\
     \sD(\bn,\mathbf{1}) \times \sD(\bm,\bn) \ar[r]_-{\com}  & \sD(\bm,\mathbf{1}). \\}$$
That is, the restriction to $\sC(n)\times \LA(\bm,\bn)$ of the displayed composition in $\sD$ lands in $\sC(m)$ and gives the action. A  proof of the second statement works the same way, using the analogous diagram
 $$ \xymatrix{\sC_G^{fin}(\bn^{\be}) \times \LA_G(\bm^{\al},\bn^{\be})  \ar[d]_{\subset}    \ar[r] &  \sC_G^{fin}(\bm^{\al})\ar[d]^{\subset} \\
     \sD_G(\bn^{\be},\mathbf{1}) \times \sD_G(\bm^{\al},\bn^{\be}) \ar[r]_-{\com} &  \sD_G(\bm^{\al},\mathbf{1}).  \qedhere  \\}  
     $$
\end{proof}

\begin{defn}  Let $\LA[G\sT]= \LA[\ul{G\sT}]$ be the category of functors $\LA\rtarr G\sT$ and  let $\LA_G[\ul{G\sT}]$ be the category $G$-functors $\LA_G\rtarr \ul{G\sT}$.  Define functors
$$\bD\colon \LA[\ul{G\sT}]\rtarr G\sT \ \ \text{and} \ \  \bD_G \colon \LA_G[\ul{G\sT}]\rtarr G\sT $$
by the categorical tensor products defined on $X\in \LA[\ul{G\sT}]$ and $Y\in \LA_G[\ul{G\sT}]$ by
$$\bD X = \sC\otimes_{\LA} X \ \ \text{and} \ \ \bD_G Y = \sC_G^{fin}\otimes_{\LA_G} Y. $$
\end{defn}

The following surprising identification is proven by a comparison of coequalizer  diagrams that is given in \cite{KMZ2}.

\begin{prop}\label{cuter}  The functor $\bD$ can be identified with the composite $\bD_G\com \bP$. 
\end{prop}

The general functor $\bD$ was used in \cite{CMT, MT}, but it restricts to a monad on $G\sT$.  We use the following definition to show that $\bD_G$ also restricts to a monad on $G\sT$.  For immediate purposes,  it suffices to consider $\LA$ and $\LA_G$, but we will shortly also need $\PI$ and $\PI_G$.  We use different notations because the conceptual roles of $\LA$ and $\LA_G$ turn out to be quite different from the roles of $\PI$ and $\PI_G$.

\begin{defn}   We define $G\sT$-functors
$$ \bR\colon G\sT \rtarr \PI[\ul{G\sT}] \ \ \text{and} \ \  \bR_G\colon G\sT\rtarr  \PI_G[\ul{G\sT}].$$
For $X\in G\sT$, we let  $(\bR X)(\bn) = X^n$.  For a based function $f\colon \bm \rtarr \bn$ such that $|f^{-1}(j)| = 0$ or $1$ for $j>0$, we let 
$(\bR X)(f)\colon X^m\rtarr X^n$ be the map that sends $(x_1,\dots,x_m)$ to $(y_1,\cdots y_j)$, where $y_j = x_{f^{-1}(j)}$ if $f^{-1}(j)>0$ and $y_j = \ast$ if $j\notin \im(f)$.  We define $\bR_G = \bP \bR$, where $\bP\colon \PI[\ul{G\sT}] \rtarr  \PI_G[\ul{G\sT}]$ is defined in  \autoref{PUone}.  We define 
$$ \bK\colon G\sT \rtarr \LA[\ul{G\sT}] \ \ \text{and} \ \  \bK_G\colon G\sT\rtarr  \LA_G[\ul{G\sT}]$$
by restricting from $\PI$ to $\LA$ and from $\PI_G$ to $\LA_G$ in the targets of $\bR$ and $\bR_G$. 
\end{defn}

{
Use of \autoref{Xal} makes the definition of $\bR_G$ more combinatorially explicit.
}
\begin{notn}  Restricting to $\LA$ and $\LA_G$, we  write
$$  \bC X =\bD \bK X  \  \ \text{and} \ \ \bC_G^{fin} X = \bD_G\bK_G X.$$
for based $G$-spaces $X$.
\end{notn}  

\begin{prop}   The functor $\bC\colon G\sT \rtarr G\sT$ is a monad with product $\mu$ induced by the  structure maps $\ga$ of $\sC$ and unit $\et$ induced by the unit $\id:\ast\rtarr \sC(1)$.   The functor $\bC_G^{fin}\colon G\sT \rtarr G\sT$ is a monad with product $\mu$ induced by the  structure maps $\ga$ of $\sC_G^{fin}$ and unit $\et$ induced by the unit $\id:\ast\rtarr \sC_G^{fin}(1)$.  
\end{prop}

The first statement goes back to \cite[Construction 2.4]{MayGeo}, with later generalizations of context.  The analog for $\bC_G^{fin}$ is not as obvious since the restriction of the structure map $\ga_G$ to composable tuples makes it appear that the product $\mu$ on $\bC_G^{fin}$ is only partially defined.  However, \autoref{cuter} saves the day. More details
are given in \cite{KMZ2}. 

We complete the formal picture with the following comparisons of ground categories and of monads.  

\begin{defn}  Define $G\sT$-functors
$$ \bL\colon \PI[\ul{G\sT}] \rtarr G\sT\ \ \text{and} \ \  \bL_G\colon  \PI_G[\ul{G\sT}] \rtarr G\sT$$
by  evaluating functors on $\mathbf{1}$.
\end{defn}  

\begin{lem}\label{mopup1}   The pairs $(\bL,\bR)$ and $(\bL_G, \bR_G)$ are adjunctions such that $\bL\bR = \id$ and $\bL_G\bR_G = \id$.   Moreover,
just as $\bP \bR = \bR_G$, so also  $\bU\bR_G = \bR$, where $(\bP,\bU)$ is as in \autoref{PUone}. 
\end{lem}

\begin{rem}\label{Segal}  The units of these adjunctions are denoted
$$ \de\colon \id \rtarr \bR\bL \ \ \text{and} \ \ \de_G\colon \id \rtarr \bR_G\bL_G$$
since, for $X\in \PI[\sT]$, the coordinates of the maps $\de\colon X(n) \rtarr X(1)^n$ are induced by delta functions $\bn \rtarr \mathbf{1}$.  These maps are called Segal maps since their importance was first noticed by Segal \cite{Seg}, who coined the word ``special'' to indicate when they are equivalences.
\end{rem}

\begin{lem}\label{mopup2} We have an isomorphism $\om\colon \bD\bR \rtarr \bR\bC$ of functors $G\sT \rtarr \PI[\ul{G\sT}]$ and an isomorphism
$\om_G \colon \bD_G\bR_G \rtarr \bR_G \bC_G^{fin}$ of functors $G\sT \rtarr \PI_G[\ul{G\sT}]$.  
\end{lem} 

Lemmas \ref{mopup1} and \ref{mopup2} were proven nonequivariantly in \cite[Lemma 5.7 and Section 6]{MT}. Equivariantly, ignoring $\bC_G^{fin}$, they were proven in \cite[Lemma 2.8 and Propositions 5.21 and 5.22]{MMO}\footnote{\cite[Proposition 5.20]{MMO} states that $\bL$ takes $\bD$-algebras to $\bC$-algebras, which is false, but Propositions 5.21 and 5.22 do not use that and are correct as stated; remember that $\bC_G^{fin}$ in \cite{MMO} is  what we are denoting by $\bC$ here.}.   The proof of the second part of \autoref{mopup2} is entirely similar to that of the first.   In both parts, the diagrams of \autoref{ass7} are easily seen to commute.

\subsection{Categories of operators in the composite adjunction context}\label{newcatop}
With this background, we have the following two special cases of \autoref{adj6}.   We specialize to $G$-spaces and $G$-spectra in this subsection, but there are analogous specializations in other contexts.  We let $\sC$ be an $E_{\infty}$ operad of $G$-spaces, as defined in \autoref{oper2}, let $\sC_G^{fin}$ be the associated $G$-operad, and let $\sD$ and $\sD_G$ be the categories of operators over  $\sF$ and $\sF_G$ constructed from $\sC$.  We have the associated monads and the following diagram \autoref{adj7}.  We note right away that \autoref{reduce} can be applied to reduce both parts of \autoref{adj7} to the corresponding diagrams with $\PI[\ul{G\sT}]$ and  $\PI_G[\ul{G\sT}]$ restricted to $\PI[\ul{G\sT}]_{ss}$ and 
 $\PI_G[\ul{G\sT}]_{ss}$.  In effect that reduces the  present context given by the big triangles  to the
 original context given by the small triangles at the right.
 
\begin{equation}\label{adj7} 
\xymatrix{
\PI[\ul{G\sT}] \ar@<.5ex>[rr]^{\bL}  \ar@<.5ex>[ddrr]^{\bF_{\bD}} & & G\sT  \ar@<.5ex>[ll]^{\bR}    \ar@<.5ex>[rr]^{\SI^{\infty}}  \ar@<.5ex>[dr]^{\bF_{\bC}}& & G\sS   \ar@<.5ex>[ll]^{\OM^{\infty}}   \ar@<.5ex>[dl]^{\OM^{\infty}_{\bC}}   
\ar@/^4pc/@<.6ex> [ddll]^-{(\bR\OM^{\infty})_{\bD}}        \\
& & &   \ar@<.5ex>[ul]^{\bU_{\bC}}  \bC[G\sT]   \ar@<.5ex>[dl]^{\bR}  \ar@<.5ex>[ur]^{\SI^{\infty}_{\bC}}& \\
& & \ar@<.5ex>[uull]^{\bU_{\bD}}  \Dalg\ar@{-->}@<.5ex>[ur]^{\bL} \ar@/_4pc/@<.6ex> [uurr]^-{(\SI^{\infty} \bL)_{\bD}}& &  \\
&  & \ar@<.5ex> [d]^{\bP}  && \\
& & \ar@<.5ex>[u]^{\bU}   && \\
\PI_G[\ul{G\sT}] \ar@<.5ex>[rr]^{\bL_G}  \ar@<.5ex>[ddrr]^{\bF_{\bD_G}} & & G\sT  \ar@<.5ex>[ll]^{\bR_G}    \ar@<.5ex>[rr]^{\SI^{\infty}}  \ar@<.5ex>[dr]^{\bF_{\bC_G^{fin}}}& & G\sS   \ar@<.5ex>[ll]^{\OM^{\infty}}   \ar@<.5ex>[dl]^{\OM^{\infty}_{\bC_G^{fin}}}   
\ar@/^4pc/@<.6ex> [ddll]^-{(\bR\OM^{\infty})_{\bD_G}}        \\
& & &   \ar@<.5ex>[ul]^{\bU_{\bC_G^{fin}}}  \bC_G^{fin}[G\sT]   \ar@<.5ex>[dl]^{\bR_G}  \ar@<.5ex>[ur]^{\SI^{\infty}_{\bC_G^{fin}}}& \\
& & \ar@<.5ex>[uull]^{\bU_{\bD_G}}  \DGalg\ar@{-->}@<.5ex>[ur]^{\bL_G} \ar@/_4pc/@<.6ex> [uurr]^-{(\SI^{\infty} \bL_G)_{\bD_G}}& &  \\}
\end{equation}

\vspace{1mm}

The small right-hand triangle of the top diagram is the $G$-spectrum level special
case of \autoref{adj5} that is discussed in \autoref{SpaceSpectra}.   Via Propositions \ref{PUone} and \ref{PUtwo}, we think
of the prolongation and forgetful functor adjunction $(\bP,\bU)$ as mapping the
entire top diagram equivalently to the entire bottom diagram.  Since the monads  $\bC$ and $\bC_G^{fin}$ on $G\sT$ are isomorphic, we think of 
$(\bP, \bU)$ as identifying the small right-hand triangles of the two diagrams.   \autoref{SpaceSpectra} applies equally well to both, compatibly. 

 The monads $\bD$ and $\bD_G$ in the left trapezoids have different but equivalent ground categories  $\PI[G\sT]=\PI[\ul{G\sT}]$ and  
 $\PI_G[\ul{G\sT}]$, and they have equivalent categories of algebras.  
The adjunction $(\bP,\bU)$ therefore gives an equivalence between the top diagram and the bottom diagram.  

However, remember that we must start with chosen subcategories of weak equivalences in  $\PI[G\sT]$ and in 
$\PI_G[\ul{G\sT}]$, as in \autoref{ass8}.  We recall from \cite{MMO} how these can and must be defined in order for $(\bP,\bU)$ to induce an equivalence between these subcategories from a homotopical point of view. In fact we have two choices, one leading to genuine $G$-spectra for finite $G$ and the other leading to classical $G$-spectra for general topological groups $G$, as in \autoref{classical}. 

Since $\SI\subset \PI$,  $X(\bn)$ and $X(\mathbf{1})^n$ are $(G\times \SI_n)$-spaces and $\de\colon X(\bn)\rtarr X(\mathbf{1})^n$ is a map  of $(G\times \SI_n)$-spaces for any $\PI$-$G$-space $X$. The following definitions are part of \cite[Definitions 2.14 and 2.35]{MMO}. Recall \autoref{famFn}.

\begin{defn}\label{weakFn} 
\begin{enumerate}[(i)]
\item A map $f\colon X\rtarr Y$ of $(G\times \SI_n)$-spaces is an  $\bF_n$-equivalence if  $f^{\LA}\colon X^{\LA} \rtarr Y^{\LA}$ is a weak equivalence for each $\LA\in \bF_n$.   
\item A map $f\colon X\rtarr Y$ of $\PI$-$G$-spaces is an  \gen-level equivalence if $f\colon X(\bn) \rtarr Y(\bn)$ is an $\bF_n$-equivalence for all $n \geq 0$. 
\item  A map $f$ of $\sD$-$G$-spaces is an \gen-level equivalence if $\bU f$ is a \gen-level equivalence of  $\PI$-$G$-spaces.
\item A map $f\colon X\rtarr Y$ of $\PI_G$-$G$-spaces is a level $G$-equivalence if each map $f\colon X(\bn^{\al})\rtarr Y(\bn^{\al})$ 
is a weak $G$-equivalence. 
\item   A map $f$ of $\sD_G$-$G$-spaces is a level $G$-equivalence if $\bU f$ is a level $G$-equivalence of  $\PI_G$-$G$-spaces.
\end{enumerate}
\end{defn}

Using these notions of weak equivalence, we rename things in the relevant special cases of \autoref{special}.

\begin{defn}\label{PIspecial}
  \begin{enumerate}[(i)]
\item   A $\PI$-$G$-space $X$ is \gen-special if 
$\de\colon X(\bn)\rtarr X(\mathbf{1})^n$ is an $\bF_n$-equivalence for all $n\geq 0$.
\item  A $\PI_G$-$G$-space $X$ is special if 
$\de\colon X(\bn^{\al}) \rtarr X(\mathbf{1})^{\al}$ is a $G$-equivalence for all
$\al\colon G\rtarr \SI_n$.
\end{enumerate}
\end{defn}

The following results are \cite[Lemmas 2.16 and 2.17]{MMO}. They follow from an explicit description of the $\LA_{\al}$-fixed points of $X^n$  for a $G$-space $X$  that is given in  \cite[Lemma 2.15]{MMO}. 

\begin{lem}\label{Rgen} If $f\colon X\rtarr Y$ is a weak equivalence of based $G$-spaces, then the induced map
$\bR f\colon \bR X\rtarr \bR Y$ is an \gen-level equivalence of $\PI$-$G$-spaces.
\end{lem}

\begin{lem}\label{iff}  Let $f\colon X\rtarr Y$ be an \gen-level equivalence of $\PI$-$G$-spaces.  Then $X$ is \gen-special if and only if $Y$ is \gen-special.  
\end{lem}

We defined \gen-level equivalences and \gen-special $\PI$-$G$-spaces  in terms of the families $\bF_n$ and thus in terms of homomorphisms $\al\colon H\rtarr \SI_n$ for all subgroups $H$ of $G$.  However, the following result, which is \cite[Lemma 2.18]{MMO},  says that we may restrict to $H=G$ and thus to finite $G$-sets.  

\begin{lem}\label{speciallevel} A map $f\colon X\rtarr Y$ of $\PI$-$G$-spaces is an \gen-level equivalence if the $f_n$ for  $n\geq 0$  are weak $\LA_{\al}$-equivalences for all subgroups $\LA_{\al}$ of $G\times \SI_n$ defined by homomorphisms $\al\colon G\rtarr \SI_n$.
A $\PI$-$G$-space $X$ is \gen-special if the maps $\de\colon X_n\rtarr X_1^n$ for $n\geq 0$ are weak $\LA_{\al}$-equivalences for all such $\LA_{\al}$. \end{lem}

This leads to the following  comparison, which is \cite[Theorem 2.38]{MMO}.

\begin{thm}\label{compFFG}  The following statements hold.
\vspace{1mm}
\begin{enumerate}[(i)]
\item A map $f\colon Y\rtarr Z$ of $\PI_G$-$G$-spaces is a level $G$-equivalence if and only if the map 
$\bU f\colon \bU Y\rtarr \bU Z$ of $\PI$-$G$-spaces is an \gen-level equivalence.  
\vspace{1mm}
\item A $\PI_G$-$G$-space $Y$ is special if and only if the $\PI$-$G$-space $\bU Y$ is \gen-special.
\vspace{1mm}
\item  A map $f$ of $\PI$-$G$-spaces is an \gen-level equivalence if and only if the map $\bP f$ of 
$\PI_G$-$G$-spaces is a level $G$-equivalence.
\item  A $\PI$-$G$-space $X$ is \gen-special if and only if the $\PI_G$-$G$-space $\bP X$ is special.
\vspace{1mm}
\end{enumerate}
These statements remain true with $\PI$ and $\PI_G$ replaced by $\sD$ and $\sD_G$.
\end{thm}

In turn, this leads to the analogues of Lemmas  \ref{Rgen} and \ref{iff}, which are \cite[Lemmas 2.41 and 2.42]{MMO}.

\begin{lem}\label{Rgen2} If $f\colon X\rtarr Y$ is a weak equivalence of based $G$-spaces, then the induced map
$\bR_Gf\colon  \bR_G X\rtarr \bR_GY$ is a level $G$-equivalence of $\PI_G$-$G$-spaces.
\end{lem}

\begin{lem}\label{iff2}  If $f\colon X\rtarr Y$ is a level $G$-equivalence of $\PI_G$-$G$-spaces, then $X$ is special if and only if $Y$ is special.
\end{lem}

The following specialization of \cite[Theorem 5.23]{MMO} tells us that restrictions of our basic constructions preserve special objects.

 \begin{thm}\label{homotopMT} 
\begin{enumerate}[(i)]
\item If $f\colon X\rtarr Y$ is an \gen-level equivalence of $\PI$-$G$-spaces,  then 
$\bD  f\colon \bD  X \rtarr \bD Y$ is an \gen-level equivalence, hence 
$\bD_G  \bP f\colon \bD_G  \bP X \rtarr \bD_G \bP Y$, is a level $G$-equivalence.
\item If $X$ is an \gen-special $\PI$-$G$-space, then $\bD X$ is \gen-special, hence
$\bD_G \bP X$ is a special $\PI_G$-$G$-space.
\end{enumerate}
\end{thm}

From here, the discussion in  Sections \ref{A}, \ref{BC}  and \ref{DE} makes clear that our Assumptions all hold.   We specialize Theorems \ref{recprincom} and \ref{adjequivcom} to the second diagram of \autoref{adj7}.  The analogous theorems hold for the first diagram, compatibly with respect to $(\bP,\bU)$. 

\begin{thm}\label{recprincom1}  There is a functor $\mathrm{Bar} \colon \DGalg\rtarr \DGalg$, written $Y\mapsto \overline{Y}$, and a natural equivalence  $\ze\colon \overline{Y} \rtarr Y$.  If $Y$ is special, the unit 
$$\et_{\bD_G}\colon \overline{Y} \rtarr (\bR_G\OM^{\infty})_{\bD_G} (\SI^{\infty}\bL_G)_{\bD_G} \overline{Y}$$  
is a group completion and is therefore an equivalence if $Y$ is grouplike.\end{thm}

\begin{thm}\label{adjequivcom2}  $\big((\SI^{\infty}\bL)_{\bD_G},(\bR\OM^{\infty})_{\bD_G}\big)$ induces an adjoint equivalence from the homotopy category of special grouplike $\bD_G$-algebras in $\PI_G[\ul{G\sT}]$ to the homotopy category of connective $G$-spectra.
\end{thm}

\begin{rem}  As in \autoref{finiteindex} and \cite[Section 9.3]{MMO}, everything above applies equally well to compact Lie groups $G$, provided that we restrict attention to representations of $G$ all of whose isotropy groups have finite index in $G$.  
\end{rem}

Following up \autoref{classical},  there is a second specialization of \autoref{adj7}, where we now start with a nonequivariant $E_{\infty}$ operad $\sC$ as defined in \autoref{oper1}, regarded as an operad of $G$-trivial $G$-spaces.  While the notion of an \gen-special $\PI$-$G$-space is only of real interest when  $G$ is finite, the following definitions apply to any topological group  $G$.\footnote{Remember that we require $e\in G$ to be a non-degenerate basepoint to avoid pathology.}  They should be compared with Definitions \ref{weakFn} and \ref{PIspecial}. 

\begin{defn} A map $f\colon X\rtarr Y$ of $\PI$-spaces is a {\em level $G$-equivalence} if each $f_n\colon X(\bn)\rtarr Y(\bn)$ is a weak 
$G$-equivalence.   A map of $\sD$-$G$-spaces is a level $G$-equivalence if its underlying map of $\PI$-$G$-spaces is a level $G$-equivalence.
\end{defn}

\begin{defn} A $\PI$-space $X$ is classically (or naively) special if  $\de\colon  X \rtarr \bR\bL X$ is a level $G$-equivalence.
A $\sD$-$G$-space $X$ is  classically (or naively) special if its underlying $\PI$-$G$-space is classically special.
\end{defn} 

The $\PI_G$ and $\sD_G$-$G$-space analogues of the previous two definitions are of no great interest.  They are obtained just by forgetting down to $\PI$-$G$-spaces and applying the definitions just given.  With these definitions, the two diagrams of \autoref{adj7} are tautologically equivalent homotopically, as well as equivalent formally.  Focusing
on the first diagram, we have classical (or naive) analogs of Theorems \ref{recprincom1} and \ref{adjequivcom2}. These apply to arbitrary (non-degenerately based) groups $G$ and refer to classical (or naive) $G$-spectra as defined in 
\autoref{classical}. 

\begin{thm}\label{recprincom3}  There is a functor $\mathrm{Bar} \colon \Dalg\rtarr \Dalg$, written $Y\mapsto \overline{Y}$, and a natural equivalence  $\ze\colon \overline{Y} \rtarr Y$.  If $Y$ is classically special, the unit 
$$\et_{\bD}\colon \overline{Y} \rtarr (\bR\OM^{\infty})_{\bD} (\SI^{\infty}\bL)_{\bD} \overline{Y}$$  
is a group completion and is therefore an equivalence if $Y$ is grouplike.\end{thm}

\begin{thm}\label{adjequivcom4}  $\big((\SI^{\infty}\bL)_{\bD},(\bR\OM^{\infty})_{\bD}\big)$ induces an adjoint equivalence from the homotopy category of special grouplike $\bD$-algebras in $\PI[\ul{G\sT}]$ to the homotopy category of connective classical $G$-spectra.
\end{thm}

\begin{rem}\label{SegalDis}
Using orthogonal rather than Lewis-May $G$-spectra, the Segal and operadic machines are proven to be equivalent in \cite{MMO}.  This means that when we regard $\sF_G$-$G$-spaces as examples of  $\sD_G$-$G$-spaces for a category of operators associated to an $E_{\infty}$ $G$-operad over ${\sK_G}$, application of the Segal and operadic machines gives equivalent orthogonal $G$-spectra.  One can compare the operadic machine with orthogonal  $G$-spectra as output with the operadic machine here by using the comparison of orthogonal  $G$-spectra to Lewis-May $G$-spectra in \cite[Chapter IV]{MM}.  That equivalence is multiplicative \cite[Theorem IV.1.1]{MM}.   

This gives a multiplicative Segal machine: apply the operadic machine  given by 
$B(\SI^{\infty}\bL_G, \bD_G, -)$  to $\sF_G$-$G$-spaces regarded as $\bD_G$-algebras by pullback along the projection  $\sD_G\rtarr \sF_G$.  From this standpoint, there is no need for the Segal machine as developed in \cite{MMO}.   The group completion property is already there from  our operadic point of view.   Retrospectively, in situations where our axioms are satisfied, there seems to be little advantage to the Segal machine.  It is symmetric monoidal, directly, as proven in \cite{GMMO}, but that does not accept the weakly structured input that is needed for the applications.  
\end{rem}

\section{The specialization to categories of orbital presheaves}\label{orbpre}

\subsection{The orbital presheaf context and a comparison of contexts}\label{GOTCOMP}

Let $G\sO$ denote the orbit category of a finite group $G$.  Its objects are the orbits $G/H$ and its morphisms are the $G$-maps.   Let $\OGop$ denote the category of contravariant functors, alias orbital presheaves, from $G\sO$ to the category $\sT$ of based spaces.   It is a long established fact \cite{EHCT, MM} that, with appropriate model structures,  the category $G\sT$ of based $G$-spaces is Quillen equivalent to $\OGop$.  The right adjoint $\bR\colon G\sT \rtarr \OGop$ sends a $G$-space $X$ to its fixed point functor, so that $(\bR X)(G/H) = X^H$.  The left adjoint $\bL$ sends an orbital presheaf  $\sX$  to the $G$-space $\sX(G/e)$.  Clearly $\bL\bR = \id$. 

In contrast to earlier parts of this paper, model structures will play a vital role in this section.   The weak equivalences in $\OGop$ are the levelwise weak equivalences, namely the maps $\sX\rtarr \sY$ such that each  $\sX(G/H)\rtarr \sY(G/H)$ is a weak equivalence of based spaces.   Thus, by definition, a map $f$ in $G\sT$ is a weak equivalence if and only if 
$\bR f$ is a weak equivalence in $\OGop$.  In contrast, as said in \autoref{Lstinks}, $\bL$ does {\em{not}} preserve weak equivalences: it takes a weak equivalence of orbital presheaves to a map of $G$-spaces that is a weak equivalence of underlying spaces.   Such weak equivalences are sometimes called Borel equivalences since only Borel cohomology, not Bredon cohomology, is invariant with respect to them.   

As we shall recall in the next subsection, it is a startling feature of the present context that $\bL$ and $\bR$ restrict to inverse isomorphisms between the respective categories of cofibrant objects.   We shall heavily exploit that fact.  It implies that cofibrant orbital presheaves $\sX$ are strictly special: the  unit $\et\colon \sX \rtarr \bR\bL \sX$ is an isomorphism.  In the previous section, the interesting $\PI$-$G$-spaces were those that are special, so that $\et$ is a weak equivalence, and strictly special objects gave no new information beyond that of the original adjunction of \autoref{ass1}.  The situation here is entirely different, and the following remark may help illuminate why.   Here we return for the moment to \autoref{ass7}.

\begin{rem}\label{light} For comparison of contexts, consider the diagram
$$\xymatrix{
\GA Y \ar[d]_{\ze} \ar[r]^{\et_{\GA Y}} & \bR\bL\GA Y \ar[d]^{\bR\bL \ze}  \\
Y \ar[r]_{\et_Y} &  \bR\bL Y\\}
$$
where $Y\in \sV$ and $\GA$ denotes a cellular cofibrant approximation functor, so that $\ze$ is a weak equivalence. The diagram makes sense in both the categories of operators and the orbital presheaves contexts and the functor $\bR$ preserves weak equivalences in both contexts.  In the category of operators context, $\bL$ also preserves weak equivalences.  Thus the diagram just tells us that $Y$ is special if and only if $\GA Y$ is special, which is the interesting case, and here there is no reason to make use of $\GA$.

In the orbital presheaves context,  the interesting case is when the presheaf $Y$ is {\em{not}} special, so generally {\em{not}}  cofibrant, and then 
$\GA$ turns any $Y$ into the weakly equivalent and strictly special $\GA Y$.  There is no contradiction since $\bR\bL \ze$ is generally not a weak equivalence.  We regard $\bR\bL Y$ as wrong and $\bR\bL\GA Y$ as right.  

Here classical cofibrant approximation plays a new and surprising role.  The composite $\bL\GA$ is a substitute for the Elmendorf construction of a $G$-space from an orbital presheaf.  That construction and its monadic elaboration were central features of the paper of Costenoble and Waner \cite{CW}, which was the first to study the orbital presheaf context.  While the Elmendorf construction is not needed with our treatment here, the monads of \cite{CW} are fundamental to the applications.   Because $\bL$ does not preserve weak equivalences, the monad $\bR\bC\bL$ is of no direct use beyond illumination of the theory; it only becomes usable after application of either cofibrant approximation or the Elmendorf construction to transform interesting orbital presheaves to equivalent orbital presheaves of the form $\bR X$.
\end{rem}

We have a new version of the Costenoble--Waner monads, and new applications.  We will give a new operadic characterization of Mackey functors  and we will construct  $G$-spectra of units, of Picard groups, and of Brauer groups of $E_{\infty}$ ring $G$-spectra.  Although that theory feeds into the present conceptual theory, it is of sufficient independent interest that we  present it in the companion paper \cite{KMZ2}.  In this section, we focus on the novelty of using cofibrant approximation to convert the orbital presheaves of interest,  which are  not strictly special, to equivalent orbital presheaves that are strictly special and feed into our machine.

\subsection{Model categories and cellular orbital presheaves}\label{Orbcells}

We use the standard model structures.   The fibrations and weak equivalences in $\OGop$ are the levelwise Serre fibrations and weak equivalences, and a map $f$ in $G\sT$ is a fibration or weak equivalence if and only if $\bR f$ is a fibration or weak equivalence.   These model structures are cofibrantly generated, in fact compactly generated in the sense of \cite[15.2.1]{morecon}.   The generating cofibrations $\sI$ and generating acyclic cofibrations $\sJ$  in  $G\sT$ are obtained by adjoining the identity map of disjoint $G$-fixed basepoints to the standard unbased generating cofibrations and acyclic cofibrations
$$  \xymatrix@1{G/H \times S^n \ar[r]^-{\subset}  & G/H\times D^n \ \ \text{and}  \ \ G/H\times D^n \ar[r]^-{i_0} & G/H\times D^n\times I.\\}$$

By definition, the generating cofibrations and generating acyclic cofibrations of $\OGop$ are obtained by applying $\bR$ to those of $G\sT$.  Therefore they are in bijective correspondence via $\bR$, with inverse $\bL$.  Cell $G$-spaces are constructed using disjoint unions, pushouts, and sequential unions.  Since passage to fixed points commutes with pushouts one leg of which is a $G$-cofibration (in the classical sense) and clearly commutes with unions, $\bR$ takes cellular objects to cellular objects, as does $\bL$.  Of course, as functors, $\bL$ and $\bR$ also  preserve retractions.   Since the cofibrant objects in either of our categories are the retracts of the cellular objects, this and \autoref{Rss}  are  all there is to the proof of the following result.

\begin{prop}\label{Gcellwonder}  The adjunction $(\bL,\bR)$ between $G\sT$ and $\OGop$ is a Quillen equivalence that restricts to an isomorphism of categories between the full subcategories of cofibrant  $G$-spaces and cofibrant orbital presheaves.  In particular, every cofibrant orbital presheaf is strictly special.
\end{prop}

Now return to the context of \autoref{SpaceSpectra}.  We choose to focus on spectra for definiteness.  Everything also works in the context of \autoref{GSPACES}.  Thus we have the adjunction $(\SI^{\infty},\OM^{\infty})$ and the monad $\bC$ associated to an $E_{\infty}$ operad that acts on $\OM^{\infty} E$ for all $G$-spectra $E$.    We have the adjunction $(\bL,\bR)$, and we assume that we also have a monad $\bD$ in $\OGop$ that satisfies \autoref{ass7}.  The obvious choice is $\bD = \bR \bC\bL$, in which case $\bC = \bL\bD\bR$ and $\bD\bR =\bR\bC$. 

\begin{rem}\label{orbitapps} Other choices will be given in the sequel \cite{KMZ2}.   For any operad $\sC$ in $G\sT$, we will there use the operad $\sC_G^{fin}$ to construct a monad $\Cp$ in $\OGop$ such that $\Cp\bR = \bR \bC$ and $\Cp$ satisfies the compatibility conditions of \autoref{ass7}.  Assuming that $\sC$ acts on $\OM^{\infty} E$ for $G$-spectra $E$, $\Cp$ acts on $\bR\OM^{\infty} E$.  As said before, it is the monads $\bD = \bC^{pre}$ that lead to applications. 
\end{rem}

Up to isomorphism, when we restrict to the category  $\OGop_{ss}$ of strictly special orbital presheaves, any choice of $\bD$ restricts to the canonical choice $\bR \bC\bL$.  With that restriction, \autoref{formalMT3} gives that we have the dotted  arrow functor $\bL$ in \autoref{adj6} and that the following result holds.

\begin{prop}\label{adjequiv1}  The adjoint equivalence $(\bL,\bR)$ between  $G\sT$ and $\OGop_{ss}$ restricts to an adjoint equivalence between  $\bC[G\sT]$ and $\bD\big[\OGop_{ss}\big]$.
\end{prop}

Note that an equivalence of a model category $\sM$ with another category $\sN$ induces a model structure on $\sN$ such that the equivalence is also a Quillen equivalence.   Using that, we can upgrade \autoref{adjequiv1} to Quillen equivalences that are given by actual equivalences of categories.  

Indeed,  the categories  $\bC[G\sT]$ and $\bD\big[\OGop\big]$  have standard structures of compactly generated model categories.  The fibrations and weak equivalences are created by forgetting to the ground categories $[G\sT]$ and $\OGop$, and the generating cofibrations and acyclic cofibrations are obtained by application of the left adjoints $\bF_{\bC}$ and $\bF_{\bD}$ to the generating cofibrations and acyclic cofibrations in the ground categories.  This is an easy application of the standard criterion (e.g. \cite[16.2.5]{morecon}) for a right adjoint to create a compactly generated model structure on its source.   Incorporating \autoref{adjequiv1}, the proof of the following result is precisely analogous to the proof of \autoref{Gcellwonder}.

\begin{prop}\label{Gcellwonder2}  The Quillen equivalence $(\bL,\bR)$ between  $G\sT$ and $\OGop$ restricts to a Quillen equivalence between $\bC[G\sT]$ and  $\bD\big[\OGop_{ss}\big]$, and that restricts to an isomorphism of categories between the full subcategories of cofibrant  $\bC$-algebras and cofibrant $\bD$-algebras.  Therefore $(\bL,\bR)$  restricts to equivalences between $G\sT$ and $\OGop_{ss}$  and between  $\bC[G\sT]$ and $\bD\big[\OGop_{ss}\big]$.
\end{prop}

\subsection{Orbital presheaves in the composite adjunction context}\label{Orbcomp}

Here our philosophy needs psychological but not mathematical change.   Consider the following diagrams.  

\begin{equation}\label{adj8} 
\xymatrix{
\OGop \ar@<.5ex>[rr]^{\bL}  \ar@<.5ex>[ddrr]^{\bF_{\bD}} & & G\sT  \ar@<.5ex>[ll]^{\bR}    \ar@<.5ex>[rr]^{\SI^{\infty}}  \ar@<.5ex>[dr]^{\bF_{\bC}}& & G\sS   \ar@<.5ex>[ll]^{\OM^{\infty}}   \ar@<.5ex>[dl]^{\OM^{\infty}_{\bC}}   
\ar@/^4pc/@<.6ex> [ddll]^-{(\bR\OM^{\infty})_{\bD}}        \\
& & &   \ar@<.5ex>[ul]^{\bU_{\bC}}  \bC[G\sT]   \ar@<.5ex>[dl]^{\bR}  \ar@<.5ex>[ur]^{\SI^{\infty}_{\bC}}& \\
& & \ar@<.5ex>[uull]^{\bU_{\bD}}  \bD[\OGop] \ar@{-->}@<.5ex>[ur]^{\bL} \ar@/_4pc/@<.6ex> [uurr]^-{(\SI^{\infty} \bL)_{\bD}}& &  \\
&  & \ar[d]^{\GA}  && \\
& & & & \\
\OGop_{ss} \ar@<.5ex>[rr]^{\bL}  \ar@<.5ex>[ddrr]^{\bF_{\bD}} & & G\sT  \ar@<.5ex>[ll]^{\bR}    \ar@<.5ex>[rr]^{\SI^{\infty}}  \ar@<.5ex>[dr]^{\bF_{\bC}}& & G\sS   \ar@<.5ex>[ll]^{\OM^{\infty}}   \ar@<.5ex>[dl]^{\OM^{\infty}_{\bC}}   
\ar@/^4pc/@<.6ex> [ddll]^-{(\bR\OM^{\infty})_{\bD}}        \\
& & &   \ar@<.5ex>[ul]^{\bU_{\bC}}  \bC[G\sT]   \ar@<.5ex>[dl]^{\bR}  \ar@<.5ex>[ur]^{\SI^{\infty}_{\bC}}& \\
& & \ar@<.5ex>[uull]^{\bU_{\bD}}  \bD[\OGop_{ss}] \ar@<.5ex>[ur]^{\bL} \ar@/_4pc/@<.6ex> [uurr]^-{(\SI^{\infty} \bL)_{\bD}}& &  \\}
\end{equation}

We think of the top diagram as the input diagram that leads to applications.  We apply cellular cofibrant approximation $\GA$ to feed the top diagram into the bottom diagram without loss of information.  The bottom diagram gives homotopical control that allows us to implement our machine; as noted in \autoref{sillybit}, we might as well erase its curved arrows.  We have noted earlier that the bar construction can often be viewed as a cofibrant approximation, but we have no interest in that interpretation here.   Rather, we just note that \autoref{YDY} and \autoref{different} imply that if an orbital presheaf $Y$ is strictly special, then so is the orbital presheaf $\overline{Y}$.  
Looking back at Sections \ref{A}, \ref{BC} and \ref{DE}, we see that all of our Assumptions are satisfied in the composite adjunction context displayed in the bottom diagram.  Thus Theorems \ref{recprincom} and \ref{adjequivcom} specialize to give the following conclusions.

\begin{thm}\label{recprincom2Ga}  There is a functor $\mathrm{Bar} \colon \bD[\OGop] \rtarr \bD[\OGop] $, written $Y\mapsto \overline{Y}$, and a natural equivalence  $\ze\colon \overline{Y} \rtarr Y$.  If $Y$ is strictly special, then so is $\overline{Y}$ and the unit 
$$\et_{\bD}\colon \overline{Y} \rtarr (\bR\OM^{\infty})_{\bD} (\SI^{\infty}\bL)_{\bD} \overline{Y}$$  
is a group completion and is therefore an equivalence if $Y$ is grouplike.
\end{thm}

\begin{thm}\label{adjequivcom3}  $\big((\SI^{\infty}\bL)_{\bD},(\bR\OM^{\infty})_{\bD}\big)$ induces an adjoint equivalence from the homotopy category of grouplike $\bD$-algebras in $\OGop_{ss}$ to the homotopy category of connective $G$-spectra.
\end{thm}

\begin{rem}\label{Apps}
We comment on how applications proceed.  We find examples of $\bD$-algebras $Y$.  As said before $Y$ is not strictly special or even special in the cases of  interest.  We therefore replace $Y$ by $\GA Y$ and apply \autoref{recprincom2Ga} with $Y$ replaced by $\GA Y$.   Because $\GA Y$ is strictly special, the $\bD$-algebra structure on $\GA Y$  comes from application of $\bR$ to the induced $\bC$-algebra structure on $\bL \GA Y$.  As said before, there is no contradiction since $\bL\GA Y$ is generally not weakly equivalent to $\bL Y$.   The $G$-spectrum we are after is $\SI^{\infty}_{\bC}(\overline{\bL \GA Y})$ or, equivalently, $(\SI^{\infty}\bL)_{\bD}(\overline{\GA Y})$.  
\end{rem}

\part{The multiplicative theory}

\section{The classical multiplicative theory}\label{Mult1}
\subsection{Monad pairs}\label{pairs}
Consider monads, $(\bC,\mu_{\oplus},\et_{\oplus})$ and 
$(\bJ_0,\mu_{\otimes},\et_{\otimes})$, on a category $\sT$. 
As the notation indicates, we think of $\bC$ as ``additive'' and $\bJ_0$ as ``multiplicative with zero''.  We think of $\sT$ as based and, intuitively, we think of basepoints as zero.  The basepoint of $\bJ_0 X$ for $X\in \sT$ is thought of as zero.

Beck \cite{Beck} gives a monadic distributivity law encoded by an action of 
$\bJ_0$ on $\bC$.  We recall the main features, following \cite[Appendix B]{Rant2}.
Let $\bC[\sT]$ and $\bJ_0[\sT]$ denote the categories of $\bC$-algebras and of
$\bJ_0$-algebras in $\sT$. 

\begin{defn}\label{COact} An action of $\bJ_0$ on $\bC$ is a structure of monad 
on $\bJ_0[\sT]$ induced by the monad $\bC$ on $\sT$. More explicitly, for an action 
 of $\bJ_0$ on $X$, there is a prescribed functorial induced action of $\bJ_0$ 
 on $\bC X$ (and thus on $\bC\bC X$ by iteration) such that
$\et_{\oplus}\colon X\rtarr \bC X$ and $\mu_{\oplus}\colon \bC\bC X\rtarr \bC X$
are maps of $\bJ_0$-algebras.  We then  say that $(\bC,\bJ_0)$ is a monad pair.
\end{defn}


 The proof of the following result is sketched in \cite[Appendix B]{Rant2}; full details are in Beck \cite{Beck}.

\begin{thm}\label{BeckthmD} The following data relating the monads $\bC$ and $\bJ_0$ 
are equivalent.
\begin{enumerate}[(i)]
\item An action of $\bJ_0$ on $\bC$.
\item A natural transformation $\mu\colon \bC\bJ_0\bC\bJ_0\rtarr \bC\bJ_0$ with 
the following properties.
\begin{enumerate}[(a)]
\item $(\bC\bJ_0,\mu,\et)$ is a monad on $\sT$, where 
$\et = \et_{\oplus}\bJ_0 \com \et_{\otimes}\colon \text{Id}\rtarr \bC\bJ_0$.
\item $\bC\et_{\otimes}\colon \bC\rtarr \bC\bJ_0$ and $\et_{\oplus} \colon \bJ_0\rtarr \bC\bJ_0$
are maps of monads.
\item  The following composite is the identity natural transformation.
\[ \xymatrix@1{
\bC\bJ_0 \ar[r]^-{\bC\et_{\oplus} \bJ_0} & \bC\bC\bJ_0\ar[r]^-{\bC\et_{\otimes}\bC\bJ_0 } & \bC\bJ_0\bC\bJ_0
\ar[r]^-{\mu} & \bC\bJ_0 \\} \]
\end{enumerate}
\item A natural transformation $\rh\colon \bJ_0\bC\rtarr \bC\bJ_0$ such that
the following diagrams commute.
\[\xymatrix{
& \bJ_0 \ar[dl]_{\bJ_0\et_{\oplus}} \ar[dr]^{\et_{\oplus}\bJ_0} & \\
 \bJ_0\bC \ar[rr]^-{\rh} & & \bC\bJ_0 \\ 
& \bC \ar[ul]^{\et_{\otimes}\bC} \ar[ur]_{\bC\et_{\otimes}} & \\ } 
\ \ \ \ 
\xymatrix{
\bJ_0\bC\bC \ar[d]_{\bJ_0\mu_{\oplus}} \ar[r]^-{\rh \bC} & \bC\bJ_0\bC \ar[r]^{\bC\rh} 
& \bC\bC\bJ_0 \ar[d]^{\mu_{\oplus}\bJ_0} \\ 
\bJ_0\bC \ar[rr]^-{\rh} &   & \bC\bJ_0 \\ 
\bJ_0\bJ_0\bC \ar[r]_-{\bJ_0\rh} \ar[u]^{\mu_{\otimes}\bC} & \bJ_0\bC\bJ_0 \ar[r]_-{\rh \bJ_0} & \bC\bJ_0\bJ_0.
\ar[u]_{\bC\mu_{\otimes}} \\ }\]
\end{enumerate}
When given such data, the category $\bC[\bJ_0[\sT]]$ of $\bC$-algebras in
$\bJ_0[\sT]$, called $(\bC,\bJ_0)$-algebras in $\sT$,  is isomorphic to the category $\bC\bJ_0[\sT]$ of $\bC\bJ_0$-algebras in $\sT$. 
\end{thm}

\subsection{Operad pairs}\label{pairs2}

In retrospect,\footnote{When he defined operad pairs \cite{MQR}, the senior author did not know Beck's work.}  the definition of an operad pair $(\sC,\sJ)$ can be summarized as follows.   As usual, we assume that $\sC$ and $\sJ$ are reduced, meaning that $\sC(0)$ and $\sJ(0)$ are each a point (a terminal object of $\sT$); we name these objects $0$ and $1$, respectively.  We assume that $\sT$ is the category of based objects of a category $\sU$ of unbased objects. The $\sC(j)$ and $\sJ(j)$ are objects of $\sU$.   Notationally, we assume that $\sT$ is cartesian monoidal for consistency with the literature, but that is not essential to the mathematics.  We use the following definition to define the monads $\bC$ and $\bJ_0$ conceptually and explicitly. 

\begin{defn}  Recall the category $\LA$ from \autoref{finitecats}.   It is generated by the permutations of the $\ul{n}$ and the ordered injections 
$\si_i\colon \ul{n-1} \rtarr \ul{n}$ that skip $i$.   For a contravariant functor  $T\colon \LA \rtarr \sU$ and a covariant functor 
$S\colon \LA \rtarr \sU$ with actions 
$$  \ta\colon T(\ul{n}) \times \LA(\ul{m},\ul{n})\rtarr T(\ul{m}) \ \ \text{and} \ \ 
\si\colon \LA(\ul{m},\ul{n})\times S(\ul{m}) \rtarr S(\ul{n}),$$  
the categorical tensor product  $T\otimes_{\LA} S$ is the coequalizer

$$ \xymatrix@1{
\coprod_{m, n}  T(\ul{n}) \times \LA(\ul{m}, \ul{n})\times  S(\ul{m})
\ar@<.5ex>[r]^-{\ta\times \id}  \ar@<-.5ex>[r]_-{\id\times \si} 
&  \coprod_j T(\ul{j}) \times S(\ul{j}) \ar[r] &  T\otimes_{\LA} S.\\} $$
More explicitly
$$  T\otimes_{\LA} S = \coprod_j T(j)\times_{\SI_j} S(j)/(y\si_ i, x) \sim (y, \si_i x)$$
for all  $y\in T(\ul{n})$, $x\in S(\ul{n-1})$ and all ordered injections $\si_i\colon \ul{n-1} \rtarr \ul{n}$.
\end{defn}

\begin{defn}\label{OtoJ}   An operad $\sC$ has the underlying contravariant functor $\LA\rtarr \sU$ given by sending $\ul{j}$ to $\sC(j)$; the functoriality is given by the right actions of the 
$\SI_j$ on the objects $\sC(j)$ and the maps $\si_i\colon \sC(n) \rtarr \sC(n-1)$ for ordered injections $\si_i$ specified by
$$  c\si_i = \ga(c; \id^{i-1}\times\, 0 \times \id^{n-i})$$
for $c\in \sC(n)$, where $\id\in \sC(1)$ is the identity element and $0\in \sC(0)$ is the element $0$.   For a based object  $X\in \sT$, with ``basepoint'' understood  to be a map 
$$\io\colon \ast = 0\rtarr X,$$
define a covariant functor  $X^*\colon \LA\rtarr \sU$ by $X^*(\ul{j}) = X^j$ and note that it takes values in $\sT$; the functoriality is given by the permutation left actions of the $\SI_j$ on the $X^j$ and the maps 
$$
 \si_i = \id^{i-1} \times \io \times \id^{j-i}\colon  X^{n-1} \rtarr X^n
$$
for ordered injections $\si_i$.  The monad $\bC$ on $\sT$ associated to $\sC$ is defined by
$$\bC X =  \sC\otimes_{\LA} X^*.$$
The coequalizer is defined in $\sU$ but takes values in $\sT$, with basepoint given by  $\{0\}= \sC(0)\times X^0$;  
$\mu_{\oplus}\colon  \bC\bC X \rtarr \bC X$ is induced by the structure maps 
$$\ga\colon \sC(k) \times \sC(j_1)\times\cdots  \times \sC(j_k)  \rtarr \sC(j), \ \ \text{where} \ \  j = j_1 + \cdots + j_k,$$
and $\et_{\oplus}$ is induced by $\id\colon  X \rtarr {\id}\times X\in \sC(1) \otimes X$.  Less formally,
$$\bC X =  \coprod_j  \sC(j)\times_{\SI_j} X^j/\text{basepoint identifications}.  $$
\end{defn}

We modify the definition of $\bC$ to define the monad $\bJ_0$ associated to $\sJ$.  We change the ground symmetric monoidal structure from cartesian product to smash product.   We define an operad $\sJ_0$, no longer reduced, by adjoining disjoint basepoints $0$ to the $\sJ(j)$, so that $\sJ(j)_0 = \sJ(j)\coprod \{0\}$.   In practice, we start with spaces under $S^0$, so given maps  $e\colon S^0\rtarr X$ and thus basepoints $0$ and $1$, but the formal theory sees only the basepoint $0$; that is, the basepoint $1$ only comes into play when $X$ is a $\sJ$-algebra.  As explained in a bit more detail in \cite[Section 4]{Rant1}, we arrive at the monad in $\sT$ given by
$$\bJ_0 X = \bigvee_{j\geq 0} \sJ(j)_0 \sma_{\SI_j}  X^{(j)},  $$
where $X^{(j)}$ is the $j$-fold smash product defined with respect to the basepoint ${0}$.  For obvious reasons, the construction is only of interest when $0$ and  $1$ are in distinct path components.  The construction is of particular interest when 
 $X = Y_0 =Y\coprod \{0\}$ for an unbased object $Y$, and then
$$\bJ_0 Y_0 = \bigvee_{j\geq 0} \big(\sJ(j) \times_{\SI_j}  Y^{j}\big)_0  $$

\begin{defn}  An action of $\sJ$ on $\sC$ consists of maps
$$\la\colon \sJ(k)\times \sC(j_1)\times \cdots \times \sC(j_k) \rtarr  \sC(j_{\times}), \ \ \text{where} \ \ 
j_{\times} = j_1\cdots j_k, $$
which satisfy distributivity, unity, equivariance, and nullity identities relating the maps $\la$ to the structure maps, permutation group actions, and units of the operads $\sC$ and $\sJ$. These identities are explicitly specified in \cite[Definition 4.2]{Rant1}.  They are precisely those needed to ensure that  $\bJ_0$ acts on $\bC$.  Thus an operad pair $(\sC,\sJ_0)$ has an associated monad pair $(\bC,\bJ_0)$.
\end{defn}

\subsection{Operad pairs in the category of $G$-spaces}\label{pairs3}   So far, we have been general.  We now specialize to $G$-spaces, where we have a canonical operad pair $(\sK,\sL)$.   Here $\sK$, denoted 
$\sK_{\infty}$ in \autoref{SPECTRA}, is the Steiner operad of $G$-spaces in the universe $U_G$, as described in \cite[Example 2.2]{GM3}.  It is the colimit of the Steiner operads $\sK_V$  associated to the finite dimensional sub  $G$-spaces  $V$ of $U_G$.  The $\sK_V$ are described in detail and compared with the little $V$-disk operads $\sD_V$ in 
\cite[Sections 1.1 and 2.3]{GM3}, where the advantages of the $\sK_V$ over the $\sD_V$ are made clear.   The operad $\sL$ is the infinite linear isometries operad in the universe $U_G$.  It was defined in \cite{LMS} and is described in \cite[Example 2.2]{GM3}; we  will recall its definition in the next subsection.   The finite dimensional precursor of the action of $\sL$ on $\sK$ is described in \cite[Definition 1.22]{GM3}.  The action of $\sL$ on $\sK$ is defined nonequivariantly in \cite[Section 3]{Rant1}, and it works the same way equivariantly. 

\begin{rem}\label{dumb} It is an embarrassment that we know of few other operad pairs, a problem that will be alleviated in \cite{Mop}.   Nevertheless, it is sensible to work more generally with an operad pair $(\sC,\sJ)$ that maps by a pair of maps of operads to $(\sK,\sL)$.   Examples might be of  the form  
$(\sC,\sJ) =(\sK\times \sO,\sL\times \sP)$ for an operad pair $(\sO,\sP)$, where $\sO$ and $\sP$ are also $E_{\infty}$ operads.  Dropping the requirement that $\sO$ and $\sP$ are $E_{\infty}$ operads, we can also take $\sO$ to be the commutativity operad $\sN$ and $\sP$ to be any operad.  This makes sense since $\sK\times \sN \iso \sK$ and since any operad $\sP$ acts on $\sN$ (because each $\sN(j)$ is a point).
\end{rem}

\subsection{Review of twisted half-smash products and $\sL$-spectra}\label{SPECTRA2}

Historically, $E_{\infty}$ ring spaces and  $E_{\infty}$ ring spectra were first defined in 1972, although a printed definition did not appear until  \cite{MQR}.   That was well before adequate foundations to do justice to the notions were available.   The later introduction of twisted half-smash products (THPs) gave foundations in the context of spectra and $G$-spectra, as  described in \autoref{SPECTRA}.  While THPs are defined more generally and have other uses, they were originally designed to internalize a simple external smash product of $G$-spectra, and that is how we shall use them here.  Their definition first appeared in print in \cite[Chapter VI]{LMS}. They were applied there and throughout \cite{BMMS}.  

Another decade later, THPs were beautifully redefined conceptually by Cole in the appendix to \cite{EKMM} and, equivariantly, in \cite[Chapter XXII]{EHCT}.  In retrospect, EKMM should be called CEKMM because Cole's appendix, which is written entirely in our present context of \autoref{SPECTRA}, reinterpreted the foundations upon which \cite{EKMM} was built.   Cole never mentioned the word ``operad" in either of the cited references. However,  \cite[Chapter VII]{LMS}, which is entitled ``Operad ring spectra'', gives exactly the equivariant foundations we need and is quite readable, provided that we take Cole's later redefinition of the twisted half-smash product as a foundational black box preliminary to that account.  

We summarize what we need following and modifying the nonequivariant summary in \cite{Rant1}, which gives more background and compares various alternative monads with  isomorphic categories of algebras.   Here we use the linear isometries operad $\sL$. Its $j$th space $\sL(j)$ is the $(G\times \SI_j)$-space of linear isometries $U_G^j \rtarr U_G$, with (left) $G$-action by conjugation and (right) $\SI_j$-action induced by the (left) $\SI_j$-action on $U_G^j$.  It is an $E_{\infty}$ operad \cite[Example 2.3]{GM3}. 

For a $j$-tuple of $G$-spectra $E_i$ indexed on $U_G$,   we have an external product    $\barwedge_i  E_i$, which is a $G$-spectrum indexed on sums ($=$ products) $V_1\oplus \cdots \oplus V_j$ in the universe $U_G^j$. The THP  $\sL(j)\thp \barwedge_i  E_i$ internalizes\footnote{Internalization is treated in general categorical terms in work in progress by Bryce Goldman.}
  this external smash product to what we shall view as the internal smash product of the $E_i$; it is  a $G$-spectrum indexed on $U_G$.  This smash product is not symmetric monoidal.  It is modified to an equivalent internal smash product in the category of $S_G$-modules constructed in \cite{EKMM}, but at the price of losing the homotopically meaningful adjunction 
 $(\SI^{\infty},\OM^{\infty})$ that we require.  As said earlier, the loss is dictated by Lewis's result \cite{Lewis} that we cannot have all of the good properties we want in the same category.   This internal smash product is suitable for defining actions of the operad $\sL$ on a $G$-spectrum $E$.  
 
 \begin{defn} An $E_{\infty}$ ring $G$-spectrum, or $\sL$-spectrum, is a 
$G$-spectrum $E$ with an
action of $\sL$ given by an equivariant, unital, and associative system
of maps
\begin{equation}\label{Eringspec}
\xi_j\colon  \sL(j)\thp E^{[j]}\rtarr E,
\end{equation}
where $E^{[j]}$ denotes the $j$th external smash power of $E$.
\end{defn}

The unspecified diagrams are exactly like those in the original definition of an action  of an operad on a space.  There are canonical maps
\begin{equation}\label{butyes} 
\xymatrix{
\sL(k)\thp (\sL(j_1)\thp E^{[j_1]}\barwedge \cdots \barwedge
\sL(j_k)\thp E^{[j_k]}) \ar[d]^{\iso} \\
(\sL(k)\times \sL(j_1)\times \cdots \times \sL(j_k))\thp E^{[j]}
\ar[d]^{\ga\thp\id} \\
\sL(j)\thp E^{[j]},\\}
\end{equation}
where $j= j_1 +\cdots + j_k$.  We use such maps to make sense of the required diagrams. 

In practice, we usually first construct an $\sL$-prespectrum $T$ and then observe that its spectrification $E = LT$ is an $\sL$-spectrum since 

\[ L(\sL(j)\thp T^{[j]})\iso \sL(j)\thp (LT)^{[j]}. \]

We take $G$-spectra indexed on the zero universe to be based $G$-spaces.   Since $\sL(0)$ is the inclusion $0\rtarr U$, the
 functor $\sL(0)\thp (-)$ from $G$-spaces to $G$-spectra can and must be interpreted as the functor $\SI^{\infty}\colon G\sT\rtarr G\sS$.  Similarly, the zero fold external smash power $E^{[0]}$ can and must be interpreted as the $G$-space $S^0$.  Since $\SI^{\infty}S^0$ is the sphere $G$-spectrum $S_G$, the $0^{th}$ structure map in (\ref{Eringspec}) is a map $e\colon S_G\rtarr E$.   We can either think of the map $e$ 
as preassigned, in which case we think of our ground category as the category $G\sS_e$ of spectra under $S_G$, or we can think of $e =\xi_0$ as part of the structure of an $E_{\infty}$ ring spectrum, in which case we think of our ground category as $G\sS$.  

We take the latter perspective. In analogy with the $G$-space level monad $\bL_0$ we define a monad $\bL_0$
on the category of $G$-spectra by letting
\[  \bL_0 E = \bigvee_{j\geq 0} \sL(j)\thp_{\SI_j} E^{[j]}. \]
The $0^{th}$ term is $S_G$, and $\et\colon S_G\rtarr \bL_0E$ is the inclusion.
The product $\mu$ is induced by passage to orbits and wedges from the canonical maps \autoref{butyes}.  

A central feature of twisted half-smash products is that there is a natural untwisting isomorphism
\begin{equation}\label{untwist}
\sL(j)\thp (\SI^{\infty} X_1\barwedge \cdots \barwedge \SI^{\infty} X_j)
\iso  
\SI^{\infty}(\sL(j)_0\sma X_1\sma\cdots X_j)
\end{equation}
for based $G$-spaces $X_i$.  Using this, we obtain a monadic natural isomorphism
\begin{equation}\label{newmon}
\bL_0\SI^{\infty}X \iso \SI^{\infty} \bL_0X
\end{equation}
relating the space and spectrum level monads $\bL_0$.

Using the subscript $+$, it has become usual to compress notation by defining 
\begin{equation}\label{sigmaplus}
 \SI^{\infty}_+ Y = \SI^{\infty}(Y_+)
\end{equation}
for an unbased $G$-space $Y$, ignoring its given basepoint if it has one.   Then
\begin{equation}\label{oldmon}
\bL_0\SI^{\infty}_+ Y \iso \SI^{\infty}_+  \big(\coprod_{j\geq 1}\sL(j) \times_{\SI_j}  Y^{j}\big).
\end{equation}
The most interesting examples are of this form.

\subsection{The multiplicative context of $G$-spaces and $G$-spectra}\label{MultThy}

We take $\sT$  in our general theory to be the category $\bJ_0[G\sT]$ of $\bJ_0$-algebras in $G\sT$ (alias $\bJ_0$-$G$-spaces), and we take $\sS$ to be the category $\bJ_0[G\sS]$ of $\bJ_0$-algebras in $G\sS$ (alias $\bJ_0$-$G$-spectra), where $(\bC,\bJ_0)$ is a monad pair as in \autoref{pairs}.   As said before, we have two operadic examples.
For the operad pair example treated here, we  take the monad pair associated to a pair $(\sC,\sJ)$ of $E_{\infty}$ operads of $G$-spaces that maps equivalently  to the pair $(\sK,\sL)$, as in \autoref{dumb}.   We are changing the ground categories from $G$-spaces and $G$-spectra to $\bJ_0$-$G$-spaces and $\bJ_0$-$G$-spectra.   Here, for the latter, we can pull back $\bL_0$-$G$-spectra along the projection $\bJ_0\rtarr \bL_0$.  We claim that this puts us in the framework of \autoref{ass1}, so that  we have the following special case of the diagram \autoref{adj5}. 

\begin{equation}\label{adj9} 
\xymatrix{
\bJ_0[G\sT]    \ar@<.5ex>[rr]^{\SI^{\infty}}  \ar@<.5ex>[dr]^{\bF_{\bC} }& & \bJ_0[G\sS]   \ar@<.5ex>[ll]^{\OM^{\infty}}   \ar@<.5ex>[dl]^{\OM^{\infty}_{\bC}}\\
&   \ar@<.5ex>[ul]^{\bU_{\bC}}  \bC\big[\bJ_0[G\sT]\big       ]    \ar@<.5ex>[ur]^{\SI^{\infty}_{\bC}}& \\}
\end{equation}

To prove the claim, we must first prove the following result.  To distinguish notationally, let us write $X$ for based $G$-spaces, $Z$ for $\bJ_0$-$G$-spaces, 
$E$ for $G$-spectra, and $R$ for $\bJ_0$-$G$--spectra, which are $E_{\infty}$ ring $G$-spectra with structure given by $\bJ_0$. 

\begin{prop}\label{formalQ} The functors $\SI^{\infty}$ and $\OM^{\infty}$ take $\bJ_0$-algebras to $\bJ_0$-algebras and the (topological) adjunction
\[  G\sS(\SI^{\infty}X,E)\iso G\sT(X,\OM^{\infty}E) \]
induces a (topological) adjunction
\[  \bJ_0[G\sS](\SI^{\infty}Z,R) \iso \bJ_0[G\sT](Z,\OM^{\infty}R). \]
Thus the monad $Q = \OM^{\infty}\SI^{\infty}$ on $G\sT$ restricts to a monad $Q$ on 
$\bJ_0[G\sT]$ and, when restricted to $\bJ_0$-$G$-spectra, the 
functor $\OM^{\infty}$ takes values in $\bJ_0$-$G$-spaces.
\end{prop}

\begin{proof} This is a formal consequence of the fact that the
isomorphism (\ref{newmon}) is monadic, as is explained in
general categorical terms in \cite[Appendix A]{Rant1}.  
If $(Z,\xi)$ is a $\bJ_0$-$G$-space, 
then $\SI^{\infty}Z$ is a $\bJ_0$-$G$-spectrum with structure map
\[ \xymatrix@1{ \bJ_0\SI^{\infty}Z \ar[r]^-{\iso} & 
\SI^{\infty}\bJ_0Z \ar[r]^-{\SI^{\infty}\xi}
& \SI^{\infty}Z.\\} \]
The isomorphism (\ref{newmon}) and the 
adjunction give a natural composite $\de$: 
\[\xymatrix@1{  
\bJ_0\OM^{\infty}E \ar[r]^-{\et} & 
\OM^{\infty}\SI^{\infty}\bJ_0\OM^{\infty}E \ar[r]^-{\iso} 
& \OM^{\infty}\bJ_0\SI^{\infty}\OM^{\infty}E \ar[r]^-{\OM^{\infty}\bJ_0\epz}
& \OM^{\infty}\bJ_0E.\\} \]
If $(R,\xi)$ is a $\bJ_0$-$G$-spectrum, then 
$\OM^{\infty}R$ is a $\bJ_0$-$G$-space with structure map
\[ \xymatrix@1{
\bJ_0\OM^{\infty}R \ar[r]^-{\de} & \OM^{\infty}\bJ_0 R\ar[r]^-{\OM^{\infty}\xi} 
& \OM^{\infty}R.\\} \]
Diagram chases show that the unit $\et$ and counit $\epz$ of the
adjunction are maps of $\bJ_0$-algebras when $Z$ and $R$ are $\bJ_0$-algebras, and the restricted adjunction follows.
\end{proof}

Taking $X = Y_+$, with the basepoint thought of as $0$ and using the notation of \autoref{sigmaplus}, this has the following consequence. 

\begin{cor}\label{themguys} The adjunction of Proposition \ref{formalQ} induces an adjunction
\[ \bJ_0[G\sS](\SI^{\infty}_+Y,R)\iso \bJ[G\sU](Y,\OM^{\infty}R) \]
between the category of $\bJ_0$-$G$-spectra $R$ (alias $E_{\infty}$-ring $G$-spectra) and the category of $\bJ$-$G$-spaces  $Y$, 
where $\bJ$-$G$-spaces are unbased $\bJ$-spaces, that is $\bJ$-spaces without $0$. 
\end{cor}
\begin{proof}  Recall that we have the obvious adjunction 
\[  G\sT(Y_+,X) \iso G\sU(Y,iX) \]
for based $G$-spaces $X$ and unbased $G$-spaces $Y$, where $i\colon \sT\rtarr \sU$ is given by forgetting basepoints.  It induces an adjunction
\[ \bJ_0[G\sT](Y_+,Z)\iso \bJ[G\sU](Y,Z) \]
for $\bJ$-$G$-spaces $Y$ and $\bJ_0$-$G$-spaces $Z$ regarded on the right as $\bJ$-spaces by forgetting the basepoint $0$.  Taking $Z=\OM^{\infty}R$, the result follows by composing with the adjunction of Proposition \ref{formalQ}.
\end{proof}

Fixing $(\sC,\sJ)$,  we take $(\sC,\sJ)$-$G$-spaces as our $E_{\infty}$ ring $G$-spaces.  Similarly, we take $\bJ_0$-$G$-spectra as our
$E_{\infty}$ ring $G$-spectra.   We then have the following result.

\begin{cor} 
The $0^{th}$ $G$-spaces of $E_{\infty}$ ring $G$-spectra are
naturally $E_{\infty}$ ring $G$-spaces.
\end{cor}

The following result makes clear that \autoref{ass1} restricts multiplicatively to another instance of \autoref{ass1}. 

\begin{cor} The natural action of the monad $\bC$ in $G\sT$ on $\OM^{\infty}E$ for $G$-spectra $E$ restricts on $\bJ_0$-$G$-spectra $R$ to a natural action of the monad $\bC$ in $\bJ_0[G\sT]$.
\end{cor} 
\begin{proof} Recall that $\bC$ maps to $\bK$ and $\bK$ acts naturally on $\OM^{\infty} E$.  The action is induced as in (i) of \autoref{converse} from the unit and counit of the monad $\OM^{\infty}\SI^{\infty}$, and those are maps of $\bJ_0$-algebras when their input is given by $\bJ_0$-algebras.
\end{proof}

\subsection{Assumptions B through E and the ring completion theorem}\label{RingThy}

Forgetting from \autoref{adj7}  down to the diagram without  the $\bJ_0$, we define the weak equivalences of \autoref{ass2} to be created in the ground categories of $G$-spaces and $G$-spectra.  We say that an $E_{\infty}$ ring $G$-space is ringlike if it is grouplike as a $G$-space, and we say that a map of $E_{\infty}$ ring $G$-spaces is a ring completion if its underlying map of $G$-spaces is a group completion.  Checking that the group completion functor $(\bG,g)$ preserves $E_{\infty}$ ring $G$-spaces, we have \autoref{ass4} by specialization of its additive version to  $(\bC,\bJ_0)$-algebras. 

The approximation theorem, \autoref{ass5}, concerns the underlying additive structure and is therefore immediate from its additive version.   The essential formal point is that if $X$ is a $\bJ_0$-algebra, then $\al\colon \bC X \rtarr \OM^{\infty}\SI^{\infty} X=  QX $ is a map of $\bJ_0$-$G$-algebras.  Finally, Assumptions \ref{ass2} and \ref{ass3} are verified by checking that the multiplicative structure on the ground categories is preserved by realization.   In other words, everything done additively works with our multiplicative change of ground categories.  For example, we need that, for a simplicial $\bJ_0$-algebra $X_*$, $\ga\colon |\OM^{\infty}X_*|\rtarr  \OM^{\infty}|X_*|$ is a map of $\bJ_0$-$G$-algebras, and that is verified by inspection of linear isometries in the  passage to colimits used in the additive verification.  

\begin{thm}  For $E_{\infty}$ ring $G$-spaces $Y$, there is a natural ring completion  
$$Y \rtarr \OM^{\infty}_{\bC}\bE Y\iso \OM^{\infty}_{\bC}  \SI^{\infty}_{\bC} \bB Y,$$
and any connective $E_{\infty}$ ring $G$-spectrum $E$ is equivalent to  $\SI^{\infty}_{\bC}  \OM^{\infty}_{\bC} E$.
\end{thm}

\section{Composite adjunctions and the multiplicative theory}\label{Mult2}
\subsection{Multiplicative structure in the composite adjunction context}\label{Beck2}

Just as we applied our original general context of \autoref{ass1} multiplicatively by incorporating multiplicative structure into the given ground categories $\sT$ and $\sS$, we can apply our composite adjunction context of \autoref{ass6} by additionally incorporating multiplicative structure into the given ground category $\sV$.  Thus suppose that, in addition to \autoref{ass6}, we are given  a monad pair $(\bC,\bJ_0)$ on $\sT$.  Remembering that  $\epz\colon \bL\bR \rtarr  \mathrm{id}$ is an isomorphism, we can transfer $(\bC,\bJ_0)$ to $\sV$ via the adjunction  $(\bL, \bR)$ between $\sV$ and $\sT$. 

\begin{prop}\label{MultRCLRJL}
 Let $\bD = \bR\bC \bL$ and $\bK_0 =\bR\bJ_0\bL$. Then $\bD$ and $\bK_0$ are
  monads on $\sV$, and the action of $\bJ_0$ on $\bC$ induces an action of
  $\bK_0$ on $\bD$, so that $(\bD,\bK_0)$ is a monad pair in $\sV$.  The adjunction $(\bL, \bR)$ restricts to an  adjunction from $\bK_0[\sV]$ to $\bJ_0[\sT]$.  Moreover,  with 
  $$\om = \bR\bC\epz\colon \bD\bR = \bR\bC\bL\bR \iso \bR\bC,$$ 
  the diagrams of \autoref{ass6} are diagrams in  $\bJ_0[G\sT]$ and, with 
  $$\om_{\otimes} = \bR\bJ_0\epz\colon \bK_0 \bR = \bR\bJ_0\bL\bR \iso \bR\bJ_0,$$ 
  the analogous multiplicative diagrams (displayed in \autoref{ass9} below) also commute.
\end{prop}
\begin{proof}
  The first claim, that $\bD$ and $\bK_0$ are monads, follows from \autoref{formalMT1}.
  To prove the second claim, we use the descriptions in
  \autoref{BeckthmD}(ii). We have $\mu_{\bC,\bJ_0}: \bC\bJ_{0}\bC\bJ_{0} \to
  \bC\bJ_{0}$ for the action of  $\bJ_0$ on $\bC$. It is easy to verify formally that 
\begin{equation*}
  \begin{tikzcd}
  \mu_{\bD,\bK_0}: \bD\bK_{0}\bD\bK_{0} \cong \bR\bC\bJ_{0}\bC\bJ_{0}\bL
  \ar[r,"\bR \mu_{\bC,\bJ_0}\bL"] & \bR\bC\bJ_{0}\bL \cong \bD\bK_{0}
  \end{tikzcd}
\end{equation*}
satisfies \autoref{BeckthmD}(ii), establishing an action of  $\bK_0$ on $\bD$.  For the restricted adjunction, a given action of $\bJ_0$ on $Y$ induces an action of $\bK_0$ on $\bR \bY$ and a given action of $\bK_0$ on $X$ induces an action of $\bJ_0$ on $\bL  X$ since $\bK_0 = \bR\bJ_0\bL$ and $\bL \bR = \id$.  Easy diagram chases give the rest.
\end{proof}

This gives an example of pairs $(\bD,\bK_0)$ for which the following multiplicative elaboration of \autoref{ass6} holds.

\begin{assmcom}\label{ass9}  We assume \autoref{ass6} and assume further that
\begin{enumerate}[(i)] 
\item we have monad pairs $(\bC,\bJ_0)$ in $\sT$ and $(\bD,\bK_0)$ in $\sV$ such that $(\bL,\bR)$ restricts to an adjunction from $\bK_0[\sV]$ to $\bJ_0[\sT]$.
\item the isomorphism $\om\colon \bD\bR \rtarr\bR\bC$ restricts to an isomorphism of functors \linebreak
$\bJ_0[\sT] \rtarr \bK_0[\sV]$ and the  diagrams in \autoref{ass8} are diagrams of functors $\bJ_0[\sT] \rtarr \bK_0[\sV]$. 
\item we have an isomorphism $\om_{\otimes}\colon \bK_0\bR \rtarr \bR\bJ_0$ such that the following  analogous diagrams commute.
\[ \xymatrix{
& \bR  \ar[dl]_{\et \bR}  \ar[dr]^{\bR \et} & \\
\bK_0 \bR \ar[rr]_-{\om_{\otimes}} &  & \bR\bJ_0 \\}
\ \ \ \ \ \ \
\xymatrix{
\bK_0\bK_0\bR \ar[r]^-{\bK_0\bR\bJ_0\om_{\otimes}} \ar[d]_{\mu\bR} & \bK_0\bR\bJ_0 \ar[r]^{\om_{\otimes} \bJ_0} & \bR\bJ_0\bJ_0 \ar[d]^{\bR\mu}\\
\bK_0\bR \ar[rr]_-{\om_{\otimes}} & & \bR\bJ_0\\}
\]
\end{enumerate}
\end{assmcom} 

The example $(\bR\bC\bL,\bR \bJ_0, \bL)$ in $\sV$ derived from a given $(\bC,\bJ_0)$  in $\sT$ is always present, and other examples and applications are known in the categories of operators context.  Examples that give applications in the orbital presheaves context are work in progress.  Given the data in \autoref{ass9}, the diagram \autoref{adj9} embeds into the following diagram, which is a multiplicative version of the general composite adjunction diagram of \autoref{adj6}.  As there, here and below we use dotted arrows when the indicated left adjoint does not necessarily exist, but the context makes sense nevertheless.  

\begin{equation}\label{adj10} 
\xymatrix{
\bK_0[\sV] \ar@<.5ex>[rr]^{\bL}  \ar@<.5ex>[ddrr]^{\bF_{\bD}} 
& & \bJ_0[\sT]  \ar@<.5ex>[ll]^{\bR}    \ar@<.5ex>[rr]^{\SI}  \ar@<.5ex>[dr]^{\bF_{\bC}}& & \bJ_0[\sS]   \ar@<.5ex>[ll]^{\OM}   \ar@<.5ex>[dl]^{\OM_{\bC}}   
\ar@/^4pc/@<.6ex> [ddll]^-{(\bR\OM)_{\bD}}     \\
& & &   \ar@<.5ex>[ul]^{\bU_{\bC}}  \bC[\bJ_0[\sT]]   \ar@<.5ex>[dl]^{\bR}  \ar@<.5ex>[ur]^{\SI_{\bC}}& \\
& & \ar@<.5ex>[uull]^{\bU_{\bD}}  \bD[\bK_0[\sV]] \ar@{-->}@<.5ex>[ur]^{\bL} \ar@/_4pc/@<.6ex> [uurr]^-{(\SI \bL)_{\bD}}& &  \\}
\end{equation}

{As in \autoref{formalMT3}, our assumptions imply that the arrows $\bR$ are obtained by specialization of their underlying non-multiplicative arrows $\bR$.   We have assumed that we have the solid arrow adjunction $(\bL,\bR)$ at the top.   As in the original additive composite adjunction context, the dotted arrow $\bL$ may or may not exist, but the functor $(\SI\bL)_{\bD}$ always exists nonetheless, by application of \autoref{keyadj}. 

From here, just as in \autoref{RingThy} in the context of \autoref{ass1},  the answer to the multiplicative analog of Question \ref{Quest} is almost the same as is discussed in the underlying additive theory in \autoref{ABCDEF}.   Our Assumptions are essentially all satisfied when they are satisfied by the underlying additive structures.  Weak equivalences  are defined by forgetting to the underlying categories ($\sT$, $\sS$, $\sV$).  

Group completion \autoref{ass4} (alias ring completion here) and the approximation theorem (\autoref{ass5}) are seen in the underlying additive theory.   Assumptions \ref{ass7} and \ref{ass8} require that realization in the underlying categories carries over to give realization in the new ground categories ($\bJ_0[\sT]$, $\bJ_0[\sS]$, $\bK_0[\sV]$), and then most of their required properties fall out from their being satisfied in the underlying additive context.  When these Assumptions hold, Theorems \ref{recprincom} and \ref{adjequivcom} hold with $\sV$ replaced by  $\bK_0[\sV]$ and $\sS$ replaced by $\bJ_0[\sS]$.
}

\begin{thm}\label{recprincom2M}  There is a functor $\mathrm{Bar} \colon \bD\big[\bK_0[\sV]\big] \rtarr \bD\big[\bK_0[\sV]\big]$, written  \linebreak
$Y\mapsto \overline{Y}$, and a natural equivalence  $\ze\colon \overline{Y} \rtarr Y$ given by a map of $(\bD,\bK_0)$-algebras.  If $Y$ is special, the unit $\et_{\bD}\colon \overline{Y} \rtarr (\bR\OM)_{\bD} (\SI\bL)_{\bD} \overline{Y}$  is a ring completion and is therefore an equivalence if $Y$ is ringlike.\end{thm}

\begin{thm}\label{adjequivcom2M}  $\big((\SI\bL)_{\bD},(\bR\OM)_{\bD}\big)$ induces an adjoint equivalence from the homotopy category of special ringlike $ (\bD,\bK_0)$-algebras in $\sV$ to the homotopy category of $\OM$-connective objects of $\bJ_0[\sS]$.
\end{thm}

\subsection{Monad pairs in the categories of operators context}\label{OpsMult}

We defined categories of operators and their associated monads in \ref{cast2} and \ref{cast4}, working equivariantly as in \cite{MMO}.   The nonequivariant multiplicative elaboration to pairs of categories of operators and their associated monad pairs is developed in 
\cite{Rant2}, but the equivariant elaboration is not in the literature.   We briefly develop it here, largely following the nonequivariant treatment in \cite{Rant2}, especially part of its Section 10.  That part is  essentially written in our present composite of adjunctions context, but it was written before the context was understood.  It appeared there as a component of a more elaborate picture aimed at a machine constructing $E_{\infty}$ ring spectra from bipermutative categories.  Since we intend to treat that differently and equivariantly 
in \cite{ Mop}, we shall be brief here.

We saw earlier that an operad pair $(\sC,\sJ)$ of $G$-spaces gives rise to a monad pair $(\bC,\bJ)$.  The pair $(\sC,\sJ)$ also gives rise to a pair of categories of operators over $\sF$ and a $G$-operad pair $(\sC_G^{fin},\sJ_G)$ gives rise to a pair of categories of operators over $\sF_G$.  These pairs of categories of operators give rise to canonical pairs of monads $(\bD,\bK_0)$ and $(\bD_G,{\bK_G}_0)$ as constructed in \autoref{cast4}.\footnote{$\bK$ and $\sK$ in this subsection are general, not referring to Steiner operads.}   However, as pointed out nonequivariantly in \cite[p. 315]{Rant2}, it is not true that these pairs of monads are monad pairs in Beck's sense prescribed in \autoref{Mult1}:  $\bK_0$ does not act on $\bD$ and ${\bK_G}_0$ does not act on ${\bD_G}_0$.  Moreover, as is easily seen from the definitions, again already nonequivariantly, our left adjoint $\bL$ does {\em{not}} restrict to a functor $\bK_0[\PI[\sT]]\rtarr \bJ_0[\sT]$. 

A way around this was explained in some detail in \cite{Rant2}, and we shall just summarize the main points.
In \cite{Rant2}, all basepoint information was ignored, correcting \cite{MayMult} in which both additive and multiplicative basepoint identifications were made.  The halfway house of using the additive basepoint $0$ and not the multiplicative basepoint $1$ seems preferable.   With this change and corresponding small modifications,\footnote{including considerable change of notation, due to our changed treatment of basepoints.}  we adapt \cite[Sections 9 and 10]{Rant2} to construct a multiplicative monad $\tilde{\bK}_0$, different from the canonical monad $\bK_0$,  that does act on the canonical monad $\bD$.  

The construction of $\tilde{\bK}_0$ is formal.  There is a wreath product construction between categories of operators \cite[Definition 5.1]{Rant2}, and there is a wreath product $\sK_0\int \PI$ with associated monad $\overline{\bK}_0$ in $(\PI\int \PI)$-$G$-spaces \cite[Definition 9.1]{Rant2}.   Its algebras are the $(\sK_0\int \PI)$-algebras.   There is an adjoint pair $(\bL',\bR')$ from $(\PI\int \PI)$-$G$-spaces to 
$\PI$-$G$-spaces.  The monad $\tilde{\bK}_0$ is the conjugate $\bL'\overline{\bK}_0\bR'$ \cite[Definition 9.4]{Rant2}.   It admits a concrete description that starts from an isomorphism $\bL\tilde{\bK}_0 \iso \bJ_0\bL$ \cite[Lemma 10.4]{Rant2} and builds from there, but that is not detailed in \cite{Rant2}.  Since we do not find it illuminating, we shall also not detail it here.   There is an analogous construction starting with monads in $\PI_G[\ul{G\sT}]$ and $(\PI_G\int\PI_G)[\ul{G\sT}]$.

The upshot is that we obtain the following multiplicative version of the diagram \autoref{adj7}. Both the upper and lower parts are examples of the multiplicative composite adjunction context of \autoref{adj10}.  We have just changed ground categories to include the multiplicative structure. 

\begin{equation}\label{adj11} 
\xymatrix{
\tilde{\bK}_0\big[\PI[\ul{G\sT}]\big] \ar@<.5ex>[rr]^{\bL}  \ar@<.5ex>[ddrr]^{\bF_{\bD}}
& & \bJ_0[G\sT]  \ar@<.5ex>[ll]^{\bR}    \ar@<.5ex>[rr]^{\SI^{\infty}}  \ar@<.5ex>[dr]^{\bF_{\bC}}& & \bJ_0[G\sS]   \ar@<.5ex>[ll]^{\OM^{\infty}}   \ar@<.5ex>[dl]^{\OM^{\infty}_{\bC}}   
\ar@/^4pc/@<.6ex> [ddll]^-{(\bR\OM^{\infty})_{\bD}}         \\
& & &   \ar@<.5ex>[ul]^{\bU_{\bC}}  \bC\big[\bJ_0[G\sT]\big]   \ar@<.5ex>[dl]^{\bR}  \ar@<.5ex>[ur]^{\SI^{\infty}_{\bC}}& \\
& & \ar@<.5ex>[uull]^{\bU_{\bD}}  \bD\big[\tilde{\bK}_0[\PI[\ul{G\sT}]]\big] \ar@{-->}@<.5ex>[ur]^{\bL} 
\ar@/_4pc/@<.6ex> [uurr]^-{(\SI^{\infty} \bL)_{\bD}}  & & \\
&  & \ar@<.5ex> [d]^{\bP}  && \\
&  & \ar@<.5ex> [u]^{\bU} & & \\
\tilde{\bK}_{G_0}\big[\PI_G[\ul{G\sT}]\big] \ar@<.5ex>[rr]^{\bL_G}  \ar@<.5ex>[ddrr]^{\bF_{\bD_G}}  
& & {\bJ_G}_0[G\sT]  \ar@<.5ex>[ll]^{\bR_G}    \ar@<.5ex>[rr]^{\SI^{\infty}}  \ar@<.5ex>[dr]^{\bF_{\bC_G^{fin}}}& & {\bJ_G}_0[G\sS]   \ar@<.5ex>[ll]^{\OM^{\infty}}   \ar@<.5ex>[dl]^{\OM^{\infty}_{\bC_G^{fin}}}   
\ar@/^4pc/@<.6ex> [ddll]^-{(\bR_G\OM^{\infty})_{\bD_G}}      \\
& & &   \ar@<.5ex>[ul]^{\bU_{\bC_G^{fin}}}  \bC_G^{fin}\big[{\bJ_G}_0[G\sT]\big]   \ar@<.5ex>[dl]^{\bR_G}  \ar@<.5ex>[ur]^{\SI^{\infty}_{\bC_G^{fin}}}& \\
& & \ar@<.5ex>[uull]^{\bU_{\bD_G}}  \bD_G\big[\tilde{\bK}_{G_0}[\PI_G[\ul{G\sT}]]\big] \ar@{-->}@<.5ex>[ur]^{\bL_G} \ar@/_4pc/@<.6ex> [uurr]^-{(\SI^{\infty} \bL_G)_{\bD_G}}& & \\}
\end{equation}

\vspace{2mm}

The essential point is that, with $\sV = \Pi[\ul{G\sT}]$ or $\sV = \PI_G[\ul{G\sT}]$, both diagrams in \autoref{adj11} are special cases of the diagram \autoref{adj10}.  The assumptions are all verified with no more trouble than in the nonequivariant case.   We state the specializations of Theorems \ref{recprincom} and \ref{adjequivcom} to categories of operators over $\sF$, and we have a compatible analog for categories of operators over $\sF_G$.

\begin{thm}\label{recprincom3Op}  There is a functor $\mathrm{Bar} \colon \bD\big[\tilde{\bK}_0[\PI[G\sT]]\big] \rtarr \bD\big[\tilde{\bK}_0[\PI[G\sT]]\big]$, written 
$Y\mapsto \overline{Y}$, and a natural equivalence  $\ze\colon \overline{Y} \rtarr Y$ given by a map of $(\bD,\tilde{\bK}_0)$-algebras.  If $Y$ is special, the unit $\et_{\bD}\colon \overline{Y} \rtarr (\bR\OM^{\infty})_{\bD} (\SI^{\infty}\bL)_{\bD} \overline{Y}$  is a ring completion and is therefore an equivalence if $Y$ is ringlike.\end{thm}

\begin{thm}\label{adjequivcom3Op}  $\big((\SI^{\infty}\bL)_{\bD},(\bR\OM^{\infty})_{\bD}\big)$ induces an adjoint equivalence from the homotopy category of special ringlike 
$(\bD,\tilde{\bK}_0)$-algebras to the homotopy category of $\OM^{\infty}$-connective $G$-spectra in $\bJ_0[G\sS]$.
\end{thm}

A key use of categories of operators is to allow comparison with Segalic infinite loop space theory and to apply infinite loop space theory to go from symmetric bimonoidal categories to $E_{\infty}$ ring spectra.   For the former, multiplicative elaboration of \cite{MMO} seems entirely doable, but modulo the limitations described in \autoref{Segal}.  We shall not follow that up here.  For the latter, a good definition of symmetric bimonoidal $G$-categories has long eluded us.  That will be given in \cite{May26},  together with a shortcut that circumvents the use of categories of operators in the passage from such categories to $E_{\infty}$ ring $G$-spectra.

\subsection{Monad pairs in the orbital presheaves context}\label{PreMult}

Here we take $\sV= G\sO^{op}[\sT]$ and assume that we have monad pairs there that satisfy the axioms, as in \autoref{Beck2}.\footnote{The construction and application of  such monad pairs is work in progress.} Starting from operads $\sC$, the key additive monads are constructed in \cite{KMZ2}, where they are called $\bC^{pre}$ as in \autoref{orbitapps}.    The model categorical arguments of \autoref{Orbcells}  go through to give model structures on  the categories $\bJ_0[G\sT]$ and $\bD\big[\bK_0[G\sO^{op}[\sT]]\big]$  such that the following enhancement of \autoref{Gcellwonder2} holds.

\begin{prop}\label{Gcellwonder3}  The Quillen equivalence $(\bL,\bR)$ between the categories $\bC[G\sT]$ and  $\bD\big[\OGop_{ss}\big]$ restricts to an isomorphism of categories between the full subcategories of cofibrant  $(\bC,\bJ_0)$-algebras and cofibrant $(\bD,\bK_0)$-algebras.  Therefore  $(\bL,\bR)$  restricts to an equivalence between  $\bC\big[\bJ_0[G\sT]\big]$ and $\bD\big[\bK_0[\OGop_{ss}] \big]$.
\end{prop}

This leads us to a multiplicatively enriched version of the diagram \autoref{adj8}.  In contrast with the categories of operators context, we are only using canonical monads as constructed in \autoref{cast4}.  The top part does not fit into \autoref{ass6}, but it is where interesting examples would live.  After application of 
$\GA$, we land in a context where \autoref{ass6} does hold, by \autoref{MultRCLRJL} and \autoref{formalMT3} applied to both the additive and multiplicative monads.  

\begin{equation}\label{adj12} 
\xymatrix{
\bK_0\big[G\sO^{op}[\sT]\big] \ar@{-->}@<.5ex>[rr]^{\bL}  \ar@<.5ex>[ddrr]^{\bF_{\bD}} 
& & \bJ_0[G\sT]  \ar@<.5ex>[ll]^{\bR}    \ar@<.5ex>[rr]^{\SI^{\infty}}  \ar@<.5ex>[dr]^{\bF_{\bC}}& & \bJ_0[G\sS]   \ar@<.5ex>[ll]^{\OM^{\infty}}   \ar@<.5ex>[dl]^{\OM^{\infty}_{\bC}}   
\ar@/^4pc/@<.6ex> [ddll]^-{(\bR\OM^{\infty})_{\bD}}      \\
& & &   \ar@<.5ex>[ul]^{\bU_{\bC}}  \bC\big[\bJ_0[G\sT]\big]   \ar@<.5ex>[dl]^{\bR}  \ar@<.5ex>[ur]^{\SI^{\infty}_{\bC}}& \\
& & \ar@<.5ex>[uull]^{\bU_{\bD}}  \bD\big[\bK_0[G\sO^{op}[\sT]]\big] \ar@{-->}@<.5ex>[ur]^{\bL} \ar@/_4pc/@<.6ex> [uurr]^-{(\SI^{\infty} \bL)_{\bD}}  & & \\
&  & \ar@<.5ex> [d]^{\GA}  && \\
& & & & \\
{\bK}_0\big[G\sO^{op}[\sT]_{ss}\big] \ar@<.5ex>[rr]^{\bL}  \ar@<.5ex>[ddrr]^{\bF_{\bD}} 
& & \bJ_0[G\sT]  \ar@<.5ex>[ll]^{\bR}    \ar@<.5ex>[rr]^{\SI^{\infty}}  \ar@<.5ex>[dr]^{\bF_{\bC}}& & \bJ_0[G\sS]   \ar@<.5ex>[ll]^{\OM^{\infty}}   \ar@<.5ex>[dl]^{\OM^{\infty}_{\bC}}   
\ar@/^4pc/@<.6ex> [ddll]^-{(\bR\OM^{\infty})_{\bD}}         \\
& & &   \ar@<.5ex>[ul]^{\bU_{\bC}}  \bC\big[\bJ_0[G\sT]\big]   \ar@<.5ex>[dl]^{\bR}  \ar@<.5ex>[ur]^{\SI^{\infty}_{\bC}}& \\
& & \ar@<.5ex>[uull]^{\bU_{\bD}}  \bD\big[{\bK}_0[G\sO^{op}[\sT]_{ss}]\big] \ar@<.5ex>[ur]^{\bL} \ar@/_4pc/@<.6ex> [uurr]^-{(\SI^{\infty} \bL)_{\bD}}& & \\}
\end{equation}

\vspace{2mm}

In this context, we have the following versions of Theorems \ref{recprincom} and \ref{adjequivcom}.

\begin{thm}\label{recprincom4}  There is a functor $\mathrm{Bar} \colon \bD\big[\bK_0\big[G\sO^{op}[\sT]\big] \rtarr \bD\big[\bK_0\big[G\sO^{op}[\sT]\big]$, written  $Y\mapsto \overline{Y}$, and a natural equivalence  $\ze\colon \overline{Y} \rtarr Y$ given by a map of $(\bD,\bK_0)$-algebras.  If $Y$ and therefore $\overline{Y}$ are strictly special, as can be arranged by applying $\GA$, the unit $\et_{\bD}\colon \overline{Y} \rtarr (\bR\OM^{\infty})_{\bD} (\SI^{\infty}\bL)_{\bD} \overline{Y}$  is a ring completion and is therefore an equivalence if $Y$ is ringlike.\end{thm}

\begin{thm}\label{adjequivcom4Ga}  After application of $\GA$, $\big((\SI^{\infty}\bL)_{\bD},(\bR\OM^{\infty})_{\bD} \big)$ induces an adjoint equivalence from the homotopy category of strictly special ringlike $(\bD,\bK_0)$-algebras to the homotopy category of $\OM^{\infty}$-connective $G$-spectra in $\bJ_0[G\sS]$.
\end{thm}

\subsection{Composites of composite adjunctions}\label{Finish}

We close by simply observing that our context is broad enough to likely have many other applications.  For an evident example, it is reasonable to compose several adjunctions in the composite adjunction context.  For example, the fixed point adjunction can be composed with the categories of operators adjunction, giving a categorical version of orbital presheaves.    We do not have present applications of that in mind, but it does suggest how a monadically generalized Segal machine might be constructed on categorical orbital presheaves.  Motivic homotopy theory gives a direction in which work is currently taking place \cite{Ajay}.

\bibliographystyle{alpha}
\bibliography{references}

\end{document}